\documentclass[review]{elsarticle}
\usepackage{amsmath}
\usepackage[letterpaper,margin=1in]{geometry}
\usepackage{graphicx}
\usepackage{amsfonts}
\usepackage{bm}
\usepackage{multirow}
\usepackage{lineno}
\usepackage{multicol}
\usepackage{subcaption}
\usepackage{floatrow}
\newfloatcommand{capbtabbox}{table}[][\FBwidth]
\usepackage{caption}
\usepackage[tableposition=top]{caption}
\usepackage[utf8]{inputenc} 
\usepackage{lineno,hyperref}
\usepackage{xcolor}
\modulolinenumbers[5]

\newcommand{\ipbtobf}[2]{\left\langle#1,#2\right\rangle_{\mathcal{F}_{h}^{\partial}}}

\newcommand{\iipbfb}[2]{\left\langle#1,#2\right\rangle_{\mathcal{F}_h}}

\newcommand{\gbold}{\bm{g}}

\newcommand{\nbold}{\bm{n}}

\newcommand{\ubold}{\bm{u}}

\newcommand{\vbold}{\bm{v}}
\newcommand{\wbold}{\bm{w}}

\newcommand{\xbold}{\bm{x}}

\newcommand{\ipbt}[2]{\left\langle#1,#2\right\rangle_{\partial \mathcal{T}_h}}

\newcommand{\ipbtkf}[2]{\left\langle#1,#2\right\rangle_{F}}
\newcommand{\ipbf}[2]{\left\langle#1,#2\right\rangle_{\mathcal{F}_h}}
\newcommand{\iipbf}[2]{\left\langle#1,#2\right\rangle_{\mathcal{F}_h^i}}

\newcommand{\llbracket}{\left[\!\left[}
\newcommand{\rrbracket}{\right] \! \right]}

\newcommand{\llcurve}{\{\!\!\{}
\newcommand{\rrcurve}{\}\!\!\}}

\begin{document}
\newtheorem{remark}{Remark}
\newtheorem{corollary}{Corollary}
\newtheorem{lemma}{Lemma}
\newtheorem{theorem}{Theorem}
\newtheorem{proof}{Proof}
\def\proof{\par\noindent{\bf Proof.\ } \ignorespaces}
\def\endproof{}
\begin{frontmatter}

\title{A unified framework of continuous and discontinuous Galerkin methods for solving the incompressible Navier--Stokes equation}

\author[mainaddress]{Xi Chen\corref{mycorrespondingauthor}}
\cortext[mycorrespondingauthor]{Corresponding author: xbc5027@psu.edu (Xi Chen)}

\author[mainaddress1]{Yuwen Li}\ead{yuwenli925@gmail.com}

\author[mainaddress2]{Corina Drapaca}\ead{csd12@psu.edu}

\author[mainaddress]{John Cimbala}\ead{jmc6@psu.edu}

\address[mainaddress]{Department of Mechanical Engineering, The Pennsylvania State
University, University Park, PA 16802, USA}

\address[mainaddress1]{Department of Mathematics, The Pennsylvania State University, University Park, PA 16802, USA}

\address[mainaddress2]{Department of Engineering Science and Mechanics, The Pennsylvania State
University, University Park, PA 16802, USA}

\begin{abstract}
In this paper, we propose a unified numerical framework for the time-dependent incompressible Navier--Stokes equation which yields the $H^1$-, $H(\text{div})$-conforming, and discontinuous Galerkin methods with the use of different viscous stress tensors and penalty terms for pressure robustness. Under minimum assumption on Galerkin spaces, the semi- and fully-discrete stability is proved when a family of implicit Runge--Kutta methods are used for time discretization. Furthermore, we present a unified discussion on the penalty term. Numerical experiments are presented to compare our schemes with classical schemes in the literature in both unsteady and steady situations. It turns out that our scheme is competitive when applied to well-known benchmark problems such as Taylor--Green vortex, Kovasznay flow, potential flow, lid driven cavity flow, and the flow around a cylinder.
\end{abstract}



\end{frontmatter}

\section{Introduction}
Continuous and discontinuous Galerkin methods for the incompressible Navier--Stokes (NS) equation have been an active research area and extensively studied, see, e.g., \cite{Ern2012,GR1986,Hesthaven07,layton2008introduction,Temam1984} and references therein. Most of the classical $H^{1}$-conforming finite element methods weakly enforce the divergence free constraint and suffer from a loss of velocity accuracy due to the influence of pressure approximation and small viscosity, see, e.g., \cite{john2017divergence}. To remedy the situation, one popular approach by Franca and Hughes~\cite{franca1988two} is to add the grad-div stabilization term. Many works can be found in this direction, from both theoretical and computational point of view, see, e.g., \cite{lube2002stable,olshanskii2009grad,olshanskii2004grad,olshanskii2002low}. Recent research has shown that the grad-div stabilization is a penalization procedure  \cite{case2011connection,jenkins2014parameter,linke2011convergence}, and large grad-div stabilization parameters might lead to Poisson locking phenomena if the finite element method is not inf-sup stable in the limiting case \cite{jenkins2014parameter}. To completely decouple the pressure and velocity, one may use the $H(\text{div})$-conforming methods. With the help of a carefully designed velocity and pressure finite element  pair~\cite{boffi2013mixed,BDM1985,RT1977}, the numerical velocity is actually pointwise divergence-free and pressure-robust, see, e.g.,  \cite{falk2013stokes,guzman2014conforming,guzman2016h,john2017divergence,lehrenfeld2016high,schroeder2018towards,zhang2005new}. Finally for two-dimensional incompressible flows, one may use the vorticity-stream formulation \cite{cai2019high,GR1986,liu2000high} to automatically enforce the divergence-free constraint. 

For pressure robustness,
discontinuous Galerkin (DG) methods usually penalize the jump of the velocity normal component \cite{akbas2018analogue,gresho1990theory,guzman2016h,joshi2016post}. In \cite{Ern2012}, a new inf-sup condition involving the jump of pressure is constructed for steady incompressible NS equation and optimal convergence is observed when $\mathcal{P}_{k+1}\times\mathcal{P}_k$ DG space for velocity and pressure is used. In \cite{schroeder2017stabilised}, a element-wise grad-div penalization has been used on tensor product meshes for non-isothermal flow, and an improvement of mass conservation is observed for both inf-sup stable $\mathcal{P}_{k+1}\times\mathcal{P}_{k}$ and $\mathcal{P}_{k}\times\mathcal{P}_{k}$ pairs. Readers are also referred to \cite{cesmelioglu2017analysis,cockburn2005locally,cockburn2007note} for DG methods with more than two variables. In particular, \cite{cockburn2007note} achieves pointwise divergence-free velocity by $H(\text{div})$-conforming finite element subspace, while a postprocessed divergence-free numerical velocity is obtained in \cite{cockburn2005locally}.

In this paper, we present a unified framework for the spatial discretization of the time-dependent incompressible NS equation that covers the $H^{1}$-conforming, $H(\text{div})$-conforming, and DG methods including penalty term for pressure robustness and upwinding term for convection. 
With carefully designed numerical fluxes and consistent terms in the unified scheme, the semi-discrete stability for the first time is proved in Theorem \ref{stability_theorem} for the time-dependent incompressible NS equation under minimal assumption on Galerkin spaces. Furthermore, a unified discussion on the penalty term for pressure robustness is presented, and thus, the motivation of penalization in $H^1$-conforming, $H(\text{div})$-conforming, and DG methods is quite transparent, see Section \ref{secDG}. Another distinct feature of this paper is the use of a family of implicit Runge--Kutta methods for time discretization of the NS equation, which is shown to guarantee fully-discrete kinetic energy stability. To the best of our knowledge, such stability analysis could not be found in existing literature, see Theorem \ref{fullystability} for details.

In contrast to previous works, our numerical scheme incorporates the classical stress tensor $\bm{\tau}_h=\nu\nabla_h\ubold_h$ or $\nu(\nabla_h\bm{u}_h+\nabla_h\bm{u}_h^T)$ as well as the full viscous stress tensor $\bm{\tau}_h=\nu \left( \nabla_h\ubold_h + \nabla_h \ubold_h^T - \frac{2}{3} \left(\nabla_h\cdot \ubold_h \right) \mathbb{I} \right)$. Due to the divergence-free constraint, the variational formulation based on $\bm{\tau}=\nu\nabla\ubold$ or $\bm{\tau}=\nu(\nabla\ubold+\nabla\bm{u}^T)$ could be recovered from the corresponding formulation based on $\bm{\tau}=\nu \left( \nabla\ubold + \nabla \ubold^T - \frac{2}{3} \left(\nabla\cdot \ubold \right) \mathbb{I} \right)$ at the continuous level. However, the equivalence breaks down at the discrete level because of insufficient  regularity. Therefore, it is meaningful to check the numerical performance of those numerical methods based on the full viscous stress tensor. 
In Section \ref{secNE}, we shall test the performance of our schemes with full viscous stress tensor applied to a number of steady and unsteady benchmark problems.

Preliminary notations for numerical methods are introduced in the rest of this section. We use $\mathcal{T}_h$ to denote a conforming and shape-regular simplex mesh on a  bounded Lipschitz domain ${\Omega}$ in $\mathbb{R}^d$ where $d\in\{2,3\}$. For each element $K\in\mathcal{T}_h$, let $h_{K}$ denote the diameter of $K$. Let $\mathcal{F}_{h}$ be the collection of faces of $\mathcal{T}_h$ with $\mathcal{F}_h^{i}$ the set of interior faces and $\mathcal{F}_h^{\partial}$ the set of boundary faces. For any $(d-1)$-dimensional set $\Sigma$, we use $\langle\cdot,\cdot\rangle_\Sigma$ to denote the $L^2$ inner product on $\Sigma$, and 
\begin{align*}
&\ipbt{\cdot}{\cdot}:= \sum_{K \in \mathcal{T}_h} \langle\cdot,\cdot\rangle_{\partial K},\quad\langle\cdot,\cdot\rangle_{\partial\mathring{\mathcal{T}}_h}:=\sum_{K\in\mathcal{T}_h}\langle\cdot,\cdot\rangle_{\partial K\backslash\partial\Omega},\\
&\ipbf{\cdot}{\cdot}:= \sum_{F \in \mathcal{F}_h} \ipbtkf{\cdot}{\cdot},\quad
\langle\cdot,\cdot\rangle_{\mathcal{F}^i_h}:= \sum_{F\in\mathcal{F}^i_h} \ipbtkf{\cdot}{\cdot}.
\end{align*}
For each $F\in\mathcal{F}_h^{i}$, we fix a unit normal $\bm{n}_F$ to $F$, which points from one element $K^+$ to the other element $K^-$ on the other side. The jump and average operators are defined as:
\begin{align*}
\llbracket \phi \rrbracket|_F&= \phi|_{K^+} - \phi|_{K^-}, \qquad \llbracket \phi \nbold \rrbracket|_F = \phi|_{K^+} \bm{n}_F-\phi|_{K^-} \bm{n}_F, \qquad \llcurve \phi \rrcurve|_F= \frac{1}{2} \left( \phi|_{K^+} + \phi|_{K^-} \right), \\
\llbracket \vbold \rrbracket|_F&= \vbold|_{K^+} - \vbold|_{K^-}, \qquad \llbracket \vbold \otimes \nbold \rrbracket|_F= \vbold|_{K^+} \otimes \bm{n}_F- \vbold|_{K^-} \otimes \bm{n}_F, \qquad  \llcurve \vbold \rrcurve|_F= \frac{1}{2} \left( \vbold|_{K^+} + \vbold|_{K^-}\right),
\end{align*}
where $\phi$ and $\vbold$ are arbitrary scalar- and vector-valued functions, respectively. For a boundary face $F\in\mathcal{F}_h^{\partial}$ which is contained in a single element $K\in\mathcal{T}_h$, we further assume that $\bm{n}_F$ is the outward pointing normal to $\partial\Omega$ and define
\begin{align*}
\llbracket \phi \rrbracket|_F&= \phi|_K, \qquad \llbracket \phi \nbold \rrbracket|_F= \phi|_K\bm{n}_F, \qquad \llcurve \phi \rrcurve|_F= \phi|_K, \\
\llbracket \vbold \rrbracket|_F&= \vbold|_K, \qquad \llbracket \vbold\otimes \nbold \rrbracket|_F= \vbold|_K\otimes \bm{n}_F, \qquad  \llcurve \vbold \rrcurve|_F= \vbold|_K.
\end{align*}
Throughout the rest of this paper, we use $\bm{n}\in\prod_{F\in\mathcal{F}_h}\mathbb{R}^d$ to denote the piecewise constant vector defined on the skeleton $\mathcal{F}_h$ such that $\bm{n}|_F:=\bm{n}_F$ for all $F\in\mathcal{F}_h.$ 
Let $\mathcal{P}_j \left(K \right)$ denote the space of polynomials of degree at most $j$. We shall make use of the following   function spaces
\begin{align*}
&[H^{m}(\mathcal{T}_h)]^d = \left\{\vbold  \in [{L}^{2}(\Omega)]^d: \vbold|_{K} \in[{H}^{m}(K)]^d,~\forall K \in \mathcal{T}_h \right\}, \\
&L^{2}_{0}(\Omega) = \left\{q \in {L}^{2}(\Omega):  \int_{\Omega}q=0\right\},\\
&\bm{V}_h \subseteq\left\{\bm{v}_h  \in [{L}^{2}(\Omega)]^d: \bm{v}_h |_{K} \in  [\mathcal{P}_r \left( K \right)]^d,~\forall K \in \mathcal{T}_h \right\},\\
&Q_h\subseteq \left\{q_h\in L^{2}_{0}(\Omega): q_{h} |_{K} \in \mathcal{P}_k \left( K \right),~\forall K \in \mathcal{T}_h \right\}, 
\end{align*}
where $\bm{V}_h$ and $Q_h$ will be given in Section \ref{secDG}. Here we do not require a specific relationship between $r$ and $k$.

The rest of this paper is organized as follows. In Section \ref{gen_method_sec}, we first present the unified framework and then prove the semi- and fully-discrete stability of the general scheme. In Section \ref{secDG}, we derive $H^1$-, $H(\text{div})$-conforming, and DG methods from the unified scheme, and discuss the expression of the penalty term for each of the three methods. In Section \ref{secNE}, we test our schemes in both unsteady and steady situations, and compare the simulation results with classical schemes and data in the literature. Finally we conclude our paper in Section \ref{conclusion}.
\section{General formulation} \label{gen_method_sec}

In this section, we present a general framework covering the $H^1$-conforming, $H(\text{div})$-conforming, and DG methods for the following incompressible Navier--Stokes equation
\begin{align*}
\partial_t\ubold + \nabla \cdot \left( \ubold \otimes \ubold + p\mathbb{I} \right) - \nu\nabla \cdot \bm{\tau}(\bm{u})  = \bm{f},  \qquad & \text{in} \quad \left(0, T\right]\times \Omega,\\
\nabla \cdot \ubold = 0, \qquad & \text{in} \quad \left(0, T\right]\times \Omega,\\  
\ubold=\bm{0},\qquad& \text{on} \quad \left(0, T\right] \times \partial\Omega,\\
\ubold(0,\xbold)=\ubold_0(\xbold),\qquad& \text{in} \quad \Omega, 
\end{align*}
where $\nu>0$ is the viscosity constant, and
$\bm{\tau}(\bm{u})$ is the viscous strain tensor that could be 
\begin{equation}\label{tau}
    \bm{\tau}(\bm{u}):=\nabla\bm{u},\text{ or }\bm{\tau}(\bm{u}):=\nabla\bm{u}+ \nabla\bm{u}^T,\text{ or }\bm{\tau}(\bm{u}):=\nabla\bm{u}+ (\nabla\bm{u})^T - \frac{2}{3} \left(\nabla\cdot \bm{u}\right) \mathbb{I}.
\end{equation}
The three choices of viscous strain tensor yield the same problem in the smooth level.
Let $(\cdot,\cdot)$ denote the usual $L^2$ inner product on $\Omega$ and $\nabla_h$ the broken gradient with respect to $\mathcal{T}_h$.
Our general semi-discrete scheme seeks unknowns  $\left(\bm{u}_h(t),p_h(t) \right)\in\bm{V}_h\times Q_h$ for each time $t\in(0,T]$ such that
\begin{subequations}\label{mixed}
    \begin{align}
&(\partial_t\ubold_h,\vbold_h)- (\ubold_h \otimes \ubold_h,\nabla_h \vbold_h)- (p_h,\nabla_h \cdot \vbold_h) + \ipbt{\hat{\bm{\sigma}}_{h} \nbold}{\bm{v}_h} \\[1.5ex] 
\nonumber & -\frac{1}{2} (\left(\nabla_h \cdot \bm{u}_h \right) \ubold_h,\vbold_h) + \frac{1}{2}  \langle\bm{u}_h,\nbold \llcurve \ubold_h \cdot \vbold_h\rrcurve\rangle_{\partial\mathring{\mathcal{T}}_h}+d_h(\bm{u}_h,\bm{v}_h)\\[1.5ex]
\nonumber & + \nu \bigg[ (\bm{\tau}_h(\bm{u}_h),\nabla_h \vbold_h) - \ipbt{\widehat{\bm{\tau}}_h\, \nbold}{\vbold_h}
+ \ipbt{\widehat{\bm{u}}_h - \ubold_h}{\bm{\tau}_h(\bm{v}_h) \nbold} \bigg]=(\bm{f},\vbold_h),\\[1.5ex]
&(\nabla_h\cdot \ubold_h,q_h)-\iipbfb{\llbracket \ubold_h \rrbracket\cdot\nbold}{\llcurve q_h \rrcurve} = 0,
\end{align}
\end{subequations}
for all $\left(\bm{v}_h,q_h\right)\in\bm{V}_h\times Q_h$ subject to the initial condition 
$\bm{u}_h(0)=I_h\bm{u}_0,$ where $I_h$ is a suitable interpolation onto $\bm{V}_h,$ and $\bm{\tau}_h(\bm{u}_h)$
is the discrete viscous strain tensor, which could be  
\begin{equation}\label{tauh}
    \bm{\tau}_h(\bm{u}_h):=\nabla_h\bm{u}_h,\text{ or } \bm{\tau}_h(\bm{u}_h):=\nabla_h\bm{u}_h+ \nabla_h\bm{u}_h^T,\text{ or }\bm{\tau}_h(\bm{u}_h):=\nabla_h\bm{u}_h+ (\nabla_h\bm{u}_h)^T - \frac{2}{3} \left(\nabla_h\cdot \bm{u}_h\right) \mathbb{I}.
\end{equation}
Note that $\frac{1}{2}  \langle\bm{u}_h,\nbold \llcurve \ubold_h \cdot \vbold_h\rrcurve\rangle_{\partial\mathring{\mathcal{T}}_h}$ is a consistent term added for convenience of analysis. In order to improve pressure robustness, we use the penalty term $d_h(\bm{u}_h,\bm{v}_h)$ which is consistent and positive semi-definite, namely, for all $v_h\in\bm{V}_h$,
\begin{align}
&d_h(\bm{v}_h,\bm{v}_h)\geq 0,\label{positivedh}\\
&d_h(\bm{u},\bm{v}_h)=0.
\end{align}
In principle, $d_h$ could also depend on the pressure although we have not found such examples in practice.
The particular expression of $d_h$ will be specified later, see Section \ref{secDG} for details. Let $h_F$ denote the diameter of $F\in\mathcal{F}_h$ and $h=\{h_F\}_{F\in\mathcal{F}_h}$  the face size function.
We recommend the following numerical fluxes 
\begin{align*}
\widehat{\bm{\sigma}}_h &= \llcurve \ubold_h \rrcurve \otimes \llcurve \ubold_h \rrcurve + \llcurve p_h \rrcurve \mathbb{I}+\zeta\left| \llcurve\ubold_h\rrcurve \cdot \nbold\right| \llbracket \ubold_h \otimes \nbold \rrbracket, \\[1.5ex]
\widehat{\bm{\tau}}_{h} & = \llcurve\bm{\tau}_h(\bm{u}_h)\rrcurve -\eta h^{-1} \llbracket \bm{u}_h \otimes \nbold \rrbracket,\\
\widehat{\bm{u}}_h &= \llcurve \ubold_h \rrcurve\text{ on }\mathcal{F}_h^i,\quad\widehat{\bm{u}}_h =\bm{0}~\text{ on }\mathcal{F}_h^\partial.
\end{align*}
where $\zeta=\{\zeta_F\}_{F\in\mathcal{F}_h}$ and $\eta=\{\eta_F\}_{F\in\mathcal{F}_h}$ are user specified piecewise non-negative constants for controlling the amount of numerical dissipation. 

Now we introduce the viscous bilinear form $a_h$, the convective bilinear form $b_h$, and the convective form $c_h$ in the following.
\begin{align}
\nonumber a_h(\vbold_h, \wbold_h)&:=( \bm{\tau}_h(\bm{v}_h),\nabla_h \wbold_h)-  \ipbt{\left( \llcurve \bm{\tau}_h(\bm{v}_h) \rrcurve -\eta h^{-1}\llbracket \vbold_h \otimes \nbold \rrbracket \right) \nbold}{\wbold_h}\\
\nonumber&\quad+ \langle\llcurve \vbold_h \rrcurve - \vbold_h,\bm{\tau}_h(\bm{w}_h)  \nbold\rangle_{\partial\mathring{\mathcal{T}}_h}-\langle\vbold_h,\bm{\tau}_h(\bm{w}_h)  \nbold\rangle_{\partial\Omega},\\
b_h(\bm{v}_h, q_h)&:= (\nabla_h \cdot \vbold_h,q_h)-\iipbfb{\llbracket \vbold_h \rrbracket\cdot\nbold}{\llcurve q_h \rrcurve},\label{bh_form}\\
\nonumber c_h \left(\bm{\beta}_h ; \vbold_h, \wbold_h \right)  &:= - (\vbold_h \otimes \bm{\beta}_h,\nabla_h \wbold_h)-\frac{1}{2} (\left(\nabla_h \cdot \bm{\beta}_h \right) \vbold_h,\wbold_h) \\
\nonumber & + \ipbt{\left(\zeta \left|  \llcurve\bm{\beta}_h\rrcurve \cdot \nbold \right| \llbracket \vbold_h \otimes \nbold \rrbracket \right) \nbold}{\wbold_h}+ \ipbt{(\llcurve\vbold_h \rrcurve \otimes \llcurve \bm{\beta}_h \rrcurve)\bm{n}}{\bm{w}_h}+  \frac{1}{2} \langle\bm{\beta}_h,\nbold \llcurve \vbold_h \cdot \wbold_h \rrcurve\rangle_{\partial\mathring{\mathcal{T}}_h}.
\end{align} 
One can then rewrite \eqref{mixed} in the following compact form
\begin{subequations}\label{compact}
\begin{align}
&(\partial_t\ubold_h,\vbold_h) +N_h(\bm{u}_h;\bm{v}_h)- b_h \left( \vbold_h, p_h \right)=(\bm{f},\vbold_h),\quad\forall\bm{v}_h\in\bm{V}_h,\label{compact1}  \\[1.5ex]
& b_h \left(\ubold_h, q_h \right) = 0,\quad \forall q_h\in Q_h,\label{compact2}
\end{align}
\end{subequations}
where 
$$N_h(\bm{u}_h;\bm{v}_h):=c_h \left(\ubold_h; \ubold_h, \vbold_h \right) +\nu a_h \left(\ubold_h, \vbold_h \right)+d_h(\bm{u}_h,\bm{v}_h).$$ Note that $N_h(\bm{u}_h;\bm{v}_h)$ is nonlinear in $\bm{u}_h$ but linear in $\bm{v}_h.$
\begin{remark}
From the derivation given above, it can be observed that the scheme \eqref{compact} is consistent if $\bm{u}(t)\in[H^{1}_{0}(\Omega)]^d\cap [H^{\frac{3}{2}+\varepsilon}(\mathcal{T}_h)]^d$ and $p(t)\in L^{2}_{0}(\Omega)\cap{H}^{\frac{1}{2}+\varepsilon}(\mathcal{T}_h)~with~\varepsilon>0$.
\end{remark}{}
Throughout the rest of this paper, we use $C$ to denote any positive  absolute constant that is independent of $h.$ We shall also make use of the following mesh-dependent norms
\begin{align*}
\|\llbracket\vbold_{h}\rrbracket\|_{L^2(\mathcal{F}_h)}&:=\langle\bm{v}_h,\bm{v}_h\rangle_{\mathcal{F}_h}^\frac{1}{2},\\
    \|\bm{v}_h\|_{1,h}&:=\left(\| \bm{\tau}_h(\bm{v}_h)\|_{{L}^2 \left(\Omega \right)}^2 + \eta h^{-1}\|\llbracket\vbold_{h}\rrbracket\|_{L^2(\mathcal{F}_h)}^2\right)^\frac{1}{2}.
\end{align*}
It is noted that $\|\cdot\|_{1,h}$ is a well-defined norm on $\bm{V}_h$, see \cite{Brenner2004,BS2008,chen2020versatile} for details.
The next theorem shows that $a_h$ is coercive with respect to the norm $\|\cdot\|_{1,h}.$
\begin{lemma}[Positivity of $a_h$]\label{positiveah}
For all $F\in\mathcal{F}_h,$ assume that $\eta_F\geq{\eta}_0$, where $\eta_0$ is a sufficiently large constant independent of $h.$
Then for $\vbold_h\in\bm{V}_h$ it holds that
\begin{align*}
     a_h \left( \vbold_h, \vbold_h \right)\geq C\|\bm{v}_h\|_{1,h}^2.
\end{align*}
\begin{proof}
Combining terms on $\mathcal{F}_h$, one can rewrite $a_h$ in the following symmetric form
\begin{equation}\label{ahsym}
a_h(\vbold_h, \wbold_h)=( \bm{\tau}_h(\bm{v}_h),\nabla_h \wbold_h)   -\ipbf{\llbracket \vbold_h \rrbracket}{\llcurve  \bm{\tau}_h(\bm{w}_h) \rrcurve \bm{n}}  
-\ipbf{\llbracket \wbold_h \rrbracket}{\llcurve  \bm{\tau}_h(\bm{v}_h) \rrcurve \bm{n}} 
+\ipbf{\eta h^{-1} \llbracket \vbold_h \rrbracket}{\llbracket \wbold_h \rrbracket}.
\end{equation}
First we assume $\bm{\tau}_h(\bm{v}_h)=\nabla_h\bm{v}_h+\nabla_h\bm{v}_h^T-\frac{2}{3}(\nabla_h\cdot\bm{v}_h)\mathbb{I}$. It follows from \eqref{ahsym} and the algebraic identity
\begin{align*}
    (\bm{\tau}_h(\bm{v}_h),\nabla_h\bm{w}_h)=\frac{1}{2}(\bm{\tau}_h(\bm{v}_h),\bm{\tau}_h(\bm{w}_h))+\left(\frac{2}{3}-\frac{2d}{9}\right)(\nabla_h\cdot\bm{v}_h,\nabla_h\cdot\bm{w}_h)
\end{align*}
that
\begin{equation*}
a_h(\vbold_h, \vbold_h)\geq\frac{1}{2}\|\bm{\tau}_h(\bm{v}_h)\|_{L^2(\Omega)}^2-2\ipbf{\llbracket \vbold_h \rrbracket}{\llcurve  \bm{\tau}_h(\bm{v}_h) \rrcurve \bm{n}} +\ipbf{\eta h^{-1} \llbracket \vbold_h \rrbracket}{\llbracket \vbold_h \rrbracket}.
\end{equation*}
Assuming $\eta$ is sufficiently large, we  conclude the proof from the trace and Cauchy--Schwarz inequalities, which is standard in the analysis of interior penalty DG methods, see, e.g., \cite{Arnold1982,BS2008} for details. The other two cases $\bm{\tau}_h(\bm{u}_h)=\nabla_h\bm{u}_h$ and $\bm{\tau}_h(\bm{u}_h)=\nabla_h\bm{u}_h+\nabla_h\bm{u}_h^T$ can be proved in a similar way.
\end{proof}
\label{coercive_vis_lemma}
\end{lemma}
\begin{lemma}[Positivity of $c_h$]\label{positivech}
Assume  $\zeta_F\geq 0.5$ for all $F\in\mathcal{F}_{h}^{\partial}$. Then for  $\bm{\beta}_h, \vbold_h\in \bm{V}_h,$ we have
\begin{align*}
    c_h \left(\bm{\beta}_h; \vbold_h, \vbold_h \right)\geq 0.
\end{align*}
\end{lemma}
\begin{proof}
Using integration by parts,
$c_h$ could be rewritten as
\begin{align}\label{convective_form}
c_h \left(\bm{\beta}_h; \vbold_h, \wbold_h \right) &= (\bm{\beta}_h \cdot \nabla_h \vbold_h,\wbold_h)+ \frac{1}{2} (\left(\nabla_h \cdot \bm{\beta}_h \right) \vbold_h,\wbold_h)- \iipbf{ \left( \llcurve\bm{\beta}_h\rrcurve \cdot \nbold \right) \llbracket \vbold_h \rrbracket}{\llcurve \wbold_h \rrcurve}\nonumber\\
&\quad-\frac{1}{2}\iipbf{\llbracket \beta_h \rrbracket\cdot\nbold}{\llcurve \vbold_h\cdot\wbold_h \rrcurve}+ \ipbf{\zeta \left| \llcurve\bm{\beta}_h\rrcurve \cdot \bm{n} \right|\llbracket \vbold_h \rrbracket}{\llbracket \wbold_h \rrbracket}.
\end{align}
It then follows from the following identities
\begin{align*}
&(\bm{\beta}_h \cdot \nabla_h \vbold_h,\vbold_h)+ \frac{1}{2} (\left(\nabla_h \cdot \bm{\beta}_h \right) \vbold_h,\vbold_h)
=\frac{1}{2}  \ipbt{\vbold_h}{\vbold_h \left(\bm{\beta}_h \cdot \nbold \right)},\\
&\frac{1}{2}  \ipbt{\vbold_h}{\vbold_h \left(\bm{\beta}_h \cdot \nbold \right)}= \frac{1}{2} \iipbf{\llbracket \bm{\beta}_h \rrbracket}{\nbold \llcurve \vbold_h \cdot \vbold_h \rrcurve} + \iipbf{ \left(\llcurve \bm{\beta}_h \rrcurve \cdot \nbold \right) \llbracket \vbold_h \rrbracket}{\llcurve \vbold_h \rrcurve}+\frac{1}{2}\ipbtobf{\bm{\beta}_h\cdot\nbold}{\vbold_h \cdot \vbold_h },
\end{align*}
and \eqref{convective_form} with $\wbold_h = \vbold_h$ that
\begin{align}
    c_h \left(\bm{\beta}_h; \vbold_h, \vbold_h \right) =  \zeta \iipbf{\left| \llcurve\bm{\beta}_h\rrcurve \cdot \nbold \right|\llbracket \vbold_h \rrbracket}{\llbracket \vbold_h \rrbracket}+ \ipbtobf{\zeta|\bm{\beta}_h\cdot\nbold|+0.5\bm{\beta}_h\cdot\nbold}{\vbold_h \cdot \vbold_h },
\end{align}
Finally, we conclude the proof by using $\zeta_{F}\geq 0.5$ for $F\in\mathcal{F}_{h}^{\partial}$.
\end{proof}
\begin{remark}
The term $\ipbtobf{\zeta|\bm{\beta}_h\cdot\nbold|}{\vbold_h \cdot \vbold_h }$ with $\zeta\geq 0.5$ plays an important role in guaranteeing the  positive semi-definiteness of the convective form, which is crucial for proving the stability of the semi-discrete time-dependent incompressible NS equation. In contrast, such stability is not considered in DG schemes for  steady-state NS equation. This new modification is one of the key differences of our formulation from the classical formulations in e.g., \cite{akbas2018analogue,Ern2012}.
\end{remark}
With the help of Lemmata~\ref{positiveah} and~\ref{positivech}, we obtain the following stability.
\begin{theorem}\label{stability_theorem}[Semi-discrete Stability Estimate] Let the assumptions in Lemmata \ref{positiveah} and \ref{positivech} hold and 
\begin{align*}
    \left\| \bm{f} \right\|_{L^1 \left(0, t;L^2 \left(\Omega\right) \right)}:= \int_0^t\| \bm{f}(s) \|_{L^2 (\Omega)} ds<\infty,\quad\forall t\in[0,T].
\end{align*}
Then for all $0\leq t\leq T$, the scheme~\eqref{compact} admits the following semi-discrete stability
\begin{align*}
      \| \ubold_h(t) \|_{L^2 (\Omega)} \leq \left\| \ubold_h \left(0 \right) \right\|_{L^2 (\Omega)} + \| \bm{f}\|_{L^1 \left(0, t;{L}^2(\Omega)\right)}.
\end{align*}
\end{theorem}
\begin{proof}
Taking $\bm{v}_h = \bm{u}_h$ in \eqref{compact1} and $q_h = p_h$ in \eqref{compact2}, we have
\begin{align}\label{semi_st}
\frac{1}{2} \frac{d}{dt} \left\| \bm{u}_h \right\|_{L^2 (\Omega)}^2 + N_h(\bm{u}_h;\bm{u}_h)=(\bm{f},\ubold_h).  
\end{align}
It then follows from the identity given above, the positivity of $a_h, c_h, d_h$ (see  \eqref{positivedh} and Lemmata \ref{positiveah} and \ref{positivech}) and the Cauchy--Schwarz inequality that
\begin{align*}
   \| \bm{u}_h\|_{L^2(\Omega)} \frac{d}{dt} \left\| \ubold_h \right\|_{L^2(\Omega)} \leq \| \bm{f}\|_{{L}^2 (\Omega)}\| \ubold_h \|_{{L}^2(\Omega)},
\end{align*}
which implies
\begin{align}
    \frac{d}{dt} \left\| \ubold_h \right\|_{L^2 \left(\Omega\right)} & \leq \left\| \bm{f} \right\|_{L^2 \left(\Omega\right)}. \label{kbound_four}
\end{align}
Integrating \eqref{kbound_four} over $[0,t]$ yields
\begin{align*}
    \| \ubold_h(t) \|_{L^2 (\Omega)} \leq \left\| \ubold_h \left(0 \right) \right\|_{L^2 (\Omega)} + \| \bm{f}\|_{L^1 \left(0, t;{L}^2(\Omega)\right)}.
\end{align*}
The proof is complete.
\end{proof}


\subsection{Fully discrete stable scheme}
Let the time interval $[0,T]$ be partitioned into $0=t_0<t_1<\cdots<t_{N-1}<t_N=T$. For each $n$, let $\tau_n:=t_{n+1}-t_{n}$. We use the Runge--Kutta (RK) method (see, e.g., \cite{HNW1993}) to discretize the semi-discrete finite-dimensional system \eqref{compact}. In particular, an $m$-stage RK method is determined by parameters $\{a_{ij}\}^m_{i,j=1}$, $\{b_i\}_{i=1}^m$, $\{c_i\}_{i=1}^m$. Applying this RK method to the time direction in \eqref{compact}, we obtain the following fully discrete RK-DG method
\begin{equation}\label{fully}
    (\bm{u}_h^{n+1},\bm{v}_h)=(\bm{u}_h^n,\bm{v}_h)+\tau_n\sum_{i=1}^mb_i\left\{(\bm{f}^i,\bm{v}_h)-N_h(\bm{U}_h^i;\bm{v}_h)+b_h(\bm{v}_h,P_h^i)\right\},\quad\forall\bm{v}_h\in\bm{V}_h,
\end{equation}
where $\bm{f}^i=\bm{f}(t_n+c_i\tau_n)$, and the internal stages $\bm{U}_h^i\in\bm{V}_h$ and $P_h^i\in Q_h $ with $1\leq i\leq m$ are determined by
\begin{subequations}\label{internal}
\begin{align}
    &(\bm{U}_h^i,\bm{v}_h)=(\bm{u}_h^n,\bm{v}_h)+\tau_n\sum_{j=1}^ma_{ij}\left\{(\bm{f}^j,\bm{v}_h)-N_h(\bm{U}_h^j;\bm{v}_h)+b_h(\bm{v}_h,P_h^j)\right\},\quad\forall\bm{v}_h\in\bm{V}_h,\label{internal1}\\
    &b_h(\bm{U}_h^i,q_h)=0,\quad\forall q_h\in Q_h.\label{internal2}
\end{align}
\end{subequations}
Note that $0\leq c_i\leq1$ for all RK methods. The internal stages $\bm{U}_h^i$ and $P_h^j$ have useful approximation property. In fact,  $\bm{U}_h^i\approx\bm{u}_h(t_n+c_i\tau_n)$ and $P_h^j\approx p_h(t_n+c_j\tau_n)$.

Although the semi-discrete stability is proved in Theorem \ref{stability_theorem}, a traditional time discretization such as the family of Backward Differentiation Formulas (BDF) methods would usually destroy such nice dynamic structure. In general, it is quite delicate to design a stability preserving time integration technique for complex dynamical systems, see, e.g., \cite{GST2001} for stability preserving RK schemes for hyperbolic conservation laws and \cite{layton2008introduction,palha2017mass} for stable low order \emph{time difference} schemes for incompressible NS equations.

In this subsection, we consider a family of Gauss--Legendre collocation Runge--Kutta (GLRK) methods \cite{Butcher1964,HLW2006} that achieve arbitrarily high order accuracy. The parameters $\{c_i\}_{i=1}^m$ are zeros of the Gauss--Legendre polynomial $\frac{d^m}{ds^m}\big(s^m(1-s)^m\big)$. Then $\{a_{ij}\}_{i,j=1}^m$ and $\{b_i\}_{i=1}^m$ are uniquely determined by $\{c_i\}_{i=1}^m$. For instance, if $m=1$, then $c_1=\frac{1}{2}$, $a_{11}=\frac{1}{2}$, $b_1=1$, which is \emph{equivalent} to the Crank--Nicolson scheme. If $m=2,$ then $c_1=\frac{1}{2}-\frac{\sqrt{3}}{6}$, $c_2=\frac{1}{2}+\frac{\sqrt{3}}{6}$, $b_1=b_2=\frac{1}{2},$ and
$$a_{11}=\frac{1}{4},\quad a_{12}=\frac{1}{4}-\frac{\sqrt{3}}{6},\quad a_{21}=\frac{1}{4}+\frac{\sqrt{3}}{6},\quad a_{22}=\frac{1}{4}.$$
It is well-known that any GLRK method satisfies (see \cite{HLW2006})
\begin{align}
    &b_{i}b_j-b_ia_{ij}-b_ja_{ji}=0\quad\forall1\leq i,j\leq m,\label{GLRKcond}\\
    &\sum_{i=1}^mb_i=1,\quad b_i>0,\quad\forall 1\leq i\leq m.\label{condb}
\end{align}
The next theorem shows that the GLRK method preserves the semi-discrete stability given in Theorem \ref{stability_theorem}.
\begin{theorem}[Fully Discrete Kinetic Energy Estimate]\label{fullystability}Let the assumptions in Theorem \ref{stability_theorem} hold. In addition, we assume one of the three following conditions holds: $(a)~\bm{\tau}_h(\bm{u}_h)=\nabla_h\bm{u}_h$; $(b)~\bm{\tau}_h(\bm{u}_h)=\nabla_h\bm{u}_h+\nabla_h\bm{u}_h^T$; $(c)$~$\|\bm{v}_h\|_{L^{2}(\Omega)}\leq C\|\bm{v}_h\|_{1,h}$ when $\bm{\tau}_h(\bm{u}_h)=\nabla_h\bm{u}_h+\nabla_h\bm{u}_h^T-\frac{2}{3}(\nabla_h\cdot\bm{u}_h)\mathbb{I}.$
Then we have the following fully discrete kinetic energy estimate
\begin{align*}
   \|\bm{u}_h^n\|^2_{L^2(\Omega)}\leq\|\bm{u}_h^0\|^2_{L^2(\Omega)}+C\sum_{j=0}^{n-1}\tau_j\sum_{i=1}^mb_i\|\bm{f}(t_j+c_i\tau_j)\|_{L^2(\Omega)}^2,\quad\forall n\geq1.
\end{align*}
\end{theorem}
\begin{proof}
Since $N_h(\bm{u}_h;\bm{v}_h)$ is linear in $\bm{v}_h$, there exists a unique  $R_h(\bm{u}_h)\in\bm{V}_h$ such that
$$(R_h(\bm{u}_h),\bm{v}_h)=N_h(\bm{u}_h;\bm{v}_h)\quad\text{ for all }\bm{v}_h\in\bm{V}_h.$$
Let $B_h: \bm{V}_h\rightarrow Q_h$ denote the linear operator associated with $b_h$, i.e., $$(B_h\bm{v}_h,q_h)=b_h(\bm{v}_h,q_h)\quad\text{ for all } q_h\in Q_h.$$ Therefore, \eqref{fully} and \eqref{internal} translate into 
\begin{subequations}
\begin{align}
    &\bm{u}_h^{n+1}=\bm{u}_h^n+\tau_n\sum_{i=1}^mb_i\bm{F}_h^i,\label{fullya}\\
    &\bm{U}_h^i=\bm{u}_h^n+\tau_n\sum_{j=1}^ma_{ij}\bm{F}_h^j,\label{internal1a}\\
    &(\bm{U}_h^i,B^T_hq_h)=(B_h\bm{U}_h^i,q_h)=0,\quad\forall q_h\in Q_h\label{Bh},
\end{align}
\end{subequations}
where  $\bm{F}_h^i:=\bm{f}^i-R_h(\bm{U}_h^i)+B_h^TP_h^i.$ It follows from \eqref{fullya} that
\begin{equation}\label{total}
    \|\bm{u}_h^{n+1}\|_{L^2(\Omega)}^2=\|\bm{u}_h^n\|_{L^2(\Omega)}^2+\tau_n\sum_{i=1}^mb_i\left(\bm{F}_h^i,\bm{u}_h^n\right)+\tau_n\sum_{j=1}^mb_j\left(\bm{u}_h^n,\bm{F}_h^j\right)+\tau^2_n\sum_{i,j=1}^mb_ib_j\left(\bm{F}_h^i,\bm{F}_h^j\right).
\end{equation}
Using \eqref{internal1a}, the second term on the right hand side becomes
\begin{equation}\label{2nd}
\begin{aligned}
    &\tau_n\sum_{i=1}^mb_i\left(\bm{F}_h^i,\bm{u}_h^n\right)=\tau_n\sum_{i=1}^mb_i\left(\bm{F}_h^i,\bm{U}_h^i-\tau_n\sum_{j=1}^ma_{ij}\bm{F}_h^j\right)\\
    &=\tau_n\sum_{i=1}^mb_i\left(\bm{f}^i-R_h(\bm{U}_h^i),\bm{U}_h^i\right)-\tau_n^2\sum_{i,j=1}^mb_ia_{ij}(\bm{F}^i_{h},\bm{F}^j_{h}),
    \end{aligned}
\end{equation}
where \eqref{Bh} is used in the last equality.
Similarly, it holds that
\begin{equation}\label{3rd}
    \tau_n\sum_{j=1}^mb_j\left(\bm{u}_h^n,\bm{F}_h^j\right)=\tau_n\sum_{j=1}^m b_j\left(\bm{f}^j-R_h(\bm{U}_h^j),\bm{U}_h^j\right)-\tau_n^2\sum_{i,j=1}^mb_ja_{ji}(\bm{F}^i,\bm{F}^j).
\end{equation}
Collecting \eqref{total}, \eqref{2nd}, \eqref{3rd} and using \eqref{GLRKcond}, we obtain
\begin{equation}
\begin{aligned}
    \|\bm{u}_h^{n+1}\|_{L^2(\Omega)}^2&=\|\bm{u}_h^n\|_{L^2(\Omega)}^2+2\tau_n\sum_{i=1}^mb_i\left(\bm{f}^i-R_h(\bm{U}_h^i),\bm{U}_h^i\right)+\tau^2_n\sum_{i,j=1}^m(b_ib_j-b_ia_{ij}-b_ja_{ji})\left(\bm{F}_h^i,\bm{F}_h^j\right)\\
    &=\|\bm{u}_h^n\|_{L^2(\Omega)}^2+2\tau_n\sum_{i=1}^mb_i\left\{\left(\bm{f}^i,\bm{U}_h^i\right)-\left(R_h(\bm{U}_h^i),\bm{U}_h^i\right)\right\}.
    \end{aligned}
\end{equation}
For $\bm{v}_h\in \bm{V}_h$, recall the discrete Poincar\'e inequality (cf.~\cite{BS2008,Ern2012}) 
\begin{equation}\label{dispoin}
    \|\bm{v}_h\|_{L^2(\Omega)}\leq C\left(\|\nabla_h\bm{v}_h\|_{L^2(\Omega)}+\|h^{-\frac{1}{2}}\llbracket\bm{v}_h\rrbracket\|_{L^2(\mathcal{F}_h)}\right),
\end{equation} 
and the discrete Korn's inequality (see Eq.~(1.19) in \cite{Brenner2004})   \begin{equation}\label{diskorn}
    \|\nabla_h\bm{v}_h\|_{L^2(\Omega)}\leq C\left(\|\nabla_h\bm{v}_h+\nabla_h\bm{v}_h^T\|_{L^2(\Omega)}+\|h^{-\frac{1}{2}}\llbracket\bm{v}_h\rrbracket\|_{L^2(\mathcal{F}_h)}\right).
\end{equation}
Using \eqref{dispoin}, \eqref{diskorn} and the definition of $\|\cdot\|_{1,h}$, it holds that 
\begin{equation}\label{Korn}
    \|\bm{v}_h\|_{L^2(\Omega)}\leq C\|\bm{v}_h\|_{1,h},\quad\forall\bm{v}_h\in\bm{V}_h,
\end{equation}
when $\bm{\tau}_h(\bm{v}_h)=\nabla_h\bm{v}_h$ or $\bm{\tau}_h(\bm{v}_h)=\nabla_h\bm{v}_h+\nabla_h\bm{v}_h^T$.
Otherwise, the previous inequality follows from the assumption (c). Now combining Lemmata \ref{positiveah}, \ref{positivech}, Equation \eqref{positivedh} and using \eqref{Korn}, \eqref{condb}, we have
\begin{equation}
    \begin{aligned}
    &\sum_{i=1}^mb_i\left\{\left(\bm{f}^i,\bm{U}_h^i\right)-\left(R_h(\bm{U}_h^i),\bm{U}_h^i\right)\right\}\leq\sum_{i=1}^mb_i\left(\left(\bm{f}^i,\bm{U}_h^i\right)-\nu a_h(\bm{U}_h^i,\bm{U}_h^i)\right)\\
    &\quad\leq\sum_{i=1}^mb_i\left(\|\bm{f}^i\|_{L^2(\Omega)}\|\bm{U}_h^i\|_{L^2(\Omega)}-\nu\|\bm{U}_h^i\|^2_{1,h}\right)\\
    &\quad\leq\sum_{i=1}^mb_i\left(\frac{\varepsilon^{-1}}{2}\|\bm{f}^i\|_{L^2(\Omega)}^2+\frac{\varepsilon}{2}\|\bm{U}_h^i\|_{L^2(\Omega)}^2-\nu\|\bm{U}_h^i\|^2_{1,h}\right)\\
    &\quad\leq\sum_{i=1}^mb_i\left(\frac{\varepsilon^{-1}}{2}\|\bm{f}^i\|_{L^2(\Omega)}^2-\left(\nu-\frac{\varepsilon C}{2}\right)\|\bm{U}_h^i\|^2_{1,h}\right),
    \end{aligned}
\end{equation}
where $C$ is in \eqref{Korn} and $\varepsilon>0.$
It then follows from the estimate given above with $\varepsilon=2C^{-1}\nu$ that
\begin{equation*}
    \|\bm{u}_h^{n+1}\|^{2}_{L^2(\Omega)}\leq\|\bm{u}_h^{n}\|^{2}_{L^2(\Omega)}+\tau_n\sum_{i=1}^m\frac{\varepsilon^{-1}}{2}b_i\|\bm{f}^i\|_{L^2(\Omega)}^2.
\end{equation*}
The proof is complete.
\end{proof}
\begin{remark}
In order to prove \eqref{Korn} with $\bm{\tau}_h(\bm{v}_h)=\nabla_h\bm{v}_h+\nabla_h\bm{v}_h^T-\frac{2}{3}(\nabla_h\cdot\bm{v}_h)\mathbb{I}$, one needs the corresponding discrete Korn's inequality \eqref{diskorn},
which is not known in the literature. A possible proof should rely on the characterization of the kernel $\{\bm{v}\in [H^1(K)]^d: \bm{\tau}_h(\bm{v})=\bm{0}\}$ on each element $K\in\mathcal{T}_h$ and estimation of  suitable semi-norm associated with that kernel, see \cite{Brenner2004}. 
\end{remark}

For each $n\geq0,$ let  $$\delta_t\bm{u}_h^n:=\frac{\bm{u}_h^{n+1}-\bm{u}_h^n}{\tau_n},\quad\bm{u}_h^{n+\frac{1}{2}}:=\frac{\bm{u}_h^{n+1}+\bm{u}_h^n}{2},\quad{\bm{f}}^{n+\frac{1}{2}}:=\bm{f}\left(t_{n}+\frac{1}{2}\tau_n\right).$$ 
The Crank--Nicolson time discretization to \eqref{compact} can be written as
\begin{subequations}\label{CrankNicolson}
\begin{align} &\left(\delta_t\bm{u}^n_h,\vbold_h\right) + N_h\left(\bm{u}_h^{n+\frac{1}{2}};\bm{v}_h \right) - b_h \left( \bm{v}_h, p_h^{n+\frac{1}{2}}\right)=(\bm{f}^{n+\frac{1}{2}},\vbold_h),\quad\forall\bm{v}_h\in\bm{V}_h,\label{fulldis_1}\\
& b_h \left(\bm{u}^{n+\frac{1}{2}}_h, q_h \right) = 0,\quad\forall q_h\in Q_h.\label{fulldis_2}
\end{align}
\end{subequations}
Here $p_h^{n+\frac{1}{2}}$ approximates $p_h(t_n+\frac{1}{2}\tau_n)$.
The popular Crank--Nicolson scheme can be written as the 1-stage GLRK ($m=1$, $a_{11}=\frac{1}{2}$, $b_1=1$, $c_1=\frac{1}{2}$) as mentioned before. Therefore, we obtain the unconditional stability of the fully discrete scheme \eqref{CrankNicolson} from Theorem \ref{fullystability}.
\section{Three methods from the unified formulation} \label{secDG}
In this section, we derive three pressure-robust methods from the unified scheme \eqref{compact} proposed in Section~\ref{gen_method_sec}. For a positive integer $k$, we introduce the $H^1$-conforming Taylor--Hood finite element spaces 
\begin{equation}\label{VQH1}
\begin{aligned}
&\bm{V}_h^{\text{C}}:= \left\{\bm{v}_h \in [C^{0} \left(\Omega\right)]^{d}: \bm{v}_h |_{K} \in  \left[\mathcal{P}_{k+1} \left( K \right) \right]^d,~\forall K \in \mathcal{T}_h~\text{and}~\bm{v}_h|_{\partial\Omega}=0\right\},\\
&Q_h^{\text{C}}:= \left\{q_h \in C^{0} \left( \Omega \right)\cap L^{2}_{0}\left(\Omega\right):  q_{h} |_{K} \in\mathcal{P}_{k} \left( K \right),~\forall K \in \mathcal{T}_h \right\}.
\end{aligned}
\end{equation}
We shall also make use of the $\mathcal{P}_{k+1}\times\mathcal{P}_{k}$ discontinuous Galerkin spaces
\begin{equation}\label{VQDG}
    \begin{aligned}
&\bm{V}_h^{\text{DG}}:= \left\{ \bm{v}_h\in[{L}^{2}(\Omega)]^d: \bm{v}_h |_{K} \in  [\mathcal{P}_{k+1} \left( K \right)]^{d}, \forall K \in \mathcal{T}_h \right\},\\[1.5ex]
&{Q}^{\text{DG}}_h:= \left\{q_h\in L^{2}_{0}(\Omega): q_{h} |_{K} \in \mathcal{P}_k \left( K \right), \forall K \in \mathcal{T}_h \right\},
\end{aligned}
\end{equation}
where $k$ could be any nonnegative integer in \eqref{VQDG}.
Let $H(\text{div};\Omega):=\{\bm{v}\in[L^2(\Omega)]^d: \nabla\cdot\bm{v}\in L^2(\Omega)\}$. Let
$$\bm{\mathcal{Q}}_k(K):=[\mathcal{P}_{k+1}(K)]^d\quad\text{ or }\quad\bm{\mathcal{Q}}_k(K):=[\mathcal{P}_{k}(K)]^d+\mathcal{P}_{k}(K)\bm{x},$$
which is the Raviart--Thomas \cite{RT1977} or Brezzi--Douglas--Marini \cite{BDM1985} shape function space, respectively. The  $H(\text{div})$-conforming finite element space is 
\begin{align}\label{VQHdiv}
&\bm{V}_h^{\text{DIV}}:= \left\{ \bm{v}_h \in H\left(\text{div}; \Omega \right):  \bm{v}_h |_{K} \in  \bm{\mathcal{Q}}_k \left( K \right),~\forall K \in \mathcal{T}_h~\text{and}~\bm{v}_h\cdot\bm{n}|_{\partial\Omega}=0\right\}.
\end{align}
\subsection{Pressure robustness}
Consider the following discrete divergence-free space
\begin{align*}
\bm{Z}_{h}:=\{\bm{v}_h\in\bm{V}_h:b_h \left(\vbold_h, q_h \right) = 0,~\forall q_h\in Q_h\}.
\end{align*}
For error estimation of velocity in \eqref{compact}, the term $|b_h \left(\vbold_h, p-p_h \right)|$ with $\bm{v}_h\in\bm{Z}_h$ measures the inconsistency of convective bilinear form and serves as a guide to design $d_h(\bm{u}_h,\bm{v}_h)$. Note that this inconsistency is directly related to the concept of pressure robustness \cite{john2017divergence}, that is, the error in pressure induces a velocity error. The goal of $d_h$ is to reduce the influence of pressure approximation on velocity approximation, which in this paper is said to improve pressure robustness. For any $\vbold_h\in\bm{Z}_h$, we have $b_h \left(\vbold_h, p_h \right)=0$ and thus
\begin{align*}
&|b_h \left(\vbold_h, p - p_h\right)|=|b_h \left(\vbold_h, p \right)|\\
&=|(\nabla_h \cdot \vbold_h,p)-\ipbf{\llbracket \vbold_h \rrbracket\cdot\bm{n}}{\llcurve p \rrcurve}|\\
&\leq \|\nabla_h\cdot\bm{v}_h\|_{L^2(\Omega)}\|p\|_{L^2(\Omega)}+\frac{\varepsilon^{-1}}{2}\|\llbracket \vbold_h \rrbracket\cdot\bm{n}\|_{L^2(\mathcal{F}_h)}^{2}+\frac{\varepsilon}{2}\|\llcurve p \rrcurve\|^2_{L^{2}(\mathcal{F}_h)},
\end{align*}
where $0<\varepsilon\ll1$ is a small number. 
In view of $\|\nabla_h\cdot\bm{v}_h\|_{L^2(\Omega)}$ and $\|\llbracket \vbold_h \rrbracket\cdot\bm{n}\|_{L^2(\mathcal{F}_h)}$ in the previous  estimate, it is reasonable to add the penalization term
\begin{align}\label{dh}
d_h(\bm{u}_h,\bm{v}_h):=\gamma_{gd}(\nabla_h\cdot\bm{u}_h,\nabla_h\cdot\bm{v}_h)+\left\langle \gamma_{F}(\llbracket \ubold_h \rrbracket\cdot\nbold),\llbracket \vbold_h \rrbracket\cdot\nbold\right\rangle_{\mathcal{F}_h},
\end{align}
where $\gamma_{gd}$, $\{\gamma_F\}_{F\in\mathcal{F}_h}\geq 0$ are sufficiently large (piecewise) constants.
\subsection{\texorpdfstring{$H^{1}$}--conforming method}
Let $\bm{u}_h, \bm{v}_h\in\bm{V}^{\text{C}}_h$ and $p_h, q_h\in Q^{\text{C}}_h.$  Then the form \eqref{ahsym} simplifies to
\begin{gather}
a_h \left(\ubold_h, \vbold_h \right)=\begin{cases}
  (\nabla\ubold_h,\nabla \vbold_h) &\text{when $\bm{\tau}_h(\bm{u}_h):=\nabla_h\bm{u}_h$ }, \\[1.5ex]   
  (\nabla\ubold_h,\nabla \vbold_h)+(\nabla\cdot\bm{u}_h,\nabla\cdot\bm{v}_h) &\text{when $\bm{\tau}_h(\bm{u}_h):=\nabla_h\bm{u}_h+\nabla_h\bm{u}_h^{T}$},  \\[1.5ex]  
  (\nabla\ubold_h,\nabla \vbold_h)+\frac{1}{3}(\nabla\cdot\bm{u}_h,\nabla\cdot\bm{v}_h) &\text{when $\bm{\tau}_h(\bm{u}_h):=\nabla_h\bm{u}_h+\nabla_h\bm{u}_h^{T}-\frac{2}{3}\left(\nabla_h\cdot \bm{u}_h\right) \mathbb{I}$},
\end{cases}
\label{ah_H1}
\end{gather}
where the identity $(\nabla\bm{u}_h^{T},\nabla\bm{v}_h)=(\nabla\cdot\bm{u}_h,\nabla\cdot\bm{v}_h)$ under $\bm{u}_h|_{\partial\Omega}=\bm{0}$ is used. The forms \eqref{bh_form}, \eqref{convective_form}, \eqref{dh} simplify to
\begin{align*}
&b_h(\bm{u}_h,\bm{v}_h)=(\nabla_h \cdot \vbold_h,q_h),\\
&c_h \left(\bm{u}_h; \ubold_h, \vbold_h \right) = (\bm{u}_h \cdot \nabla \ubold_h,\vbold_h) + \frac{1}{2} (\left(\nabla \cdot \bm{u}_h \right) \ubold_h,\vbold_h), \\
&d_h(\bm{u}_h,\bm{v}_h)=\gamma_{gd}(\nabla_h\cdot\bm{v}_h,\nabla_h\cdot\bm{u}_h).
\end{align*}
Therefore,
the corresponding scheme \eqref{compact} with $\bm{V}_h=\bm{V}_h^\text{C}$, $Q_h=Q_h^\text{C}$ recovers the skew symmetric formulation \cite{charnyi2017conservation,layton2008introduction} with grad-div stabilization \cite{franca1988two,john2017divergence}. 
The inf-sup condition is guaranteed by 
\begin{align*}
C\|q_h\|_{L^{2}(\Omega)}\leq \sup_{\bm{v}_h\in\bm{V}_h^{\text{C}}\setminus \{\bm{0}\}}\frac{(q_h ,\nabla \cdot \vbold_h)}{\left\|\vbold_h \right\|_{H^1(\Omega)}},\quad\forall q_h \in Q_h^{\text{C}}.
\end{align*}
\begin{remark}
The $H^1$-conforming method has a minimum number of degrees of freedom, hence significantly reduces the computational cost. 
However,  the $H^1$-conforming method is not able to handle convection dominated flows. 
\end{remark}
\subsection{\texorpdfstring{$H(\emph{div})$}--conforming method}
Let $\bm{u}_h, \bm{v}_h\in\bm{V}^{\text{DIV}}_h$ and $p_h, q_h\in Q^{\text{DG}}_h.$ It follows from $\llbracket\bm{u}_h\cdot\bm{n}\rrbracket=0$ on $\mathcal{F}_h$, the inclusion  $\nabla\cdot\bm{V}_{h}^{\text{DIV}}\subset Q_h^{\text{DG}}$, and  \eqref{compact2} that 
$$\nabla\cdot\bm{u}_h=0\quad\text{ on }\Omega.$$ 
Therefore the full viscous strain tensor $\bm{\tau}_h(\bm{u}_h)=\nabla_h\bm{u}_h+\nabla_h\bm{u}^T_h-\frac{2}{3}(\nabla\cdot\ubold_h)\mathbb{I}$ and the symmetric gradient strain tensor $\bm{\tau}_h(\bm{u}_h)=\nabla_h\bm{u}_h+\nabla_h\bm{u}^T_h$ coincide. Then $a_h$ and $b_h$ reduce to
\begin{gather*}
a_h \left(\ubold_h, \vbold_h \right)=\begin{cases}
   (\nabla_h\bm{u}_h,\nabla_h \vbold_h)   -\ipbf{\llbracket \vbold_h \rrbracket}{\llcurve \nabla_h\bm{u}_h \rrcurve \bm{n}}  
\\\quad-\ipbf{\llbracket \ubold_h \rrbracket}{\llcurve  \nabla_h\bm{v}_h \rrcurve \bm{n}} 
+\ipbf{\eta h^{-1} \llbracket \ubold_h \rrbracket}{\llbracket \vbold_h \rrbracket}, &\text{when $\bm{\tau}_h(\bm{u}_h):=\nabla_h\bm{u}_h$,}   \\ 
  (\nabla_h\bm{u}_h+\nabla_h\bm{u}_h^{T},\nabla_h \vbold_h) -\ipbf{\llbracket \vbold_h \rrbracket}{\llcurve \nabla_h\bm{u}_h+\nabla_h\bm{u}_h^{T} \rrcurve \bm{n}}  
\\\quad-\ipbf{\llbracket \ubold_h \rrbracket}{\llcurve  \nabla_h\bm{v}_h+\nabla_h\bm{v}_h^{T} \rrcurve \bm{n}} 
+\ipbf{\eta h^{-1} \llbracket \ubold_h \rrbracket}{\llbracket \vbold_h \rrbracket}  &\parbox{15em}{when $\bm{\tau}_h(\bm{u}_h):=\nabla_h\bm{u}_h+\nabla_h\bm{u}_h^{T}$ or\\ $\nabla_h\bm{u}_h+\nabla_h\bm{u}_h^{T}-\frac{2}{3}\left(\nabla_h\cdot \bm{u}_h\right) \mathbb{I}$.}
\end{cases}
\end{gather*} 
and $$b_h(\bm{u}_h,\bm{v}_h)=(\nabla_h \cdot \vbold_h,q_h),$$ respectively. Similarly, using $\nabla\cdot\bm{u}_h=0$ and $\llbracket\bm{u}_h\cdot\bm{n}\rrbracket=0$, we obtain the simplified convective term
\begin{align*}
c_h \left(\bm{u}_h; \ubold_h, \vbold_h \right) = (\bm{u}_h \cdot \nabla_h \ubold_h,\vbold_h) -\iipbf{ \left( \bm{u}_h \cdot \nbold \right) \llbracket \ubold_h \rrbracket}{\llcurve \vbold_h \rrcurve}+  \iipbf{\zeta\left| \bm{u}_h \cdot \nbold \right|\llbracket \ubold_h \rrbracket}{\llbracket \vbold_h \rrbracket},
\end{align*}{}
and the vanishing penalty term $d_h$ \eqref{dh}, i.e.,
$$d_h(\bm{u}_h,\bm{v}_h)=0.$$
In this case, the scheme \eqref{compact} with $\bm{V}_h=\bm{V}_h^{\text{DIV}}$, $Q_h={Q}^{\text{DG}}_h$ reduces to the classical $H(\text{div})$-conforming method \cite{guzman2016h,schroeder2018divergence}, but with symmetrical gradient formulation for the viscous bilinear form. 
Finally the inf-sup condition is guaranteed by 
\begin{align*}
C\|q_h\|_{L^{2}(\Omega)}\leq \underset{\vbold_h \in \bm{V}_h^{\text{DIV}} \setminus \{\bm{0}\}}{\sup} \frac{(q_h ,\nabla \cdot \vbold_h)}{\left\|\vbold_h \right\|_{H(\text{div};\Omega)}},\quad\forall q_h \in Q_h^{\text{DG}}.
\end{align*}
\begin{remark}
The $H(\emph{div})$-conforming method is naturally pressure robust, since the pressure approximation is completely decoupled from the velocity approximation \cite{guzman2016h,john2017divergence}. With the help of upwind flux ($\zeta\geq 0$), the $H(\emph{div})$-conforming method could deal with convection dominated flow. 
\end{remark}
\subsection{Discontinuous Galerkin method}\label{subsecDG}
The scheme \eqref{compact} with $\bm{V}_h=\bm{V}^{\text{DG}}_h$ and $Q_h=Q_h^{\text{DG}}$ yields our DG scheme. If the normal component of the velocity is penalized sufficiently, then we obtain that for $\bm{v}_h\in\bm{Z}_h$, 
\begin{align*}
0=b_h \left(\vbold_h, q_h \right) =(\nabla_h \cdot \vbold_h,q_h)-\ipbf{\llbracket \vbold_h \rrbracket\cdot\nbold_F}{\llcurve q_h \rrcurve}\approx (\nabla_h \cdot \vbold_h,q_h),\quad\forall q_h\in Q_h^{\text{DG}}.
\end{align*}
By $\nabla_h\cdot\bm{V}_h^{\text{DG}}\subseteq Q_h^{\text{DG}}$ and the previous reasoning, we may further conclude that 
\begin{align*}
\nabla_h\cdot\bm{v}_h\approx 0.
\end{align*}
Hence $d_h$ should be of the form 
\begin{align}
d_h(\bm{u}_h,\bm{v}_h)=\sum_{F\in\mathcal{F}_h}\gamma_{F}\langle\llbracket \vbold_h \rrbracket\cdot\nbold_F,\llbracket \ubold_h \rrbracket\cdot\nbold_F\rangle_F.
\label{Q_1}
\end{align}
Assuming $\gamma_F=\gamma h^{-1}_F$, (cf.~\cite{akbas2018analogue,guzman2016h}), the penalty term $d_h$ further simplifies to
\begin{align}
d_h(\bm{u}_h,\bm{v}_h)=\gamma\sum_{F\in\mathcal{F}_h}h_F^{-1}\langle\llbracket \vbold_h \rrbracket\cdot\nbold_F,\llbracket \ubold_h \rrbracket\cdot\nbold_F\rangle_F,\label{Penalty_expression}
\end{align}
where $\gamma$ is a sufficiently large parameter. The symmetric form in \eqref{ahsym} is given as
\begin{gather*}
a_h \left(\ubold_h, \vbold_h \right)=\begin{cases}
   (\nabla_h\bm{u}_h,\nabla_h \vbold_h)   -\ipbf{\llbracket \vbold_h \rrbracket}{\llcurve \nabla_h\bm{u}_h \rrcurve \bm{n}}  
\\\quad-\ipbf{\llbracket \ubold_h \rrbracket}{\llcurve  \nabla_h\bm{v}_h \rrcurve \bm{n}} 
+\ipbf{\eta h^{-1} \llbracket \ubold_h \rrbracket}{\llbracket \vbold_h \rrbracket}, &\text{when $\bm{\tau}_h(\bm{u}_h):=\nabla_h\bm{u}_h$,} \\  
  (\nabla_h\bm{u}_h+\nabla_h\bm{u}_h^{T},\nabla_h \vbold_h) -\ipbf{\llbracket \vbold_h \rrbracket}{\llcurve \nabla_h\bm{u}_h+\nabla_h\bm{u}_h^{T} \rrcurve \bm{n}}  
\\\quad-\ipbf{\llbracket \ubold_h \rrbracket}{\llcurve  \nabla_h\bm{v}_h+\nabla_h\bm{v}_h^{T} \rrcurve \bm{n}} 
+\ipbf{\eta h^{-1} \llbracket \ubold_h \rrbracket}{\llbracket \vbold_h \rrbracket}, &\parbox{15em}{when $\bm{\tau}_h(\bm{u}_h):=\nabla_h\bm{u}_h+\nabla_h\bm{u}_h^{T}$,}\\
  (\nabla_h\bm{u}_h+\nabla_h\bm{u}_h^{T}-\frac{2}{3}\left(\nabla_h\cdot \bm{u}_h\right) \mathbb{I},\nabla_h \vbold_h)+\ipbf{\eta h^{-1} \llbracket \ubold_h \rrbracket}{\llbracket \vbold_h \rrbracket}\\ \quad-\ipbf{\llbracket \vbold_h \rrbracket}{\llcurve \nabla_h\bm{u}_h+\nabla_h\bm{u}_h^{T}-\frac{2}{3}\left(\nabla_h\cdot \bm{u}_h\right) \mathbb{I} \rrcurve \bm{n}}  
\\\quad-\ipbf{\llbracket \ubold_h \rrbracket}{\llcurve  \nabla_h\bm{v}_h+\nabla_h\bm{v}_h^{T}-\frac{2}{3}\left(\nabla_h\cdot \bm{v}_h\right) \mathbb{I} \rrcurve \bm{n}},  &\parbox{15em}{when $\bm{\tau}_h(\bm{u}_h):=\nabla_h\bm{u}_h+\nabla_h\bm{u}_h^{T}\\-\frac{2}{3}\left(\nabla_h\cdot \bm{u}_h\right) \mathbb{I}$.}
\end{cases}
\end{gather*}
For $b_h(\ubold_h,\vbold_h)$ and $c_h(\bm{u}_h;\bm{u}_h,\bm{v}_h)$, we use the same form as in \eqref{bh_form} and \eqref{convective_form}, respectively. Finally the pressure stability is guaranteed by observing the following inf-sup condition \cite{akbas2018analogue}
\begin{align*}
C\|q_h\|_{L^{2}(\Omega)}\leq \underset{\vbold_h \in \bm{V}_h^{\text{DG}} \setminus \{\bm{0}\}}{\sup} \frac{b_h \left(\vbold_h, q_h \right)}{\left\|\vbold_h \right\|_{\text{sip}}},\quad
\forall q_h \in Q_h^{\text{DG}},
\end{align*}{}
where  $\|\bm{v}_h\|_{\text{sip}}:=\left(\|\bm{v}_h\|^{2}_{L^{2}(\Omega)}+\|h^{-\frac{1}{2}}\llbracket \vbold_h \rrbracket\|^{2}_{L^{2}(\mathcal{F}_h)}\right)^\frac{1}{2}.$
\begin{remark}\label{data_enforce}
It is clear that $\bm{u}|_{\partial\Omega}=\bm{0}$ holds point-wise on the boundary for the $H^{1}$-conforming method. However, for $H(\emph{\text{div}})$-conforming and DG methods, $\bm{u}|_{\partial\Omega}=\bm{0}$ is weakly imposed in \eqref{compact}. In fact, any non-homogeneous Dirichlet boundary condition could be weakly enforced via modifying the right hand side of \eqref{CrankNicolson}. 
In particular, the scheme \eqref{CrankNicolson} under the boundary condition $\bm{u}|_{\partial\Omega}=\bm{g}$ is modified as 
\begin{align*}
&\left(\delta_t\bm{u}^n_h,\vbold_h\right) + N_h\left(\bm{u}_h^{n+\frac{1}{2}};\bm{v}_h \right) - b_h \left( \bm{v}_h, p_h^{n+\frac{1}{2}}\right)\\
&\quad=(\bm{f}^{n+\frac{1}{2}},\vbold_h)+\nu f_{h,a}^{n+\frac{1}{2}}(\vbold_h)+ sf_{h,c}^{n+\frac{1}{2}}(\vbold_h)+sf_{h,d}^{n+\frac{1}{2}}(\vbold_h),\quad\forall\bm{v}_h\in\bm{V}_h,\\
& b_h \left(\bm{u}^{n+\frac{1}{2}}_h, q_h \right) = sf_{h,b}^{n+\frac{1}{2}} \left(q_h \right),\quad\forall q_h\in Q_h,
\end{align*}
where $s=0$ for the $H(\emph{div})$ scheme ($\bm{V}_h\times Q_h=\bm{V}_h^{\emph{DIV}}\times Q_h^{\emph{DG}}$) and $s=1$ for the DG scheme ($\bm{V}_h\times Q_h=\bm{V}_h^{\emph{DG}}\times Q_h^{\emph{DG}}$). The newly introduced $f_{h,a}^{n+\frac{1}{2}}, f_{h,b}^{n+\frac{1}{2}}, f_{h,c}^{n+\frac{1}{2}}, f_{h,d}^{n+\frac{1}{2}}$ are defined as 
\begin{align*}
&f_{h,a}^{n+\frac{1}{2}}(\vbold_h):=\sum_{F\in\mathcal{F}^{\partial}_{h}}\frac{\eta}{h_{F}}\ipbtkf{\gbold^{n+\frac{1}{2}}}{\vbold_h}-\sum_{F\in \mathcal{F}^{\partial}_{h}}\ipbtkf{\gbold^{n+\frac{1}{2}}}{\bm{\tau}_h(\bm{v}_h)\nbold},\\
&f_{h,b}^{n+\frac{1}{2}}  \left(q_h \right):=\sum_{F\in \mathcal{F}^{\partial}_{h}}\ipbtkf{\gbold^{n+\frac{1}{2}}}{q_h\nbold},\\
&f_{h,c}^{n+\frac{1}{2}}(\vbold_h):=\sum_{F\in\mathcal{F}^{\partial}_{h}}\ipbtkf{\zeta\big|\gbold^{n+\frac{1}{2}}\cdot\nbold\big|\gbold^{n+\frac{1}{2}}}{\vbold_h},\\
&f_{h,d}^{n+\frac{1}{2}}(\vbold_h):=\sum_{F\in \mathcal{F}^{\partial}_{h}}\frac{\gamma}{h_{F}}\ipbtkf{\gbold^{n+\frac{1}{2}}\cdot\nbold}{\vbold_h\cdot\nbold}.
\end{align*}

\end{remark}
\begin{remark}
Similarly to the $H(\emph{div})$-conforming method, the DG method is able to handle convection dominated flows when upwind flux is introduced. In addition, the DG scheme allows non-conforming and polygonal meshes. However, the DG scheme may lack pressure robustness, which could be cured by increasing the parameter $\gamma$.
\end{remark}

It is worth mentioning that energy-stable and convergent $H(\text{div})$ and DG schemes in \cite{guzman2016h} are designed for the Euler equation modelling incompressible and inviscid flows. Our $H(\text{div})$-conforming scheme shares the same convective form $c_h$ with the $H(\text{div})$ scheme in \cite{guzman2016h}. However, in contrast to our DG scheme, the convective form $c_h$ of the DG scheme in \cite{guzman2016h} relies on a postprocessed velocity. We also point out that,  only semi-discrete stability is shown in \cite{guzman2016h} and the BDF1 time integrator used in the fully discrete scheme there might not yield decaying numerical energy.
\section{Numerical Experiments}\label{secNE}
In this section, we test the performance of several methods in the form \eqref{CrankNicolson} with $$\bm{V}_h\times Q_h=\bm{V}^\text{C}_h\times Q^\text{C}_h,\text{ or } \bm{V}^{\text{DIV}}_h\times Q^{\text{DG}}_h,\text{ or }\bm{V}_h^{\text{DG}}\times Q^{\text{DG}}_h.$$ 
The corresponding scheme is denoted as Scheme $H^1$, $H(\text{div})$, or DG-N, respectively. The viscous strain tensor in \eqref{CrankNicolson} is chosen as $\bm{\tau}_h(\bm{u}_h)=\nabla_h\bm{u}_h+\nabla_h\bm{u}_h^{T}-\frac{2}{3}\left(\nabla_h\cdot \bm{u}_h\right) \mathbb{I}$.  Schemes $H^{1}$ and $H(\text{div})$ are considered in the first and second experiments, while the DG-N scheme from our framework are tested in all experiments. Recall that the  incompressibility condition $\nabla\cdot\bm{u}=0$ is weakly enforced via the condition $b_h(\bm{u}_h,q_h)=0\ \forall q_h\in Q_h$, where the bilinear form $b_h$ is introduced in Section \ref{secDG}. When implementing our schemes, that condition yields the linear system of equations $B_hU_h=\bm{0}$, where $U_h$ is the vector representation of $\bm{u}_h$ and $B_h$ is a matrix representing $b_h.$ For nonhomogeneous boundary condition, the right hand side of $b_h(\bm{u}_h,q_h)=0$ (and $B_hU_h=\bm{0}$) is modified as discussed in Remark \ref{data_enforce}. Although our framework is designed for unsteady problems, we compare our DG-N spatial discretization  with the scheme proposed in \cite{akbas2018analogue,Ern2012}, which we will denote as DG-C and is of the form \eqref{compact} with the following bilinear and convective forms
\begin{align*}
&a_h(\ubold_h,\vbold_h)=(\nabla_h\bm{u}_h,\nabla_h \vbold_h)   -\ipbf{\llbracket \vbold_h \rrbracket}{\llcurve \nabla_h\bm{u}_h \rrcurve \bm{n}}-\ipbf{\llbracket \ubold_h \rrbracket}{\llcurve  \nabla_h\bm{v}_h \rrcurve \bm{n}} 
+\ipbf{\eta h^{-1} \llbracket \ubold_h \rrbracket}{\llbracket \vbold_h \rrbracket},\\[1.5ex]
&b_h(\bm{v}_h, q_h)= (\nabla_h \cdot \vbold_h,q_h)-\iipbfb{\llbracket \vbold_h \rrbracket\cdot\nbold}{\llcurve q_h \rrcurve},\\[1.5ex]
&c_h \left(\ubold_h; \ubold_h, \vbold_h \right) = (\ubold_h \cdot \nabla_h \ubold_h,\vbold_h)- \iipbf{ \left( \llcurve\ubold_h\rrcurve \cdot \nbold \right) \llbracket \ubold_h \rrbracket}{\llcurve \vbold_h \rrcurve}\nonumber+ \iipbf{\frac{1}{2} \left| \llcurve\ubold_h\rrcurve \cdot \bm{n} \right|\llbracket \ubold_h \rrbracket}{\llbracket \vbold_h \rrbracket},\\[1.5ex]
&d_h(\bm{u}_h,\bm{v}_h)=\gamma(\nabla\cdot\ubold_h,\nabla\cdot\vbold_h)+\gamma\sum_{F\in\mathcal{F}_h}h_F^{-1}\langle\llbracket \vbold_h \rrbracket\cdot\nbold_F,\llbracket \ubold_h \rrbracket\cdot\nbold_F\rangle_F.
\end{align*}
In contrast to DG-N, the scheme DG-C in \cite{akbas2018analogue,Ern2012} is designed only for steady incompressible flow and not proved to be energy stable for unsteady flow.

In  $a_h$ and $c_h$,  the penalization parameters $\eta$ and $\zeta$ are empirically set to be $\eta=3(k+1)(k+2)$ (cf.~\cite{schroeder2018divergence}) and $0.5$ respectively, where $k$ is the degree of polynomials in \eqref{VQDG}. The penalty parameters $\gamma$ (for DG-N and DG-C) and $\gamma_{gd}$ (for $H^{1}$) will be specified in each numerical example. 

The numerical simulations are performed in FEniCS \cite{LangtangenLogg2017}  on a laptop with Intel Core i5 CPU (2.7 GHz) and 8 GB RAM. We use the Newton nonlinear solver with the \textsf{MUMPS} linear solver inside FEniCS to solve the nonlinear systems of equations arising from fully discrete schemes. We set absolute and relative error tolerances used in the Newton solver to be $10^{-8}$ for dynamic problems and $10^{-10}$ for stationary problems.  

\subsection{Taylor--Green Vortex}\label{Tay_Ex}
The analytical solutions of Taylor--Green vortex \cite{guzman2016h} in $\mathbb{R}^2$ are given by 
\begin{align*}
    &\ubold(t,\xbold)=\bigg(\sin(x_{1})\cos(x_{2})e^{-2\nu t},-\cos(x_{1})\sin(x_{2})e^{-2\nu t}\bigg), \\ &p(t,\xbold)=\frac{1}{4}\bigg(\cos(2x_{1})+\cos(2x_{2})\bigg)e^{-4\nu t}
\end{align*}
with $\nu=0.01$.
The space domain and time interval are set to be $\Omega:=[0,2\pi]^2$  and $[0,T]$ with $T=1s$, respectively. All  schemes are based on the Crank--Nicolson time discretization  with uniform time step $\tau=0.01s$. The space domain is partitioned by uniform meshes with mesh sizes $h_{\text{max}}\in\{0.8886, 0.4443, 0.2221, 0.1777\}$, see Figure~\ref{fig:mesh_example} for sample meshes. For the $H^{1}$ scheme, we choose $k\in\{1, 2\}$ in \eqref{VQH1} and $\gamma_{gd}=0$ in \eqref{dh}. Note that the Taylor--Hood  space \eqref{VQH1} with $k=0$ is not inf-sup stable. For $H(\text{div})$ and DG schemes,  we set $k\in\{0, 1, 2\}$ in \eqref{VQDG} and $\bm{\mathcal{Q}}_k(K)=[\mathcal{P}_{k+1}(K)]^d$ (Brezzi-Douglas-Marini element) in \eqref{VQHdiv}. In addition, the DG scheme uses the penalty parameter $\gamma\in\{0,10\}$.  Numerical results are presented in Figure \ref{DG_vorticity_taylor_green} and Tables \ref{Taylor_H1div} to \ref{table_unsteady_0}.
\begin{figure}[ht]
\centering
\begin{subfigure}{0.4\textwidth}
  \centering
  \includegraphics[width=1.0\linewidth,trim={6.5cm 2cm 0cm 2cm},clip]{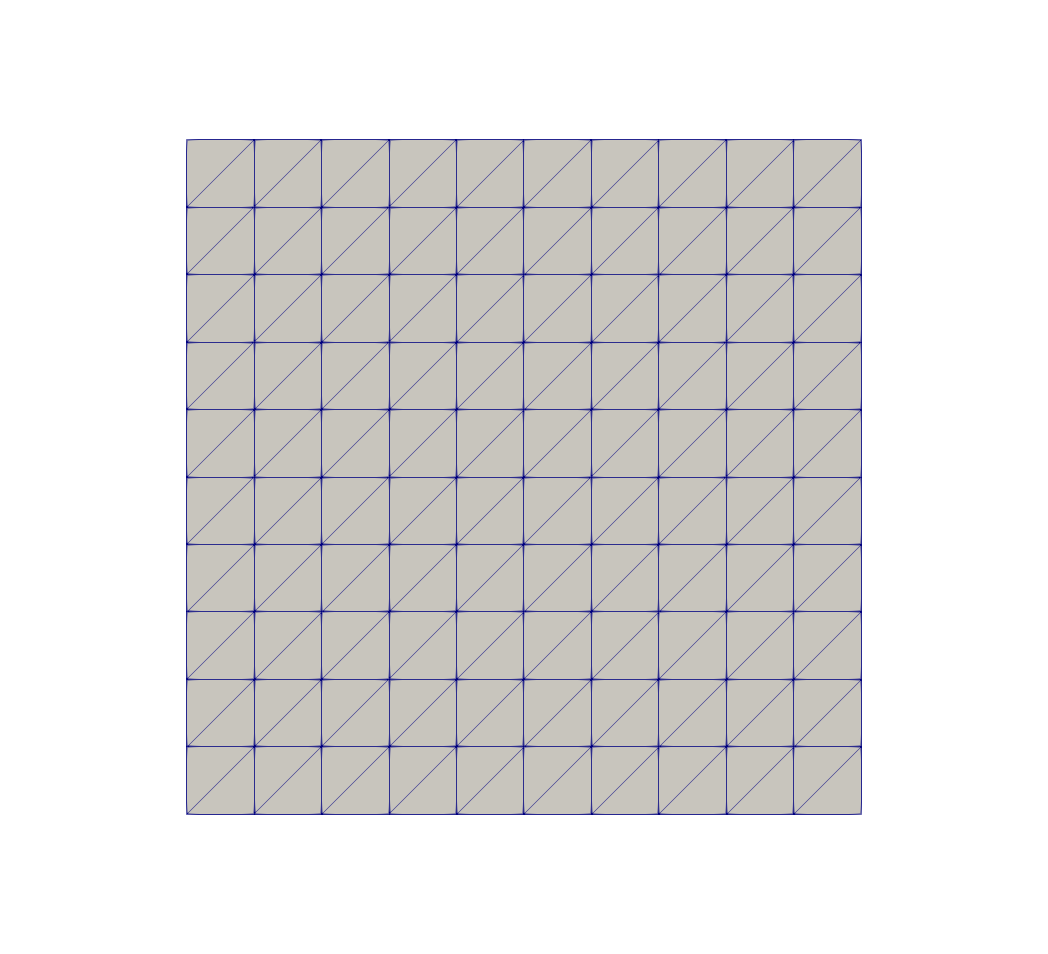}
\end{subfigure}%
\begin{subfigure}{.4\textwidth}
  \centering
  \includegraphics[width=1.0\linewidth,trim={6.5cm 2cm 0cm 2cm},clip]{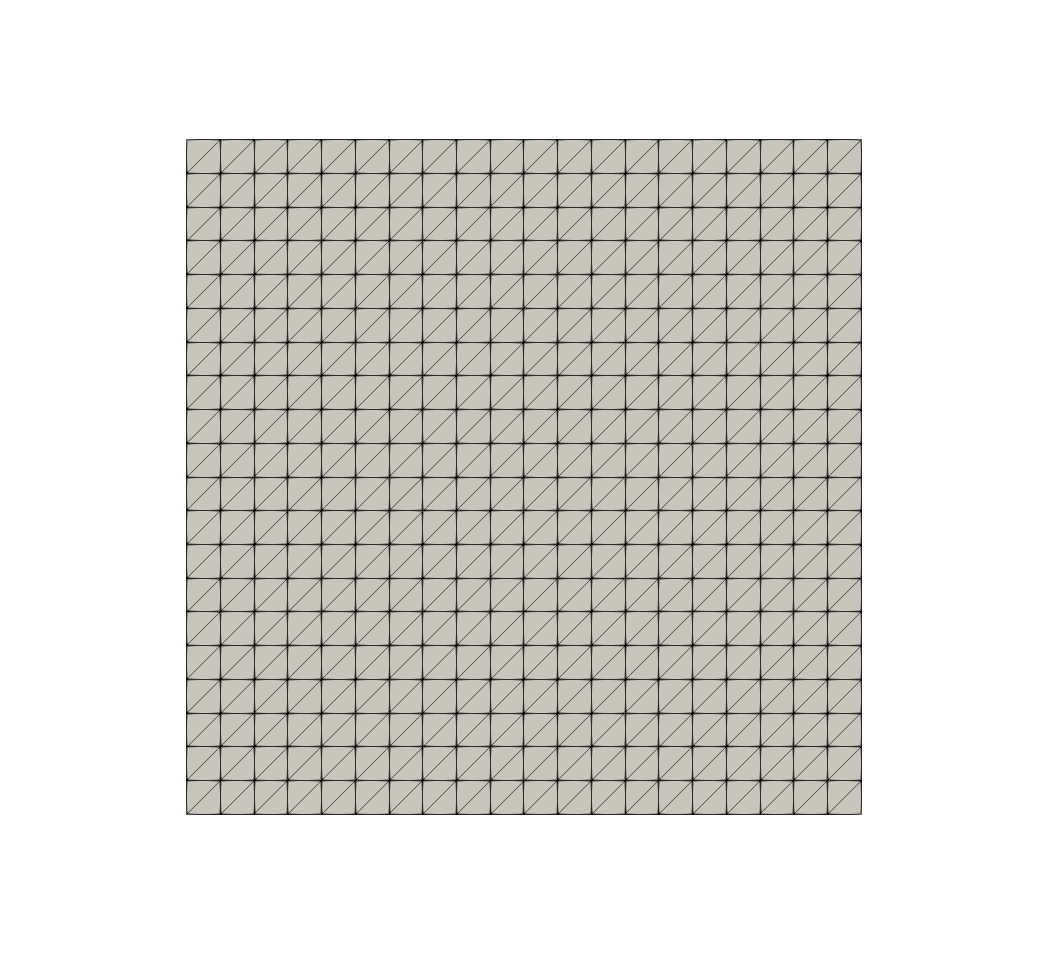}
\end{subfigure}
\caption{Uniform Meshes}
\label{fig:mesh_example}
\end{figure}
\begin{figure*}
\centering
\begin{multicols}{2}
    \includegraphics[width=1.0\linewidth,trim={7cm 2cm 0cm 2cm},clip]{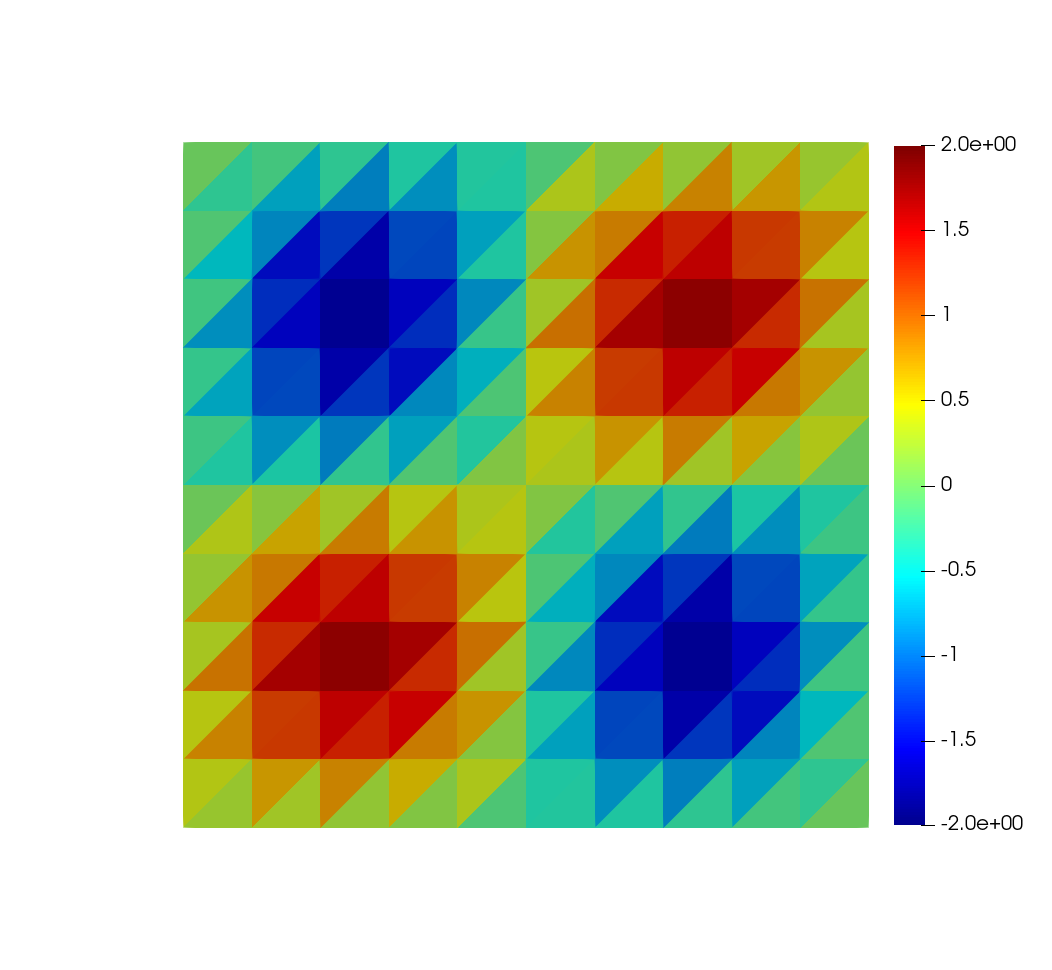}\par
    \includegraphics[width=1.0\linewidth,trim={7cm 2cm 0cm 2cm},clip]{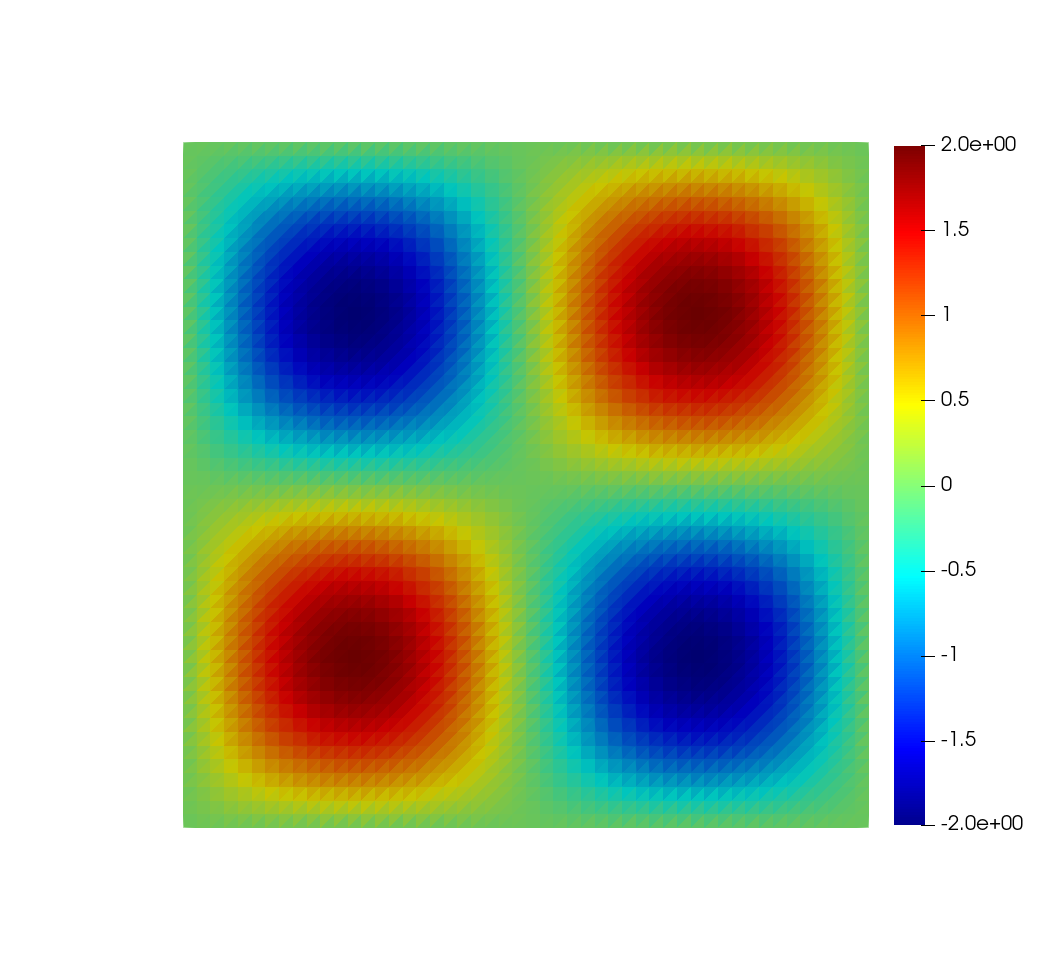}\par 
    \end{multicols}
\begin{multicols}{2}
    \includegraphics[width=1.0\linewidth,trim={7cm 2cm 0cm 2cm},clip]{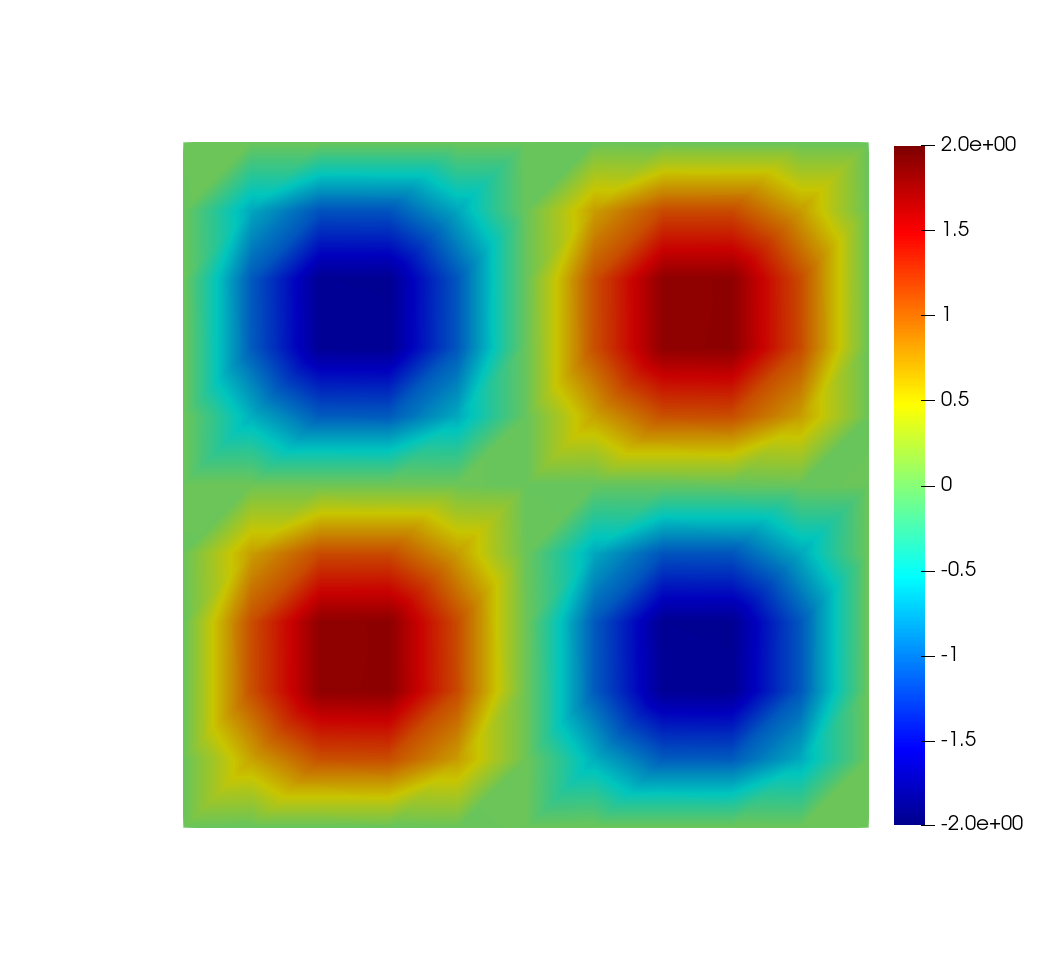}\par
    \includegraphics[width=1.0\linewidth,trim={7cm 2cm 0cm 2cm},clip]{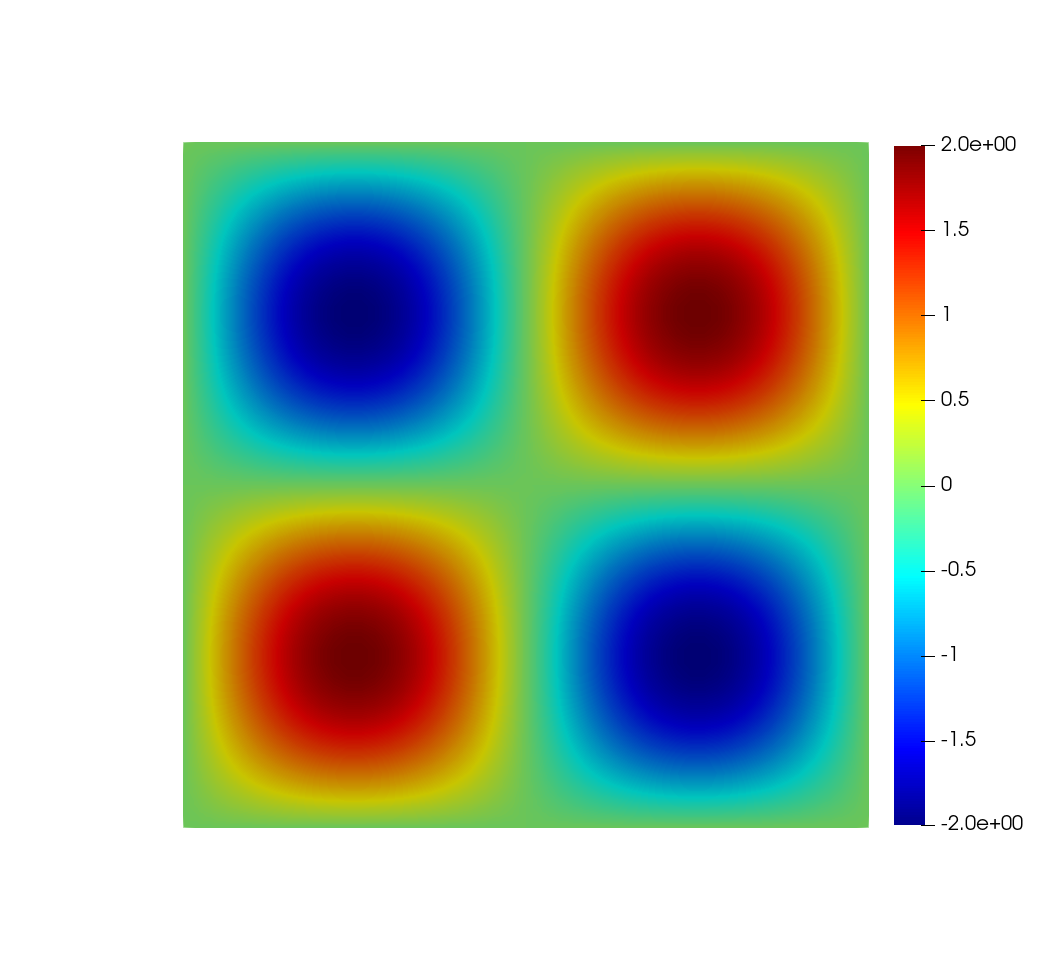}\par
\end{multicols}
\caption{Taylor--Green Vortex: Contours of vorticity $\nabla_h\times\bm{u}_h$ from DG-N with $h_{\text{max}} = 0.8886$ (left) and $h_{\text{max}} = 0.1777$ (right) at $t = 1.0s$ when $k=0$ (upper row) and $k=2$ (bottom row), $\nu=0.01$, $\gamma=10$.}
\label{DG_vorticity_taylor_green}
\end{figure*}

From Tables \ref{table_unsteady_10} and \ref{table_unsteady_0}, we observe that both DG schemes achieve the expected convergence rates (when $\gamma=10$), and achieve roughly the same level of accuracy for both velocity and pressure with the same order of runtime. In addition, we observe a decreasing of errors in both velocity and pressure when we increase $\gamma$ from $0$ to $10$.  In order to ensure stability, and to test the behaviors of the full viscous strain tensor, our DG scheme has more terms (in both $a_h$ and $c_h$) to be updated at each time step compared with DG-C. Therefore, the running time of DG-N scheme is slightly longer. We do not observe a clear trend of runtime when we increase $\gamma$ from $0$ to $10$, for both DG schemes. Figure~\ref{DG_vorticity_taylor_green} shows that the approximation becomes better when increasing polynomial degree and/or decreasing mesh sizes for our DG-N scheme.

Numerical results on $H^{1}$ and $H(\text{div})$  schemes are presented in Table \ref{Taylor_H1div}.
Due to  smaller numbers of degrees of freedom, the runtime of the  $H^1$ scheme is less than $H(\text{div})$ and DG schemes. An interesting phenomenon is the apparent superconvergence of  the $H^{1}$ scheme when $k=1$.  It can be observed from Tables \ref{Taylor_H1div} and  \ref{table_unsteady_10} that errors of  $H(\text{div})$, DG-N, and DG-C schemes are of the same magnitude, while the $H^1$ scheme is much less accurate.  It is noted that the $H(\text{div})$ scheme has a longer running time than DG schemes although it has less number of degrees of freedom and a simpler expression. We will not pursue a rigorous explanation on this and conjecture that the `unreasonble' runtime of $H(\text{div})$ schemes might be due to the inefficiency of assembling process for Brezzi--Douglas--Marini elements in FEniCS.

\begin{table}
\begin{center}
\resizebox{\columnwidth}{!}{
\begin{tabular}{|l|l|c|l|l|l|l|c|l|c|l|l|l|l|l|c|}
\hline
\multicolumn{1}{|c|}{\multirow{2}{*}{$k$}} & \multicolumn{1}{c|}{\multirow{2}{*}{$h_{\text{max}}$}} & \multicolumn{7}{c|}{$H^1$}                                                                                                                                                                                                                        & \multicolumn{7}{c|}{$H(\text{div})$}                                                                                                                                                                          \\ \cline{3-16} 
\multicolumn{1}{|c|}{}                   & \multicolumn{1}{c|}{}                   & \multicolumn{1}{l|}{d.o.f} & \multicolumn{2}{c|}{\begin{tabular}[c]{@{}c@{}}$\|\ubold-\ubold_h\|_{L^{2}(\Omega)}$\\ error   order\end{tabular}} & \multicolumn{2}{c|}{\begin{tabular}[c]{@{}c@{}}$\|p-p_h\|_{L^{2}(\Omega)}$\\ error   order\end{tabular}} & \multicolumn{2}{c|}{Runtime}  & d.o.f  & \multicolumn{2}{c|}{\begin{tabular}[c]{@{}c@{}}$\|\ubold-\ubold_h\|_{L^{2}(\Omega)}$\\ error   order\end{tabular}} & \multicolumn{3}{c|}{\begin{tabular}[c]{@{}c@{}}$\|p-p_h\|_{L^{2}(\Omega)}$\\ error   order\end{tabular}} & Runtime \\ \hline
\multirow{4}{*}{0}                       & 0.8886                                  & N/A                        & \multicolumn{2}{c|}{N/A}                                                                & \multicolumn{2}{c|}{N/A}                                                                & \multicolumn{2}{c|}{N/A}      & 841    & \multicolumn{2}{c|}{2.26e-1\:\quad   -----}                                                      & \multicolumn{3}{c|}{4.55e-1\:\quad   -----}                                                      & 1.97s   \\ \cline{2-16} 
& 0.4443                                  & N/A                        & \multicolumn{2}{c|}{N/A}                                                                & \multicolumn{2}{c|}{N/A}                                                                & \multicolumn{2}{c|}{N/A}      & 3281   & \multicolumn{2}{c|}{5.21e-2\quad  2.12}                                                      & \multicolumn{3}{c|}{2.25e-1\quad  1.01}                                                      & 5.71s   \\ \cline{2-16} 
& 0.2221                                  & N/A                        & \multicolumn{2}{c|}{N/A}                                                                & \multicolumn{2}{c|}{N/A}                                                                & \multicolumn{2}{c|}{N/A}      & 12961  & \multicolumn{2}{c|}{1.20e-2\quad  2.11}                                                      & \multicolumn{3}{c|}{1.12e-1\quad  1.01}                                                      & 20.78s  \\ \cline{2-16} 
& 0.1777                                  & N/A                        & \multicolumn{2}{c|}{N/A}                                                                & \multicolumn{2}{c|}{N/A}                                                                & \multicolumn{2}{c|}{N/A}      & 20201  & \multicolumn{2}{c|}{7.57e-3\quad  2.09}                                                      & \multicolumn{3}{c|}{8.97e-2\quad  1.00}                                                      & 35.62s  \\ \hline
\multirow{4}{*}{1}                       & 0.8886                                  & 1004                       & \multicolumn{2}{c|}{2.86e-1\:\quad -----}                                                      & \multicolumn{2}{c|}{1.54e-1\:\quad   -----}                                                      & \multicolumn{2}{c|}{1.83es} & 2161   & \multicolumn{2}{c|}{2.01e-2\:\quad   -----}                                                      & \multicolumn{3}{c|}{6.80e-2\:\quad   -----}                                                      & 5.17s   \\ \cline{2-16} 
& 0.4443                                  & 3804                       & \multicolumn{2}{c|}{2.55e-2\quad  3.49}                                                      & \multicolumn{2}{c|}{2.37e-2\quad  2.70}                                                      & \multicolumn{2}{c|}{4.66s}    & 8521   & \multicolumn{2}{c|}{2.44e-3\quad 3.04}                                                      & \multicolumn{3}{c|}{1.72e-2\quad  1.99}                                                      & 20.31s  \\ \cline{2-16} 
& 0.2221                                  & 14804                      & \multicolumn{2}{c|}{1.52e-3\quad  4.07}                                                      & \multicolumn{2}{c|}{5.62e-3\quad  2.08}                                                      & \multicolumn{2}{c|}{16.18s}    & 33841  & \multicolumn{2}{c|}{2.93e-4\quad  3.06}                                                      & \multicolumn{3}{l|}{4.31e-3\quad  2.00}                                                      & 104.10s  \\ \cline{2-16} 
& 0.1777                                  & 23004                      & \multicolumn{2}{c|}{6.33e-4\quad  3.92}                                                      & \multicolumn{2}{c|}{3.58e-3\quad  2.02}                                                      & \multicolumn{2}{c|}{27.13s}   & 52801  & \multicolumn{2}{c|}{1.49e-4\quad  3.04}                                                      & \multicolumn{3}{c|}{2.76e-3\quad  2.00}                                                      & 170.04s  \\ \hline
\multirow{4}{*}{2}                       & 0.8886                                  & 2364                       & \multicolumn{2}{c|}{5.03e-2\:\quad   -----}                                                      & \multicolumn{2}{c|}{2.56e-2\:\quad   -----}                                                      & \multicolumn{2}{c|}{4.26s}    & 4081   & \multicolumn{2}{c|}{1.29e-3\:\quad   -----}                                                      & \multicolumn{3}{c|}{7.05e-3\:\quad   -----}                                                      & 14.42s   \\ \cline{2-16} 
& 0.4443                                  & 9124                       & \multicolumn{2}{c|}{5.67e-3\quad  3.15}                                                      & \multicolumn{2}{c|}{3.27e-3\quad  2.97}                                                      & \multicolumn{2}{c|}{13.23s}    & 16161  & \multicolumn{2}{c|}{7.44e-5\quad  4.12}                                                      & \multicolumn{3}{c|}{8.90e-4\quad  2.99}                                                      & 70.15s  \\ \cline{2-16} 
& 0.2221                                  & 35844                      & \multicolumn{2}{c|}{3.75e-4\quad  3.92}                                                      & \multicolumn{2}{c|}{3.48e-4\quad  3.23}                                                      & \multicolumn{2}{c|}{55.46s}   & 64321  & \multicolumn{2}{c|}{4.57e-6\quad  4.02}                                                      & \multicolumn{3}{c|}{1.11e-4\quad  3.00}                                                      & 360.00s \\ \cline{2-16} 
& 0.1777                                  & 55804                      & \multicolumn{2}{c|}{1.52-4\quad  4.04}                                                      & \multicolumn{2}{c|}{1.68e-4\quad  3.28}                                                      & \multicolumn{2}{c|}{88.89s}   & 100401 & \multicolumn{2}{c|}{1.88e-6\quad  3.99}                                                      & \multicolumn{3}{c|}{5.71e-5\quad  3.00}                                                      & 622.18s \\ \hline
\end{tabular}}
\caption{{Taylor--Green vortex: Velocity (at $t = 1.0s$) and  pressure(at $t=0.995s$) of  $H^1$ and upwind $H(\text{div})$ schemes, $\nu=0.01$}}
\label{Taylor_H1div}
\end{center}
\end{table}
\begin{table}
\begin{center}
\resizebox{\columnwidth}{!}{
\begin{tabular}{|c|c|c|c|c|c|c|c|c|c|c|c|c|c|c|}
\hline
\multicolumn{1}{|c|}{\multirow{2}{*}{$k$}} & \multicolumn{1}{c|}{\multirow{2}{*}{$h_{\text{max}}$}} & \multicolumn{1}{c|}{\multirow{2}{*}{d.o.f}} & \multicolumn{6}{c|}{DG-N}                                                                                                                                                                                          & \multicolumn{6}{c|}{DG-C}                                                                                                                                                                                      \\ \cline{4-15} 
\multicolumn{1}{|c|}{}                   & \multicolumn{1}{c|}{}                   & \multicolumn{1}{c|}{}                       & \multicolumn{2}{c|}{\begin{tabular}[c]{@{}c@{}}$\|\ubold-\ubold_h\|_{L^{2}(\Omega)}$\\ error\quad   order\end{tabular}} & \multicolumn{2}{c|}{\begin{tabular}[c]{@{}c@{}}$\|p-p_h\|_{L^{2}(\Omega)}$\\ error\quad   order\end{tabular}} & \multicolumn{2}{c|}{{{Runtime}}} & \multicolumn{2}{c|}{\begin{tabular}[c]{@{}c@{}}$\|\ubold-\ubold_h\|_{L^{2}(\Omega)}$\\ error  \quad order\end{tabular}} & \multicolumn{2}{c|}{\begin{tabular}[c]{@{}c@{}}$\|p-p_h\|_{L^{2}(\Omega)}$\\ error  \quad order\end{tabular}} & \multicolumn{2}{l|}{{Runtime}} \\ \hline
\multirow{4}{*}{0}                       & 0.8886                                  &        1401                                     & \multicolumn{2}{l|}{2.35e-1 \quad  ---}                                                                   & \multicolumn{2}{l|}{4.55e-1\quad ---}                                                                   & \multicolumn{2}{c|}{2.30s}        & \multicolumn{2}{l|}{2.27e-1\quad ---}                                                                   & \multicolumn{2}{l|}{4.51e-1\quad ---}                                                                   & \multicolumn{2}{c|}{1.69s}        \\ \cline{2-15} 
& 0.4443                                  &           5601                                  & \multicolumn{2}{l|}{5.44e-2\quad 2.11}                                                                   & \multicolumn{2}{l|}{2.26e-1\quad 1.01}                                                                   & \multicolumn{2}{c|}{5.73s}        & \multicolumn{2}{l|}{5.28e-2\quad 2.10}                                                                   & \multicolumn{2}{l|}{2.26e-1\quad 1.00}                                                                   & \multicolumn{2}{c|}{5.64s}        \\ \cline{2-15} 
& 0.2221                                  &    22401                                         & \multicolumn{2}{l|}{1.26e-2\quad 2.11}                                                                   & \multicolumn{2}{l|}{1.12e-1\quad 1.01}                                                                   & \multicolumn{2}{c|}{29.65s}        & \multicolumn{2}{l|}{1.24e-2\quad 2.10}                                                                   & \multicolumn{2}{l|}{1.13e-1\quad 1.00}                                                                   & \multicolumn{2}{c|}{22.50s}        \\ \cline{2-15} 
& 0.1777                                  &     35001                                        & \multicolumn{2}{l|}{7.92e-3\quad 2.08}                                                                   & \multicolumn{2}{l|}{8.97e-2\quad 1.00}                                                                   & \multicolumn{2}{c|}{54.50s}        & \multicolumn{2}{l|}{7.78e-3\quad 2.07}                                                                   & \multicolumn{2}{l|}{8.99e-2\quad 1.00}                                                                   & \multicolumn{2}{c|}{38.00s}        \\ \hline
\multirow{4}{*}{1}                       & 0.8886                                  &       3001                                      & \multicolumn{2}{l|}{2.07e-2\quad ---}                                                                   & \multicolumn{2}{l|}{6.80e-2\quad ---}                                                                   & \multicolumn{2}{c|}{4.05s}        & \multicolumn{2}{l|}{2.00e-2\quad ---}                                                                   & \multicolumn{2}{l|}{8.68e-2\quad ---}                                                                   & \multicolumn{2}{c|}{3.78s}        \\ \cline{2-15} 
& 0.4443                                  &       12001                                      & \multicolumn{2}{l|}{2.54e-3\quad 3.03}                                                                   & \multicolumn{2}{l|}{1.72e-2\quad 1.99}                                                                   & \multicolumn{2}{c|}{17.32s}        & \multicolumn{2}{l|}{2.42e-3\quad 3.04}                                                                   & \multicolumn{2}{l|}{2.23e-2\quad 1.96}                                                                   & \multicolumn{2}{c|}{15.62s}        \\ \cline{2-15} 
& 0.2221                                  &      48001                                       & \multicolumn{2}{l|}{3.03e-4\quad 3.07}                                                                   & \multicolumn{2}{l|}{4.31e-3\quad2.00}                                                                   & \multicolumn{2}{c|}{98.24s}        & \multicolumn{2}{l|}{2.83e-4\quad 3.10}                                                                   & \multicolumn{2}{l|}{5.60e-3\quad 1.99}                                                                   & \multicolumn{2}{c|}{87.66s}        \\ \cline{2-15} 
& 0.1777                                  &           75001                                  & \multicolumn{2}{l|}{1.53e-4\quad 3.06}                                                                   & \multicolumn{2}{l|}{2.76e-3\quad 2.00}                                                                   & \multicolumn{2}{c|}{175.95s}        & \multicolumn{2}{l|}{1.42e-4\quad 3.08}                                                                   & \multicolumn{2}{l|}{3.58e-3\quad 2.00}                                                                   & \multicolumn{2}{c|}{158.56s}        \\ \hline
\multirow{4}{*}{2}                       & 0.8886                                  &    5201                                         & \multicolumn{2}{l|}{1.44e-3\quad ---}                                                                   & \multicolumn{2}{l|}{7.04e-3 \quad ---}                                                                   & \multicolumn{2}{c|}{9.93s}        & \multicolumn{2}{l|}{1.37e-3\quad ---}                                                                   & \multicolumn{2}{l|}{8.00e-3\quad ---}                                                                   & \multicolumn{2}{c|}{9.38s}        \\ \cline{2-15} 
& 0.4443                                  &       20801                                      & \multicolumn{2}{l|}{8.14e-5\quad 4.15}                                                                   & \multicolumn{2}{l|}{8.90e-4\quad 2.99}                                                                   & \multicolumn{2}{c|}{55.93s}        & \multicolumn{2}{l|}{7.80e-5\quad 4.14}                                                                   & \multicolumn{2}{l|}{9.72e-4\quad3.04}                                                                   & \multicolumn{2}{c|}{52.93s}        \\ \cline{2-15} 
& 0.2221                                  &             83201                                & \multicolumn{2}{l|}{4.90e-6\quad 4.05}                                                                   & \multicolumn{2}{l|}{1.11e-4\quad 3.00}                                                                     & \multicolumn{2}{c|}{316.97s}        & \multicolumn{2}{l|}{4.65e-6\quad 4.07}                                                                   & \multicolumn{2}{l|}{1.20e-4\quad 3.01}                                                                   & \multicolumn{2}{c|}{308.49s}        \\ \cline{2-15} 
& 0.1777                                  &          130001                                   & \multicolumn{2}{l|}{2.00e-6\quad 4.02}                                                                   & \multicolumn{2}{l|}{5.71e-5\quad 3.00}                                                                   & \multicolumn{2}{c|}{569.20s}        & \multicolumn{2}{l|}{1.90e-6\quad 4.01}                                                                   & \multicolumn{2}{l|}{6.15e-5\quad 3.01}                                                                   & \multicolumn{2}{c|}{562.17s}        \\ \hline
\end{tabular}}
\caption{Taylor--Green vortex: Comparison of DG-N and DG-C schemes for the velocity at $t=1.0s$ and pressure at $t = 0.995s$ when $\nu=0.01$ and $\gamma=10$. }
\label{table_unsteady_10}
\end{center}
\end{table}

\begin{table}
\begin{center}
\resizebox{\columnwidth}{!}{
\begin{tabular}{|c|c|c|c|c|c|c|c|c|c|c|c|c|c|c|}
\hline
\multicolumn{1}{|c|}{\multirow{2}{*}{$k$}} & \multicolumn{1}{c|}{\multirow{2}{*}{$h_{\text{max}}$}} & \multicolumn{1}{c|}{\multirow{2}{*}{d.o.f}} & \multicolumn{6}{c|}{DG-N}                                                                                                                                                                                          & \multicolumn{6}{c|}{DG-C}                                                                                                                                                                                      \\ \cline{4-15} 
\multicolumn{1}{|c|}{}                   & \multicolumn{1}{c|}{}                   & \multicolumn{1}{c|}{}                       & \multicolumn{2}{c|}{\begin{tabular}[c]{@{}c@{}}$\|\ubold-\ubold_h\|_{L^{2}(\Omega)}$\\ error\quad   order\end{tabular}} & \multicolumn{2}{c|}{\begin{tabular}[c]{@{}c@{}}$\|p-p_h\|_{L^{2}(\Omega)}$\\ error\quad   order\end{tabular}} & \multicolumn{2}{c|}{{{Runtime}}} & \multicolumn{2}{c|}{\begin{tabular}[c]{@{}c@{}}$\|\ubold-\ubold_h\|_{L^{2}(\Omega)}$\\ error  \quad order\end{tabular}} & \multicolumn{2}{c|}{\begin{tabular}[c]{@{}c@{}}$\|p-p_h\|_{L^{2}(\Omega)}$\\ error  \quad order\end{tabular}} & \multicolumn{2}{l|}{{{Runtime}}} \\ \hline
\multirow{4}{*}{0}                       & 0.8886                                  &        1401                                     & \multicolumn{2}{l|}{8.11e-1 \quad  ---}                                                                   & \multicolumn{2}{l|}{5.34e-1\quad ---}                                                                   & \multicolumn{2}{c|}{2.99s}        & \multicolumn{2}{l|}{8.28e-1\quad ---}                                                                   & \multicolumn{2}{l|}{5.17e-1\quad ---}                                                                   & \multicolumn{2}{c|}{2.57s}        \\ \cline{2-15} 
& 0.4443                                  &           5601                                  & \multicolumn{2}{l|}{3.04e-2\quad 1.42}                                                                   & \multicolumn{2}{l|}{2.55e-1\quad 1.07}                                                                   & \multicolumn{2}{c|}{9.36s}        & \multicolumn{2}{l|}{3.05e-1\quad 1.44}                                                                   & \multicolumn{2}{l|}{2.52e-1\quad 1.04}                                                                   & \multicolumn{2}{c|}{5.74s}        \\ \cline{2-15} 
& 0.2221                                  &    22401                                         & \multicolumn{2}{l|}{9.92e-2\quad 1.62}                                                                   & \multicolumn{2}{l|}{1.21e-1\quad 1.07}                                                                   & \multicolumn{2}{c|}{28.35s}        & \multicolumn{2}{l|}{1.04e-1\quad 1.55}                                                                   & \multicolumn{2}{l|}{1.21e-1\quad 1.06}                                                                   & \multicolumn{2}{c|}{21.91s}        \\ \cline{2-15} 
& 0.1777                                  &     35001                                        & \multicolumn{2}{l|}{6.89e-2\quad 1.63}                                                                   & \multicolumn{2}{l|}{9.58e-2\quad 1.06}                                                                   & \multicolumn{2}{c|}{42.47s}        & \multicolumn{2}{l|}{7.26e-2\quad 1.61}                                                                   & \multicolumn{2}{l|}{9.58e-2\quad 1.06}                                                                   & \multicolumn{2}{c|}{36.60s}        \\ \hline
\multirow{4}{*}{1}                       & 0.8886                                  &       3001                                      & \multicolumn{2}{l|}{1.66e-1\quad ---}                                                                   & \multicolumn{2}{l|}{8.80e-2\quad ---}                                                                   & \multicolumn{2}{c|}{6.98s}        & \multicolumn{2}{l|}{1.50e-1\quad ---}                                                                   & \multicolumn{2}{l|}{7.90e-2\quad ---}                                                                   & \multicolumn{2}{c|}{4.14s}        \\ \cline{2-15} 
& 0.4443                                  &       12001                                      & \multicolumn{2}{l|}{2.60e-2\quad 2.68}                                                                   & \multicolumn{2}{l|}{1.98e-2\quad 2.15}                                                                   & \multicolumn{2}{c|}{20.10s}        & \multicolumn{2}{l|}{2.30e-2\quad 2.71}                                                                   & \multicolumn{2}{l|}{1.86e-2\quad 2.09}                                                                   & \multicolumn{2}{c|}{15.87s}        \\ \cline{2-15} 
& 0.2221                                  &      48001                                       & \multicolumn{2}{l|}{3.18e-3\quad 3.03}                                                                   & \multicolumn{2}{l|}{4.55e-3\quad 2.12}                                                                   & \multicolumn{2}{c|}{92.16s}        & \multicolumn{2}{l|}{3.24e-3\quad 2.83}                                                                   & \multicolumn{2}{l|}{4.45e-3\quad 2.06}                                                                   & \multicolumn{2}{c|}{88.11s}        \\ \cline{2-15} 
& 0.1777                                  &           75001                                  & \multicolumn{2}{l|}{1.62e-3\quad 3.02}                                                                   & \multicolumn{2}{l|}{2.87e-3\quad 2.06}                                                                   & \multicolumn{2}{c|}{163.94s}        & \multicolumn{2}{l|}{1.70e-3\quad 2.89}                                                                   & \multicolumn{2}{l|}{2.83e-3\quad 2.03}                                                                   & \multicolumn{2}{c|}{158.39s}        \\ \hline
\multirow{4}{*}{2}                       & 0.8886                                  &    5201                                         & \multicolumn{2}{l|}{6.99e-3\quad ---}                                                                   & \multicolumn{2}{l|}{7.47e-3 \quad ---}                                                                   & \multicolumn{2}{c|}{11.95s}        & \multicolumn{2}{l|}{5.95e-3\quad ---}                                                                   & \multicolumn{2}{l|}{7.24e-3\quad ---}                                                                   & \multicolumn{2}{c|}{9.57s}        \\ \cline{2-15} 
& 0.4443                                  &       20801                                      & \multicolumn{2}{l|}{3.30e-4\quad 4.40}                                                                   & \multicolumn{2}{l|}{9.01e-4\quad 3.05}                                                                   & \multicolumn{2}{c|}{55.58s}        & \multicolumn{2}{l|}{3.19e-4\quad 4.22}                                                                   & \multicolumn{2}{l|}{8.98e-4\quad 3.01}                                                                   & \multicolumn{2}{c|}{56.46}        \\ \cline{2-15} 
& 0.2221                                  &             83201                                & \multicolumn{2}{l|}{1.90e-5\quad 4.12}                                                                   & \multicolumn{2}{l|}{1.12e-4\quad 3.01}                                                                     & \multicolumn{2}{c|}{315.91s}        & \multicolumn{2}{l|}{1.99e-5\quad 4.00}                                                                   & \multicolumn{2}{l|}{1.12e-4\quad 3.00}                                                                   & \multicolumn{2}{c|}{309.51s}        \\ \cline{2-15} 
& 0.1777                                  &          130001                                   & \multicolumn{2}{l|}{7.78e-6\quad 4.00}                                                                   & \multicolumn{2}{l|}{5.74e-5\quad 3.00}                                                                   & \multicolumn{2}{c|}{570.54s}        & \multicolumn{2}{l|}{8.27e-6\quad 3.94}                                                                   & \multicolumn{2}{l|}{5.74e-5\quad 3.00}                                                                   & \multicolumn{2}{c|}{562.16s}        \\ \hline
\end{tabular}}
\caption{Taylor--Green vortex: Comparison of DG-N and DG-C schemes for the velocity at $t = 1.0s$ and pressure at $t=0.995s$, when $\nu=0.01$, $\gamma=0$.}
\label{table_unsteady_0}
\end{center}
\end{table}
\subsection{Kovasznay Flow}\label{K_flow}
In this experiment, we consider the steady Kovasznay flow \cite{Ern2012} with the analytical solutions given by 
\begin{align*}
\label{kovasznay}
    &\ubold(t,\xbold)=\bigg(1-\text{e}^{\lambda x_1}\cos(2\pi x_2),\frac{\lambda}{2\pi}\text{e}^{\lambda x_1}\sin(2\pi x_2)\bigg), \\
    &p(t,\xbold)=-\frac{1}{2}\text{e}^{2\lambda  x_1}-\frac{1}{8\lambda}\bigg(\text{e}^{- \lambda}-\text{e}^{3\lambda}\bigg)
\end{align*}
with $\lambda=\frac{1}{2\nu}-(\frac{1}{4\nu^2}+4\pi^{2})^{\frac{1}{2}}$ and the simulation domain  $\Omega:=[-0.5,0]\times[1.5,2]$.  All schemes with $\nu=0.025$ are tested on uniform meshes with mesh sizes $h_{\text{max}}\in\{ 0.1768, 0.0884, 0.0442, 0.0354\}$. Other parameters are identical to those given in Experiment \ref{Tay_Ex}.  

The overall performance of $H^1$, $H(\text{div})$, and DG schemes are similar to those in Experiment 4.1, see Tables  \ref{K_H1div}, \ref{table_steady_10}, and \ref{table_steady_0}.
The $H^{1}$ scheme is the best among all schemes when $k=2$. Figure~\ref{DG_velocity_Kovasznay_Flow} shows that the approximation becomes better when increasing polynomial degree and/or decreasing the mesh size for our DG scheme.

From Tables \ref{table_steady_10} and \ref{table_steady_0}, we observe that both DG schemes achieve the expected convergence rates for velocity and pressure when $\gamma\in\{0,10\}$. The errors and running time of DG-N are slightly smaller than DG-C when $\gamma=0$. In contrast to the Taylor--Green vortex, the running time of both DG schemes for the stationary Kovasznay flow are similar because there is no dynamic update at each time step. We also observe that the runtime tends to decrease when $\gamma$ increases from $0$ to $10$, especially for DG-C. An interesting observation is that there is a trend of increasing of errors in both velocity and pressure when we increase $\gamma$ from $0$ to $10$, which indicates that the penalty term \eqref{Penalty_expression} may fail to reduce the errors in some cases when the convective term appears. It is shown in \cite{akbas2018analogue} that the solution will converge to BDM solution if $\gamma\rightarrow\infty$ for the Stokes problem, which indicates  a decreasing of absolute errors when increasing $\gamma$ (at least for Stokes flow). Due to nonlinearity of Naiver--Stokes equations, theoretical analysis on the optimality of $\gamma$ seems not available. Hence we will investigate this influence numerically in the next experiment. 


\begin{table}
\begin{center}
\resizebox{\columnwidth}{!}{
\begin{tabular}{|l|l|c|l|l|l|l|c|l|c|l|l|l|l|l|c|}
\hline
\multicolumn{1}{|c|}{\multirow{2}{*}{$k$}} & \multicolumn{1}{c|}{\multirow{2}{*}{$h_{\text{max}}$}} & \multicolumn{7}{c|}{$H^{1}$}                                                                                                                                                                                                                        & \multicolumn{7}{c|}{$H(\text{div})$}                                                                                                                                                                           \\ \cline{3-16} 
\multicolumn{1}{|c|}{}                   & \multicolumn{1}{c|}{}                   & \multicolumn{1}{l|}{d.o.f} & \multicolumn{2}{c|}{\begin{tabular}[c]{@{}c@{}}$\|\ubold-\ubold_h\|_{L^{2}(\Omega)}$\\ error   order\end{tabular}} & \multicolumn{2}{c|}{\begin{tabular}[c]{@{}c@{}}$\|p-p_h\|_{L^{2}(\Omega)}$\\ error   order\end{tabular}} & \multicolumn{2}{c|}{Runtime}  & d.o.f  & \multicolumn{2}{c|}{\begin{tabular}[c]{@{}c@{}}$\|\ubold-\ubold_h\|_{L^{2}(\Omega)}$\\ error   order\end{tabular}} & \multicolumn{3}{c|}{\begin{tabular}[c]{@{}c@{}}$\|p-p_h\|_{L^{2}(\Omega)}$\\ error   order\end{tabular}} & Runtime  \\ \hline
\multirow{4}{*}{0}                       & 0.1768                                  & N/A                        & \multicolumn{2}{c|}{N/A}                                                                & \multicolumn{2}{c|}{N/A}                                                                & \multicolumn{2}{c|}{N/A}      & 2113   & \multicolumn{2}{c|}{4.45e-2\:\quad   -----}                                                      & \multicolumn{3}{c|}{6.73e-2\:\quad   -----}                                                      & 1.49e-1s \\ \cline{2-16} 
& 0.0884                                  & N/A                        & \multicolumn{2}{c|}{N/A}                                                                & \multicolumn{2}{c|}{N/A}                                                                & \multicolumn{2}{c|}{N/A}      & 8321   & \multicolumn{2}{c|}{1.00e-2\quad  2.15}                                                      & \multicolumn{3}{c|}{3.00e-2\quad  1.17}                                                      & 6.01e-1s \\ \cline{2-16} 
& 0.0442                                  & N/A                        & \multicolumn{2}{c|}{N/A}                                                                & \multicolumn{2}{c|}{N/A}                                                                & \multicolumn{2}{c|}{N/A}      & 33025  & \multicolumn{2}{c|}{2.46e-3\quad  2.03}                                                      & \multicolumn{3}{c|}{1.43e-2\quad  1.06}                                                      & 3.06s    \\ \cline{2-16} 
& 0.0354                                  & N/A                        & \multicolumn{2}{c|}{N/A}                                                                & \multicolumn{2}{c|}{N/A}                                                                & \multicolumn{2}{c|}{N/A}      & 51521  & \multicolumn{2}{c|}{1.57e-3\quad  1.99}                                                      & \multicolumn{3}{c|}{1.14e-2\quad  1.03}                                                      & 5.05s    \\ \hline
\multirow{4}{*}{1}                       & 0.1768                                  & 2468                       & \multicolumn{2}{c|}{3.37e-3\:\quad   -----}                                                      & \multicolumn{2}{c|}{2.26e-3\:\quad   -----}                                                      & \multicolumn{2}{c|}{1.13e-1s} & 5473   & \multicolumn{2}{c|}{2.69e-3\:\quad   -----}                                                      & \multicolumn{3}{c|}{3.70e-3\:\quad   -----}                                                      & 5.79e-1s \\ \cline{2-16} 
& 0.0884                                  & 9540                       & \multicolumn{2}{c|}{4.17e-4\quad  3.01}                                                      & \multicolumn{2}{c|}{5.17e-4\quad  2.13}                                                      & \multicolumn{2}{c|}{4.61e-1s} & 21697  & \multicolumn{2}{c|}{3.31e-4\quad  3.02}                                                      & \multicolumn{3}{c|}{8.11e-4\quad  2.19}                                                      & 2.93s    \\ \cline{2-16} 
& 0.0442                                  & 37508                      & \multicolumn{2}{c|}{5.20e-5\quad  3.00}                                                      & \multicolumn{2}{c|}{1.28e-4\quad  2.01}                                                      & \multicolumn{2}{c|}{2.22s}    & 86401  & \multicolumn{2}{c|}{4.13e-5\quad  3.00}                                                      & \multicolumn{3}{c|}{1.87e-4\quad  2.12}                                                      & 15.56s   \\ \cline{2-16} 
& 0.0354                                  & 58404                      & \multicolumn{2}{c|}{2.66e-5\quad  3.00}                                                      & \multicolumn{2}{c|}{8.18e-5\quad  2.00}                                                      & \multicolumn{2}{c|}{3.54s}    & 134881 & \multicolumn{2}{c|}{2.11e-5\quad  3.00}                                                      & \multicolumn{3}{c|}{1.17e-4\quad  2.08}                                                      & 26.68s   \\ \hline
\multirow{4}{*}{2}                       & 0.1768                                  & 5892                       & \multicolumn{2}{c|}{1.61e-4\:\quad   -----}                                                      & \multicolumn{2}{c|}{1.23e-4\:\quad   -----}                                                      & \multicolumn{2}{c|}{3.66e-1s} & 10369  & \multicolumn{2}{c|}{1.68e-4\:\quad   -----}                                                      & \multicolumn{3}{c|}{3.06e-4\:\quad   -----}                                                      & 2.00s    \\ \cline{2-16} 
& 0.0884                                  & 23044                      & \multicolumn{2}{c|}{1.01e-5\quad  4.00}                                                      & \multicolumn{2}{c|}{1.15e-5\quad  3.42}                                                      & \multicolumn{2}{c|}{1.58s}    & 41217  & \multicolumn{2}{c|}{1.08e-5\quad  3.96}                                                      & \multicolumn{3}{c|}{3.12e-5\quad  3.29}                                                      & 10.32s    \\ \cline{2-16} 
& 0.0442                                  & 91140                      & \multicolumn{2}{c|}{6.32e-7\quad  4.00}                                                      & \multicolumn{2}{c|}{1.26e-6\quad  3.19}                                                      & \multicolumn{2}{c|}{7.24s}    & 164353 & \multicolumn{2}{c|}{6.87e-7\quad  3.98}                                                      & \multicolumn{3}{c|}{3.50e-6\quad  3.16}                                                      & 58.00s   \\ \cline{2-16} 
& 0.0354                                  & 142084                     & \multicolumn{2}{c|}{2.59e-7\quad  4.00}                                                      & \multicolumn{2}{c|}{6.34e-7\quad  3.09}                                                      & \multicolumn{2}{c|}{11.81s}    & 256641 & \multicolumn{2}{c|}{2.82e-7\quad  3.99}                                                      & \multicolumn{3}{c|}{1.75e-6\quad  3.10}                                                      & 105.40s   \\ \hline
\end{tabular}}
\caption{{{Kovasznay flow: Velocity and pressure of $H^1$ and upwind $H(\text{div})$ schemes, $\nu=0.025$.}}}
\label{K_H1div}
\end{center}
\end{table}
\begin{table}
\begin{center}
\resizebox{\columnwidth}{!}{
\begin{tabular}{|c|c|c|c|c|c|c|c|c|c|c|c|c|c|c|}
\hline
\multicolumn{1}{|c|}{\multirow{2}{*}{$k$}} & \multicolumn{1}{c|}{\multirow{2}{*}{$h_{\text{max}}$}} & \multicolumn{1}{c|}{\multirow{2}{*}{d.o.f}} & \multicolumn{6}{c|}{DG-N}                                                                                                                                                                                          & \multicolumn{6}{c|}{DG-C}                                                                                                                                                                                      \\ \cline{4-15} 
\multicolumn{1}{|c|}{}                   & \multicolumn{1}{c|}{}                   & \multicolumn{1}{c|}{}                       & \multicolumn{2}{c|}{\begin{tabular}[c]{@{}c@{}}$\|\ubold-\ubold_h\|_{L^{2}(\Omega)}$\\ error\quad   order\end{tabular}} & \multicolumn{2}{c|}{\begin{tabular}[c]{@{}c@{}}$\|p-p_h\|_{L^{2}(\Omega)}$\\ error\quad   order\end{tabular}} & \multicolumn{2}{c|}{{Runtime}} & \multicolumn{2}{c|}{\begin{tabular}[c]{@{}c@{}}$\|\ubold-\ubold_h\|_{L^{2}(\Omega)}$\\ error  \quad order\end{tabular}} & \multicolumn{2}{c|}{\begin{tabular}[c]{@{}c@{}}$\|p-p_h\|_{L^{2}(\Omega)}$\\ error  \quad order\end{tabular}} & \multicolumn{2}{l|}{{Runtime}} \\ \hline
\multirow{4}{*}{0}                       & 0.1768                                  &        3585                                     & \multicolumn{2}{l|}{3.77e-2 \quad  ---}                                                                   & \multicolumn{2}{l|}{5.89e-2\quad ---}                                                                   & \multicolumn{2}{c|}{1.27e-1s}        & \multicolumn{2}{l|}{4.91e-2\quad ---}                                                                   & \multicolumn{2}{l|}{6.52e-2\quad ---}                                                                   & \multicolumn{2}{c|}{1.25e-1s}        \\ \cline{2-15} 
& 0.0884                                  &           14337                                  & \multicolumn{2}{l|}{9.62e-3\quad 1.97}                                                                   & \multicolumn{2}{l|}{2.86e-2\quad 1.04}                                                                   & \multicolumn{2}{c|}{5.55e-1s}        & \multicolumn{2}{l|}{1.04e-2\quad 2.24}                                                                   & \multicolumn{2}{l|}{3.01e-2\quad 1.11}                                                                   & \multicolumn{2}{c|}{5.81e-1s}        \\ \cline{2-15} 
& 0.0442                                  &    57345                                         & \multicolumn{2}{l|}{2.44e-3\quad 1.98}                                                                   & \multicolumn{2}{l|}{1.41e-2\quad 1.02}                                                                   & \multicolumn{2}{c|}{2.91s}        & \multicolumn{2}{l|}{2.46e-3\quad 2.08}                                                                   & \multicolumn{2}{l|}{1.47e-2\quad 1.04}                                                                   & \multicolumn{2}{c|}{2.96s}        \\ \cline{2-15} 
& 0.0354                                  &     89601                                        & \multicolumn{2}{l|}{1.57e-3\quad 1.98}                                                                   & \multicolumn{2}{l|}{1.13e-2\quad 1.01}                                                                   & \multicolumn{2}{c|}{4.82s}        & \multicolumn{2}{l|}{1.57e-3\quad 2.02}                                                                   & \multicolumn{2}{l|}{1.17e-2\quad 1.02}                                                                   & \multicolumn{2}{c|}{4.90s}        \\ \hline
\multirow{4}{*}{1}                       & 0.1768                                  &       7681                                      & \multicolumn{2}{l|}{2.59e-3\quad ---}                                                                   & \multicolumn{2}{l|}{2.99e-3\quad ---}                                                                   & \multicolumn{2}{c|}{4.65e-1s}        & \multicolumn{2}{l|}{2.59e-3\quad ---}                                                                   & \multicolumn{2}{l|}{2.55e-3\quad ---}                                                                   & \multicolumn{2}{c|}{4.62e-1s}        \\ \cline{2-15} 
& 0.0884                                  &       30721                                      & \multicolumn{2}{l|}{3.25e-4\quad 3.00}                                                                   & \multicolumn{2}{l|}{7.01e-4\quad 2.09}                                                                   & \multicolumn{2}{c|}{2.61s}        & \multicolumn{2}{l|}{3.23e-4\quad 3.00}                                                                   & \multicolumn{2}{l|}{5.64e-4\quad 2.17}                                                                   & \multicolumn{2}{c|}{2.63s}        \\ \cline{2-15} 
& 0.0442                                  &      122881                                       & \multicolumn{2}{l|}{4.07e-5\quad 3.00}                                                                   & \multicolumn{2}{l|}{1.69e-4\quad 2.05}                                                                   & \multicolumn{2}{c|}{14.73s}        & \multicolumn{2}{l|}{4.03e-5\quad 3.00}                                                                   & \multicolumn{2}{l|}{1.31e-4\quad 2.11}                                                                   & \multicolumn{2}{c|}{14.98s}        \\ \cline{2-15} 
& 0.0354                                  &           192001                                  & \multicolumn{2}{l|}{2.08e-5\quad 3.00}                                                                   & \multicolumn{2}{l|}{1.08e-4\quad 2.03}                                                                   & \multicolumn{2}{c|}{26.13s}        & \multicolumn{2}{l|}{2.06e-5\quad 3.00}                                                                   & \multicolumn{2}{l|}{8.24e-5\quad 2.07}                                                                   & \multicolumn{2}{c|}{26.25s}        \\ \hline
\multirow{4}{*}{2}                       & 0.1768                                  &    13313                                         & \multicolumn{2}{l|}{1.37e-4\quad ---}                                                                   & \multicolumn{2}{l|}{2.02e-4 \quad ---}                                                                   & \multicolumn{2}{c|}{1.58s}        & \multicolumn{2}{l|}{1.37e-4\quad ---}                                                                   & \multicolumn{2}{l|}{1.83e-4\quad ---}                                                                   & \multicolumn{2}{c|}{1.53s}        \\ \cline{2-15} 
& 0.0884                                  &       53249                                      & \multicolumn{2}{l|}{8.86e-6\quad 3.95}                                                                   & \multicolumn{2}{l|}{2.50e-5\quad 3.01}                                                                   & \multicolumn{2}{c|}{9.17s}        & \multicolumn{2}{l|}{8.87e-6\quad 3.95}                                                                   & \multicolumn{2}{l|}{2.36e-5\quad 2.96}                                                                   & \multicolumn{2}{c|}{9.16s}        \\ \cline{2-15} 
& 0.0442                                  &             212993                                & \multicolumn{2}{l|}{5.62e-7\quad 3.98}                                                                   & \multicolumn{2}{l|}{3.08e-6\quad 3.02}                                                                     & \multicolumn{2}{c|}{53.67s}        & \multicolumn{2}{l|}{5.63e-7\quad 3.98}                                                                   & \multicolumn{2}{l|}{2.96e-6\quad 2.99}                                                                   & \multicolumn{2}{c|}{55.35s}        \\ \cline{2-15} 
& 0.0354                                  &          332801                                   & \multicolumn{2}{l|}{2.31e-7\quad 3.99}                                                                   & \multicolumn{2}{l|}{1.57e-6\quad 3.01}                                                                   & \multicolumn{2}{c|}{101.30s}        & \multicolumn{2}{l|}{2.31e-7\quad 3.99}                                                                   & \multicolumn{2}{l|}{1.52e-6\quad 3.00}                                                                   & \multicolumn{2}{c|}{95.62s}        \\ \hline
\end{tabular}}
\caption{Kovasznay flow: Comparison of DG-N and DG-C schemes when $\nu=0.025$ and $\gamma=10$.}
\label{table_steady_10}
\end{center}
\end{table}


\begin{table}
\begin{center}
\resizebox{\columnwidth}{!}{
\begin{tabular}{|c|c|c|c|c|c|c|c|c|c|c|c|c|c|c|}
\hline
\multicolumn{1}{|c|}{\multirow{2}{*}{$k$}} & \multicolumn{1}{c|}{\multirow{2}{*}{$h_{\text{max}}$}} & \multicolumn{1}{c|}{\multirow{2}{*}{d.o.f}} & \multicolumn{6}{c|}{DG-N}                                                                                                                                                                                          & \multicolumn{6}{c|}{DG-C}                                                                                                                                                                                      \\ \cline{4-15} 
\multicolumn{1}{|c|}{}                   & \multicolumn{1}{c|}{}                   & \multicolumn{1}{c|}{}                       & \multicolumn{2}{c|}{\begin{tabular}[c]{@{}c@{}}$\|\ubold-\ubold_h\|_{L^{2}(\Omega)}$\\ error\quad   order\end{tabular}} & \multicolumn{2}{c|}{\begin{tabular}[c]{@{}c@{}}$\|p-p_h\|_{L^{2}(\Omega)}$\\ error\quad   order\end{tabular}} & \multicolumn{2}{c|}{{{Runtime}}} & \multicolumn{2}{c|}{\begin{tabular}[c]{@{}c@{}}$\|\ubold-\ubold_h\|_{L^{2}(\Omega)}$\\ error  \quad order\end{tabular}} & \multicolumn{2}{c|}{\begin{tabular}[c]{@{}c@{}}$\|p-p_h\|_{L^{2}(\Omega)}$\\ error  \quad order\end{tabular}} & \multicolumn{2}{l|}{{{Runtime}}} \\ \hline
\multirow{4}{*}{0}                       & 0.1768                                  &        3585                                     & \multicolumn{2}{l|}{4.04e-2 \quad  ---}                                                                   & \multicolumn{2}{l|}{6.05e-2\quad ---}                                                                   & \multicolumn{2}{c|}{1.62e-1s}        & \multicolumn{2}{l|}{8.00e-2\quad ---}                                                                   & \multicolumn{2}{l|}{1.08e-1\quad ---}                                                                   & \multicolumn{2}{c|}{1.28e-1s}        \\ \cline{2-15} 
& 0.0884                                  &           14337                                  & \multicolumn{2}{l|}{1.03e-2\quad 1.97}                                                                   & \multicolumn{2}{l|}{2.84e-2\quad 1.09}                                                                   & \multicolumn{2}{c|}{6.86e-1s}        & \multicolumn{2}{l|}{1.14e-2\quad 2.81}                                                                   & \multicolumn{2}{l|}{2.96e-2\quad 1.87}                                                                   & \multicolumn{2}{c|}{7.10e-1s}        \\ \cline{2-15} 
& 0.0442                                  &    57345                                         & \multicolumn{2}{l|}{2.50e-3\quad 2.05}                                                                   & \multicolumn{2}{l|}{1.38e-2\quad 1.04}                                                                   & \multicolumn{2}{c|}{3.21s}        & \multicolumn{2}{l|}{2.65e-3\quad 2.11}                                                                   & \multicolumn{2}{l|}{1.40e-2\quad 1.08}                                                                   & \multicolumn{2}{c|}{3.59s}        \\ \cline{2-15} 
& 0.0354                                  &     89601                                        & \multicolumn{2}{l|}{1.58e-3\quad 2.06}                                                                   & \multicolumn{2}{l|}{1.10e-2\quad 1.02}                                                                   & \multicolumn{2}{c|}{4.86s}        & \multicolumn{2}{l|}{1.67e-3\quad 2.06}                                                                   & \multicolumn{2}{l|}{1.12e-2\quad 1.02}                                                                   & \multicolumn{2}{c|}{5.67s}        \\ \hline
\multirow{4}{*}{1}                       & 0.1768                                  &       7681                                      & \multicolumn{2}{l|}{2.34e-3\quad ---}                                                                   & \multicolumn{2}{l|}{2.18e-3\quad ---}                                                                   & \multicolumn{2}{c|}{4.54e-1s}        & \multicolumn{2}{l|}{2.29e-3\quad ---}                                                                   & \multicolumn{2}{l|}{2.12e-3\quad ---}                                                                   & \multicolumn{2}{c|}{4.75e-1s}        \\ \cline{2-15} 
& 0.0884                                  &       30721                                      & \multicolumn{2}{l|}{2.87e-4\quad 3.03}                                                                   & \multicolumn{2}{l|}{5.01e-4\quad 2.12}                                                                   & \multicolumn{2}{c|}{2.60s}        & \multicolumn{2}{l|}{2.84e-4\quad 3.01}                                                                   & \multicolumn{2}{l|}{5.06e-4\quad 2.07}                                                                   & \multicolumn{2}{c|}{2.84s}        \\ \cline{2-15} 
& 0.0442                                  &      122881                                       & \multicolumn{2}{l|}{3.54e-5\quad 3.02}                                                                   & \multicolumn{2}{l|}{1.22e-4\quad 2.04}                                                                   & \multicolumn{2}{c|}{14.63s}        & \multicolumn{2}{l|}{3.55e-5\quad 3.00}                                                                   & \multicolumn{2}{l|}{1.25e-4\quad 2.02}                                                                   & \multicolumn{2}{c|}{15.26s}        \\ \cline{2-15} 
& 0.0354                                  &           192001                                  & \multicolumn{2}{l|}{1.81e-5\quad 3.01}                                                                   & \multicolumn{2}{l|}{7.77e-5\quad 2.02}                                                                   & \multicolumn{2}{c|}{25.49s}        & \multicolumn{2}{l|}{1.82e-5\quad 3.00}                                                                   & \multicolumn{2}{l|}{7.96e-5\quad 2.01}                                                                   & \multicolumn{2}{c|}{26.52s}        \\ \hline
\multirow{4}{*}{2}                       & 0.1768                                  &    13313                                         & \multicolumn{2}{l|}{1.21e-4\quad ---}                                                                   & \multicolumn{2}{l|}{1.42e-4 \quad ---}                                                                   & \multicolumn{2}{c|}{1.61s}        & \multicolumn{2}{l|}{1.21e-4\quad ---}                                                                   & \multicolumn{2}{l|}{1.51e-4\quad ---}                                                                   & \multicolumn{2}{c|}{1.56s}        \\ \cline{2-15} 
& 0.0884                                  &       53249                                      & \multicolumn{2}{l|}{7.75e-6\quad 3.96}                                                                   & \multicolumn{2}{l|}{1.82e-5\quad 2.97}                                                                   & \multicolumn{2}{c|}{9.39s}        & \multicolumn{2}{l|}{7.78e-6\quad 3.96}                                                                   & \multicolumn{2}{l|}{1.97e-5\quad 2.94}                                                                   & \multicolumn{2}{c|}{9.74s}        \\ \cline{2-15} 
& 0.0442                                  &             212993                                & \multicolumn{2}{l|}{4.90e-7\quad 3.98}                                                                   & \multicolumn{2}{l|}{2.28e-6\quad 2.99}                                                                     & \multicolumn{2}{c|}{54.09s}        & \multicolumn{2}{l|}{4.93e-7\quad 3.98}                                                                   & \multicolumn{2}{l|}{2.49e-6\quad 2.99}                                                                   & \multicolumn{2}{c|}{55.67s}        \\ \cline{2-15} 
& 0.0354                                  &          332801                                   & \multicolumn{2}{l|}{2.01e-7\quad 3.99}                                                                   & \multicolumn{2}{l|}{1.17e-6\quad 3.00}                                                                   & \multicolumn{2}{c|}{94.04s}        & \multicolumn{2}{l|}{2.02e-7\quad 3.99}                                                                   & \multicolumn{2}{l|}{1.27e-6\quad 3.00}                                                                   & \multicolumn{2}{c|}{94.78s}        \\ \hline
\end{tabular}}
\caption{Kovasznay flow: Comparison of DG-N and DG-C schemes when $\nu=0.025$ and $\gamma=0$.}
\label{table_steady_0}
\end{center}
\end{table}
\begin{figure*}
\centering
\begin{multicols}{2}
    \includegraphics[width=1.0\linewidth,trim={7cm 2cm 0cm 2cm},clip]{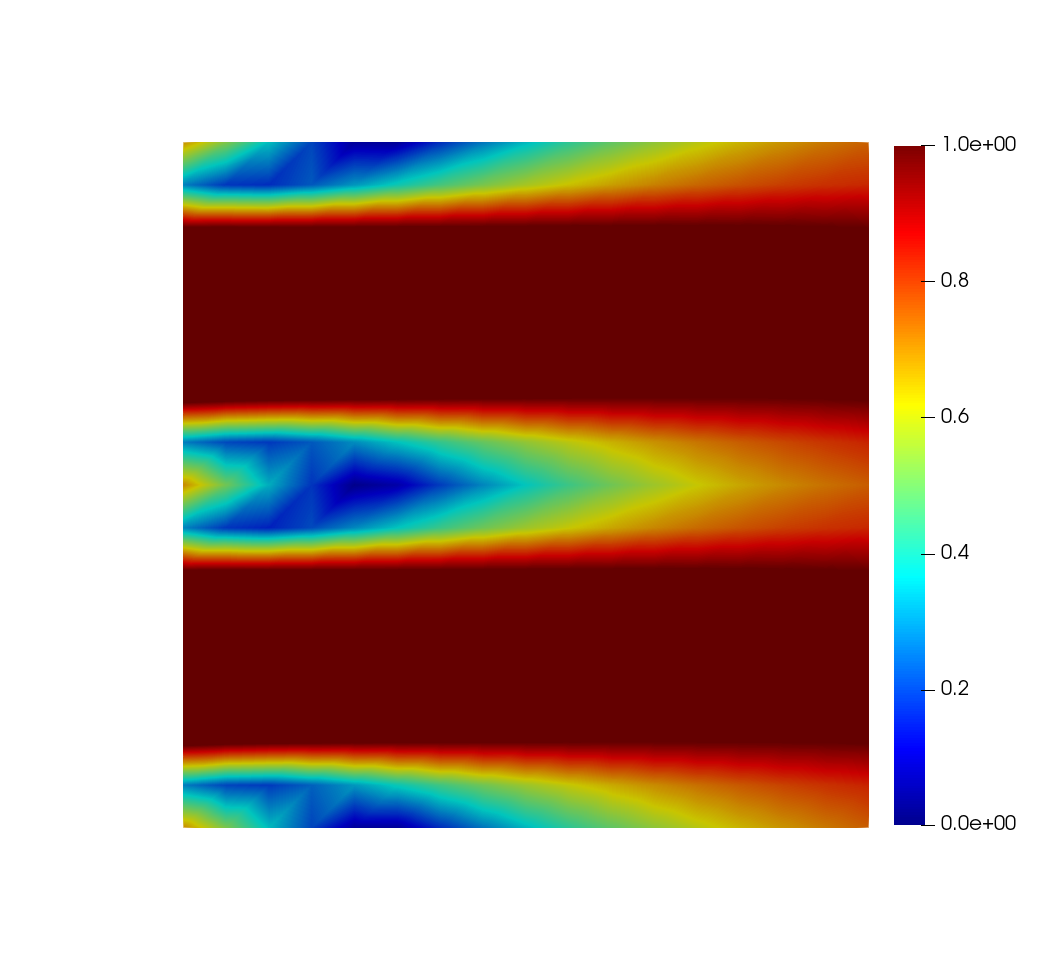}\par
    \includegraphics[width=1.0\linewidth,trim={7cm 2cm 0cm 2cm},clip]{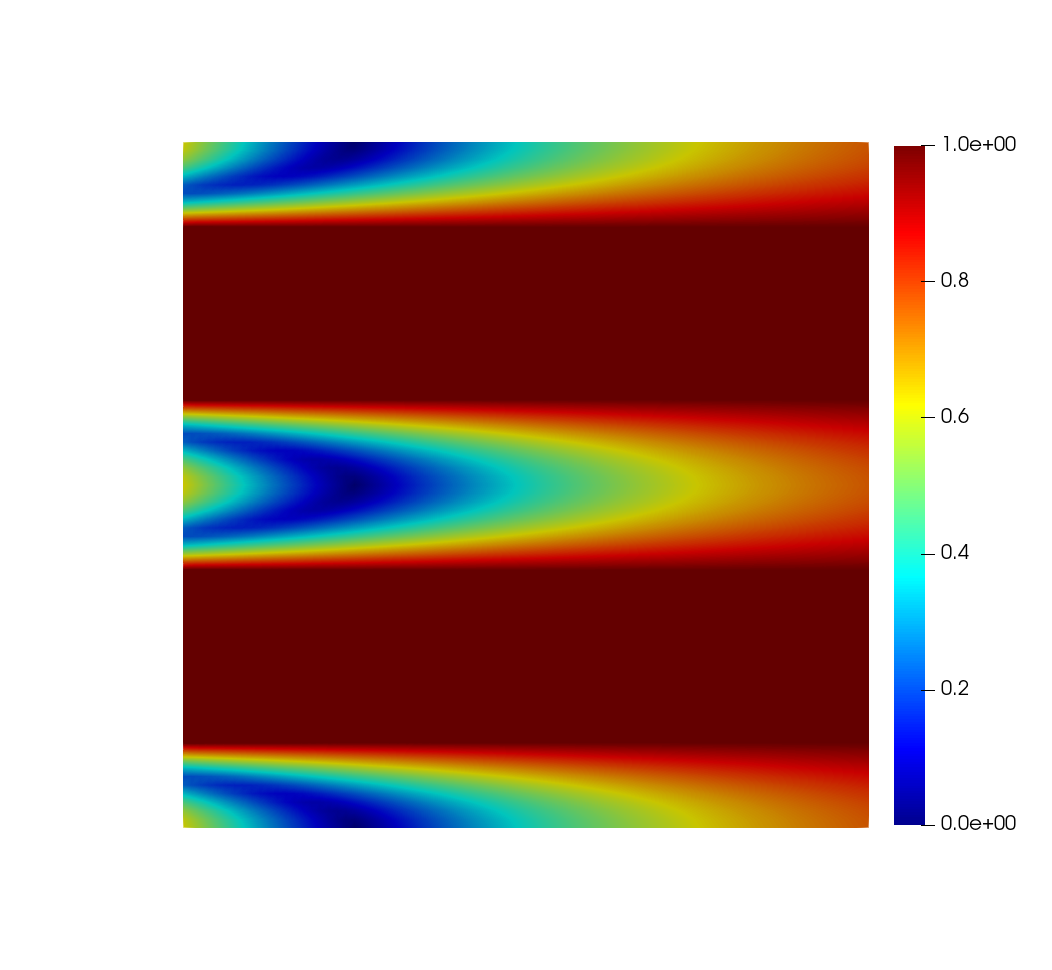}\par 
    \end{multicols}
\begin{multicols}{2}
    \includegraphics[width=1.0\linewidth,trim={7cm 2cm 0cm 2cm},clip]{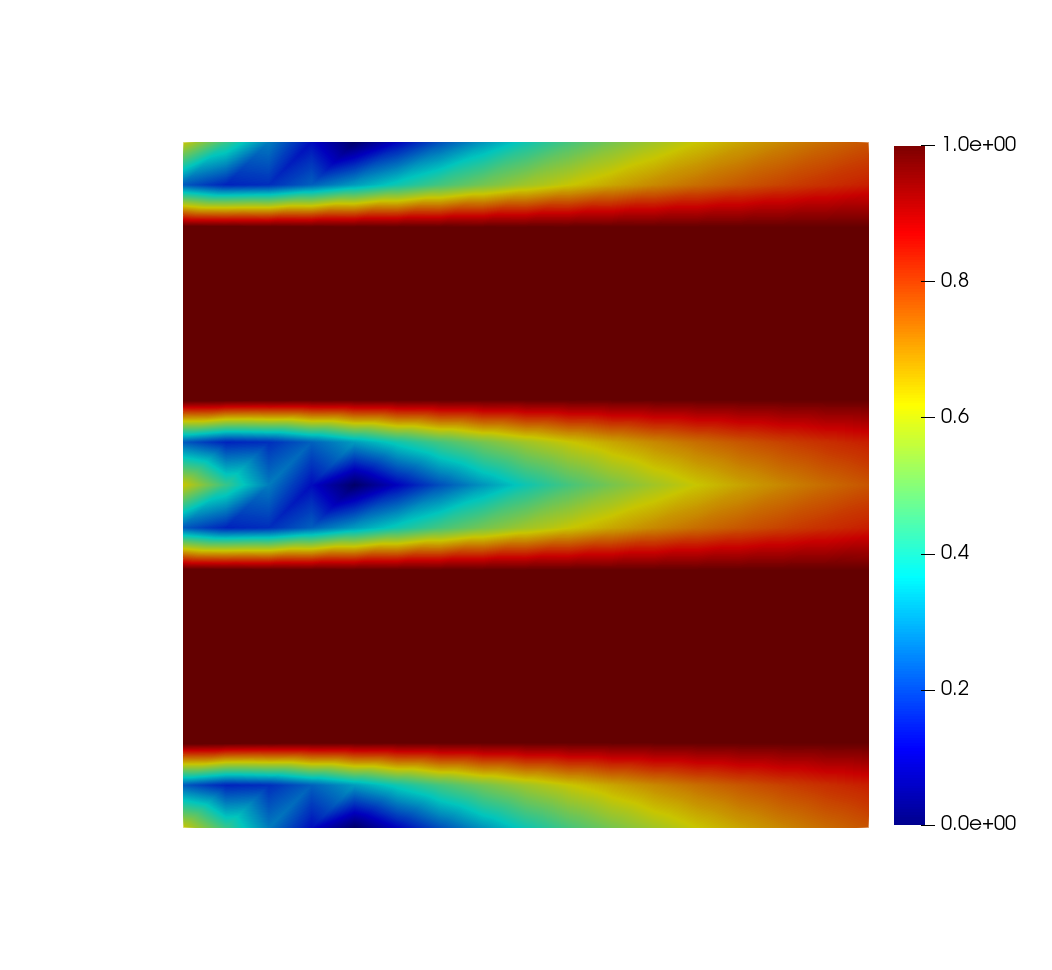}\par
    \includegraphics[width=1.0\linewidth,trim={7cm 2cm 0cm 2cm},clip]{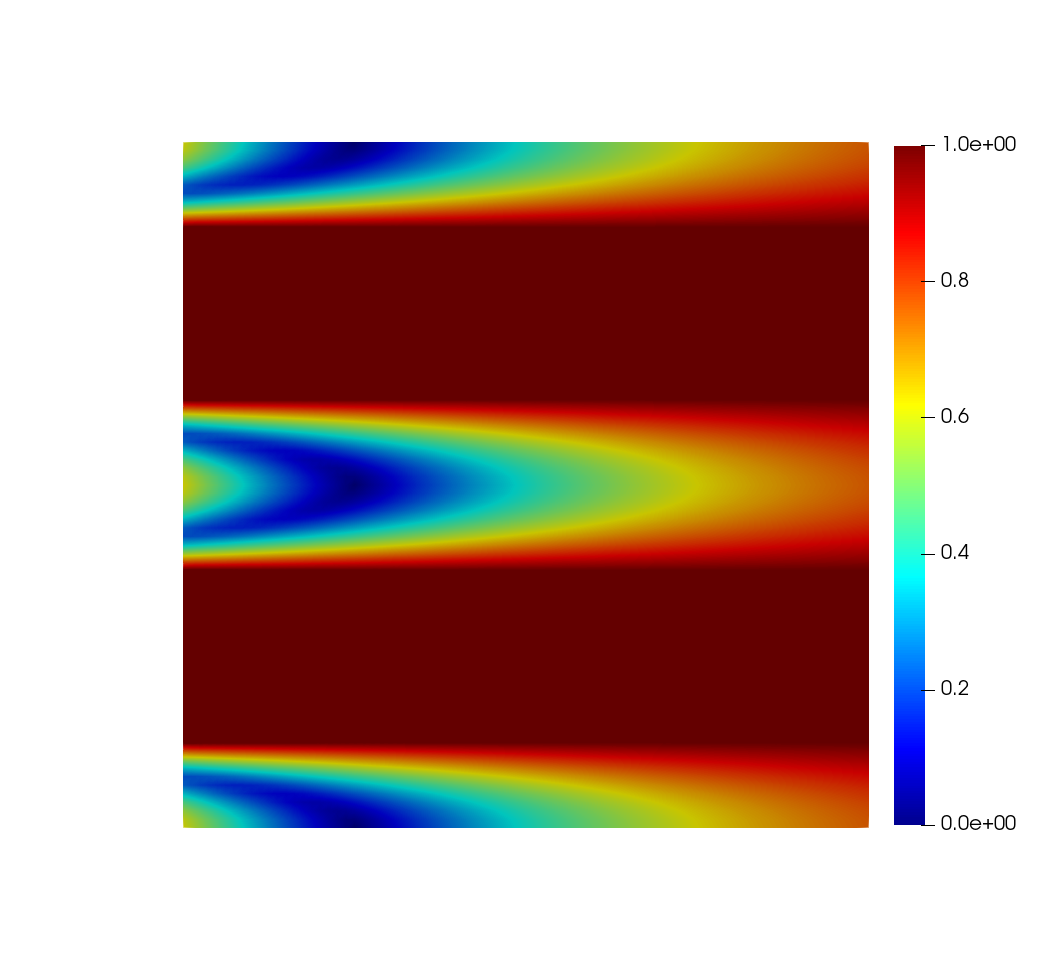}\par
\end{multicols}
\caption{Kovasznay Flow: Contours of velocity magnitude from the DG-N scheme with $h_{\text{max}} = 0.1768$ (left) and $h_{\text{max}}=0.0354$ (right) when $k=0$ (upper row) and $k=2$ (bottom row), $\nu=0.025$ and $\gamma=10$.}
\label{DG_velocity_Kovasznay_Flow}
\end{figure*}
\subsection{Influence of $\gamma$ and $\gamma_{gd}$} \label{gamma_influence}
In this subsection, we go back to the original form of the penalty term \eqref{dh} with $\gamma_F=\gamma h_F^{-1}$ which is
\begin{align*}
d_h(\bm{u}_h,\bm{v}_h)=\gamma_{gd}(\nabla_h\cdot\bm{u}_h,\nabla_h\cdot\bm{v}_h)+\gamma\sum_{F\in\mathcal{F}_h}h_F^{-1}\langle\llbracket \vbold_h \rrbracket\cdot\nbold_F,\llbracket \ubold_h \rrbracket\cdot\nbold_F\rangle_F
\end{align*}
and study the influence of $\gamma_{gd}$ and $\gamma$ {on the velocity approximation of the DG schemes}. The model problems is the potential flow (cf.~\cite{akbas2018analogue})
\begin{align*}
    &\ubold(t,\xbold)=\bigg(5x_1^{4}-30x_1^{2}x_2^{2}+5x_2^{4},-20x_1^{3}x_2+20x_1x_2^{3}\bigg),\\
    &p(t,\xbold)=-\frac{1}{2}|\ubold(t,\xbold)|^{2}
\end{align*}
on the domain $\Omega:=[-1,1]^2$ consisting of ten colliding jets which meets at the stagnation point $(0, 0)$ (see Figure~\ref{stream_line_potential}) and the Kovasznay flow (in Subsection \ref{K_flow}) with $\nu=0.025$ and $\gamma$ (respectively $\gamma_{gd}$) ranging from $0,1,5,25,125$ when $k=2,3$ and $h=0.0884$.

It can be observed from Tables \ref{table_lamda} and \ref{table_lamda_gd}  that larger $\gamma$ (while $\gamma_{gd}=0$) decreases the velocity error tremendously in potential flow, but increases the errors in Kovasznay flow while keeping the same order of error magnitude. This observation indicates that the addition of penalty term $\gamma\sum_{F\in\mathcal{F}_h}h_F^{-1}\langle\llbracket \vbold_h \rrbracket\cdot\nbold_F,\llbracket \ubold_h \rrbracket\cdot\nbold_F\rangle_F$ may fail to decrease errors, but it may not affect the order of error too much. To the best of our knowledge, we have not seen a similar report for the DG schemes in the literature. Finally, we do not observe an obvious increasing or decreasing error when increasing $\gamma_{gd}$ (while keeping $\gamma=0$) except for the case from $\gamma_{gd}=0$ to $\gamma_{gd}=1$.
\begin{figure}[ht]
\centering
\begin{subfigure}{0.4\textwidth}
  \centering
  \includegraphics[width=1.0\linewidth,trim={7cm 2cm 0cm 2cm},clip]{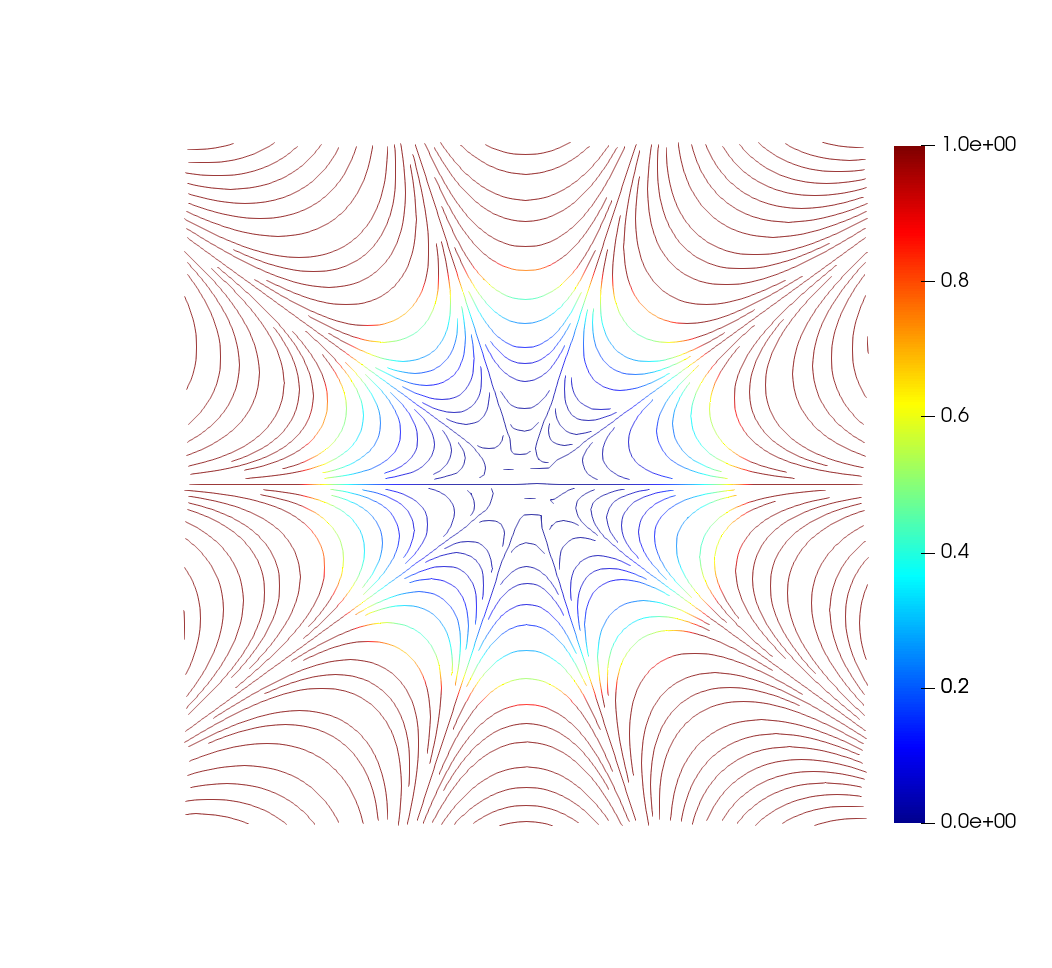}
\end{subfigure}
\caption{Stream lines for potential flow.}
\label{stream_line_potential}
\end{figure}
\begin{table}
\begin{center}
\makeatletter\def\@captype{table}\makeatother
\resizebox{\textwidth}{40mm}{
\begin{tabular}{|p{0.1cm}|c|c|l|l|l|l|}
\hline
\multicolumn{1}{|c|}{\multirow{3}{*}{$k$}} & \multirow{3}{*}{$\gamma$} & \multicolumn{2}{c|}{Kovasznay flow $\|\ubold-\ubold_{h}\|_{L^{2}(\Omega)}$}                                                                                                     & \multicolumn{2}{c|}{Potential flow $\|\ubold-\ubold_{h}\|_{L^{2}(\Omega)}$}                                                                                                   \\ \cline{3-6} 
\multicolumn{1}{|c|}{}                   &                    &                         \begin{tabular}[c]{@{}l@{}} {DG-N}\end{tabular} & \begin{tabular}[c]{@{}l@{}} \quad~{\qquad DG-C} \end{tabular} & \begin{tabular}[c]{@{}l@{}}~\:{DG-N}\end{tabular} & \begin{tabular}[c]{@{}l@{}}{\qquad DG-C}\end{tabular} \\ \hline
\multirow{5}{*}{2} & 0  & 7.75e-6                      & \qquad~ 7.78e-6    & ~\:  3.26e-4  & \quad ~ 2.60e-4           \\ \cline{2-6} 
& 1  & 8.25e-6     & \qquad ~ 8.27e-6                    &  ~\: 2.09e-4   & \quad ~ 2.09e-4                                                      \\ \cline{2-6} 
& 5  & 8.72e-6                                   
& \qquad ~ 8.73e-6   & ~\: 9.76e-5                                   & \quad ~ 9.13e-5                                                \\ \cline{2-6} 
& 25  & 8.96e-6    & \qquad ~ 8.96e-6                    & ~\: 2.86e-5    & \quad ~ 2.57e-5   \\

\cline{2-6} 
& 125  & 9.02e-6    & \qquad ~ 9.02e-6                    & ~\: 7.48e-6    & \quad ~ 6.82e-6  
\\ \hline
\multirow{5}{*}{3} & 0  & 1.41e-7                      & \qquad ~ 1.42e-7   & ~\: 6.43e-6  & \quad ~ 5.05e-6           \\ \cline{2-6} 
& 1  & 1.47e-7   & \qquad ~ 1.47e-7                    & ~\: 4.39e-6   & \quad ~ 4.24e-6        \\ \cline{2-6} 
& 5 & 1.55e-7   & \qquad ~ 1.55e-7        
& ~\: 2.05e-6    & \quad ~ 1.92e-6                                                    \\ \cline{2-6} 
& 25  & 1.60e-7                                       
& \qquad ~ 1.59e-7      & ~\: 5.86e-7                               & \quad ~ 5.32e-7  
\\
\cline{2-6}
& 125  & 1.61e-7                                    
 & \qquad ~ 1.61e-7     & ~\: 1.30e-7            & \quad ~ 1.16e-7  
\\ \hline
\end{tabular}}
\caption{Comparison between DG-N and DG-C schemes with $\nu=0.025$ and $h_{\text{max}} =0.0884$ for Kovasznay flow and potential flow when $\gamma_{gd}=0$.}
\label{table_lamda}
\end{center}
\end{table}
\begin{table}
\begin{center}
\makeatletter\def\@captype{table}\makeatother
\resizebox{\textwidth}{40mm}{
\begin{tabular}{|p{0.1cm}|c|c|l|l|l|l|}
\hline
\multicolumn{1}{|c|}{\multirow{3}{*}{$k$}} & \multirow{3}{*}{$\gamma_{gd}$} & \multicolumn{2}{c|}{Kovasznay flow $\|\ubold-\ubold_{h}\|_{L^{2}(\Omega)}$}                                                                                                     & \multicolumn{2}{c|}{Potential flow $\|\ubold-\ubold_{h}\|_{L^{2}(\Omega)}$}                                                                                                   \\ \cline{3-6} 
\multicolumn{1}{|c|}{}                   &                    &                         \begin{tabular}[c]{@{}l@{}} {DG-N}\end{tabular} & \begin{tabular}[c]{@{}l@{}} \quad~{\qquad DG-C} \end{tabular} & \begin{tabular}[c]{@{}l@{}}~\:{DG-N}\end{tabular} & \begin{tabular}[c]{@{}l@{}}{\qquad DG-C}\end{tabular} \\ \hline
\multirow{5}{*}{2} & 0  & 7.75e-6                      & \qquad~ 7.78e-6    & ~\:  3.26e-4  & \quad ~ 2.60e-4           \\ \cline{2-6} 
& 1  & 7.80e-6     & \qquad ~ 7.82e-6                    &  ~\: 2.47e-4   & \quad ~ 2.63e-4                                                      \\ \cline{2-6} 
& 5  & 7.81e-6                                   
& \qquad ~ 7.82e-6   & ~\: 2.47e-4                                   & \quad ~ 2.61e-4                                                \\ \cline{2-6} 
& 25  & 7.81e-6    & \qquad ~ 7.82e-6                    & ~\: 2.48e-4    & \quad ~ 2.61e-4   \\

\cline{2-6} 
& 125  & 7.81e-6    & \qquad ~ 7.82e-6                    & ~\: 2.48e-4    & \quad ~ 2.61e-4  
\\ \hline
\multirow{5}{*}{3} & 0  & 1.41e-7                      & \qquad ~ 1.42e-7   & ~\: 6.43e-6  & \quad ~ 5.05e-6           \\ \cline{2-6} 
& 1  & 1.44e-7   & \qquad ~ 1.44e-7                    & ~\: 4.20e-6   & \quad ~ 4.17e-6        \\ \cline{2-6} 
& 5 & 1.44e-7   & \qquad ~ 1.45e-7        
& ~\: 4.05e-6    & \quad ~ 4.04e-6                                                    \\ \cline{2-6} 
& 25  & 1.44e-7                                       
& \qquad ~ 1.45e-7      & ~\: 4.02e-6                               & \quad ~ 4.01e-6  
\\
\cline{2-6}
& 125  & 1.44e-7                                    
 & \qquad ~ 1.45e-7     & ~\: 4.02e-6            & \quad ~ 4.01e-6  
\\ \hline
\end{tabular}}
\caption{Comparison between DG-N and DG-C schemes with $\nu=0.025$ and $h_{\text{max}} =0.0884$ for Kovasznay flow and potential flow when $\gamma=0$.}
\label{table_lamda_gd}
\end{center}
\end{table}
\subsection{Lid Driven Flow}
In this section, we consider lid driven flow {for the DG-N scheme} with $Re=100,400$ (which corresponds to $\nu=0.01,0.0025$ respectively in this setting) \cite{ghia1982high} and the square domain $\Omega:=[0,1]^2$. There is a tangential velocity $\bm{u}=(1,0)$ on the top, while no-slip boundary conditions are applied on the other sides. In this test, we use polynomial $k=3$ for the pressure, $\gamma=10$ and $h=0.0283$ (which corresponds to a mesh size of $50\times50$).

Figures~\ref{fig:L_contour} and \ref{fig:L_trace} show the contour and streamlines of the lid driven cavity flow. The important aspect of Figure~\ref{fig:L_trace} is that a small corner vortex at the bottom right corner, which normally requires a very high mesh resolution, has been predicted by our scheme. In Figures \ref{fig:L_Horizontal_Velocity} and \ref{fig:L_Vertical_Velocity} we compare our simulation results with the data reported in \cite{ghia1982high}, we find our simulation results match the data perfectly except at the coordinate point $(0.9063,0.5)$ for vertical velocity; we are not sure if there is a typo in the original data from \cite{ghia1982high} as even in classical fluid finite element book like \cite{ZIENKIEWICZ2014127} (Figure 4.5(b), page 134) this point is ignored when making comparisons.
\begin{figure}[ht]
\centering
\begin{subfigure}{0.4\textwidth}
  \centering
  \includegraphics[width=1.0\linewidth,trim={7cm 2cm 0cm 2cm},clip]{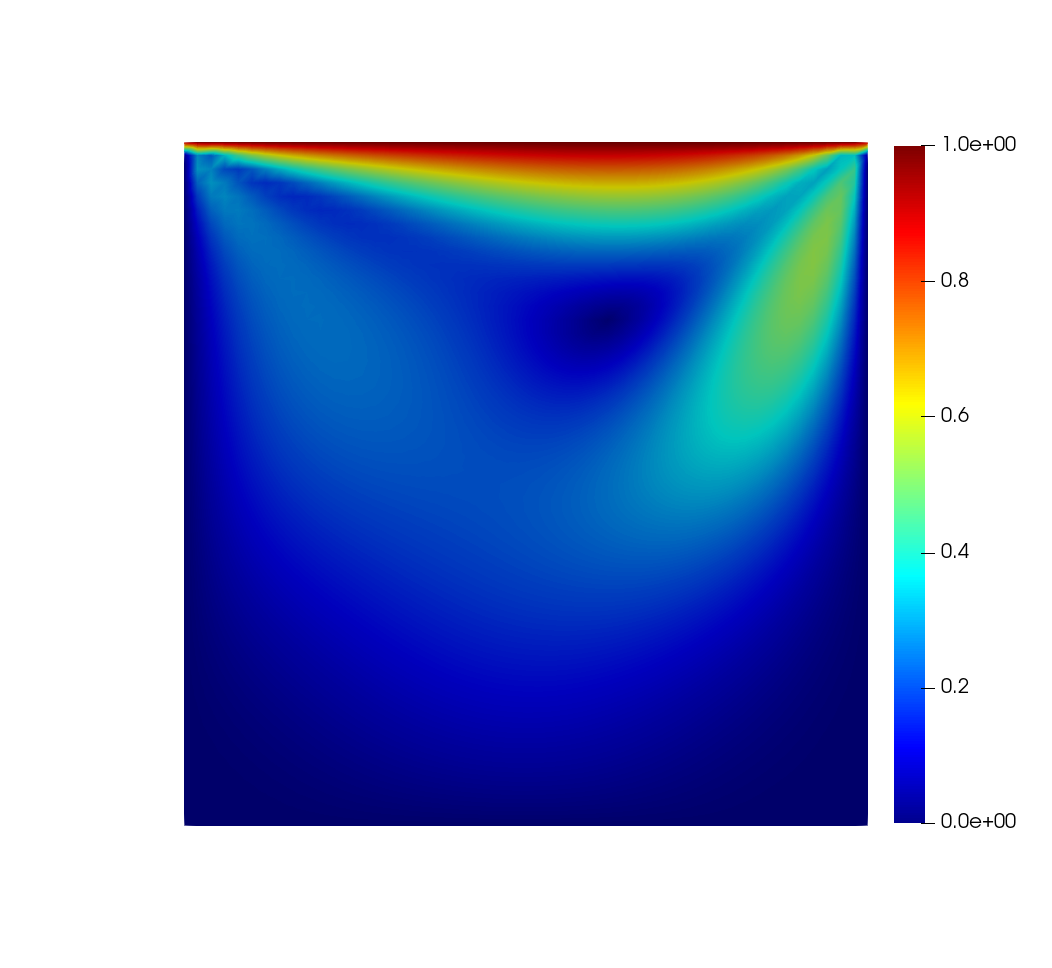}
\end{subfigure}%
\begin{subfigure}{.4\textwidth}
  \centering
  \includegraphics[width=1.0\linewidth,trim={7cm 2cm 0cm 2cm},clip]{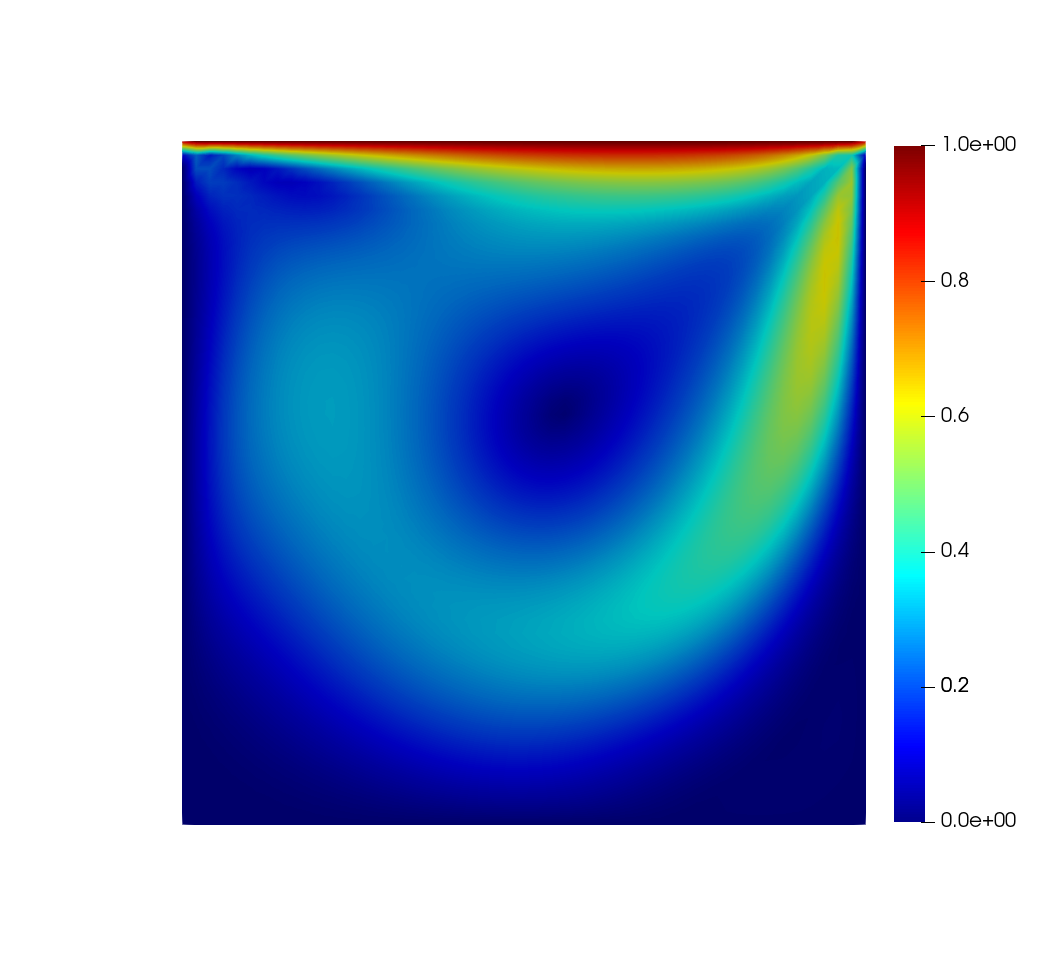}
\end{subfigure}
\caption{Lid Driven Cavity: Contours of velocity magnitude from DG-N with $Re=100$ (left) and $Re=400$ (right) when $p=3$, $h_{\text{max}}=0.0283$ and $\gamma=10$.}
\label{fig:L_contour}
\end{figure}
\begin{figure}[ht]
\centering
\begin{subfigure}{0.4\textwidth}
  \centering
  \includegraphics[width=1.0\linewidth,trim={7cm 2cm 0cm 2cm},clip]{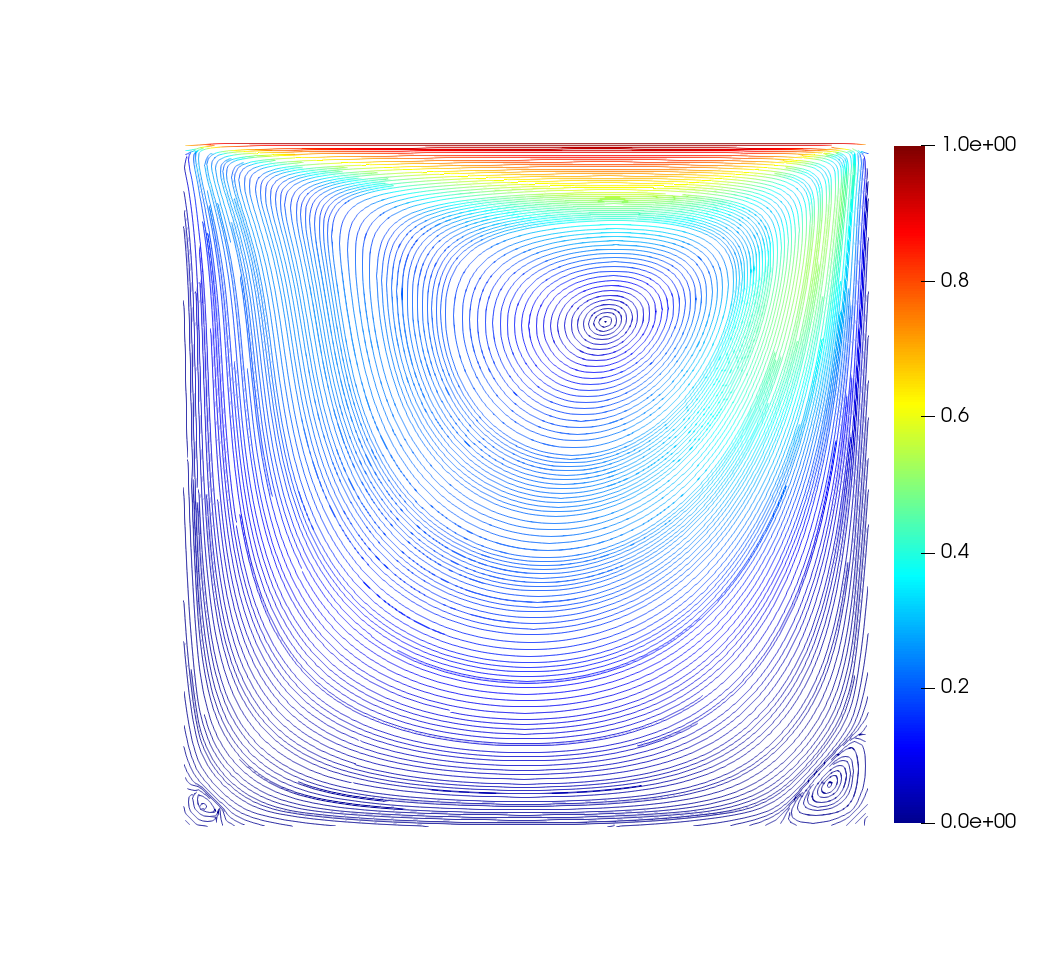}
\end{subfigure}%
\begin{subfigure}{.4\textwidth}
  \centering
  \includegraphics[width=1.0\linewidth,trim={7cm 2cm 0cm 2cm},clip]{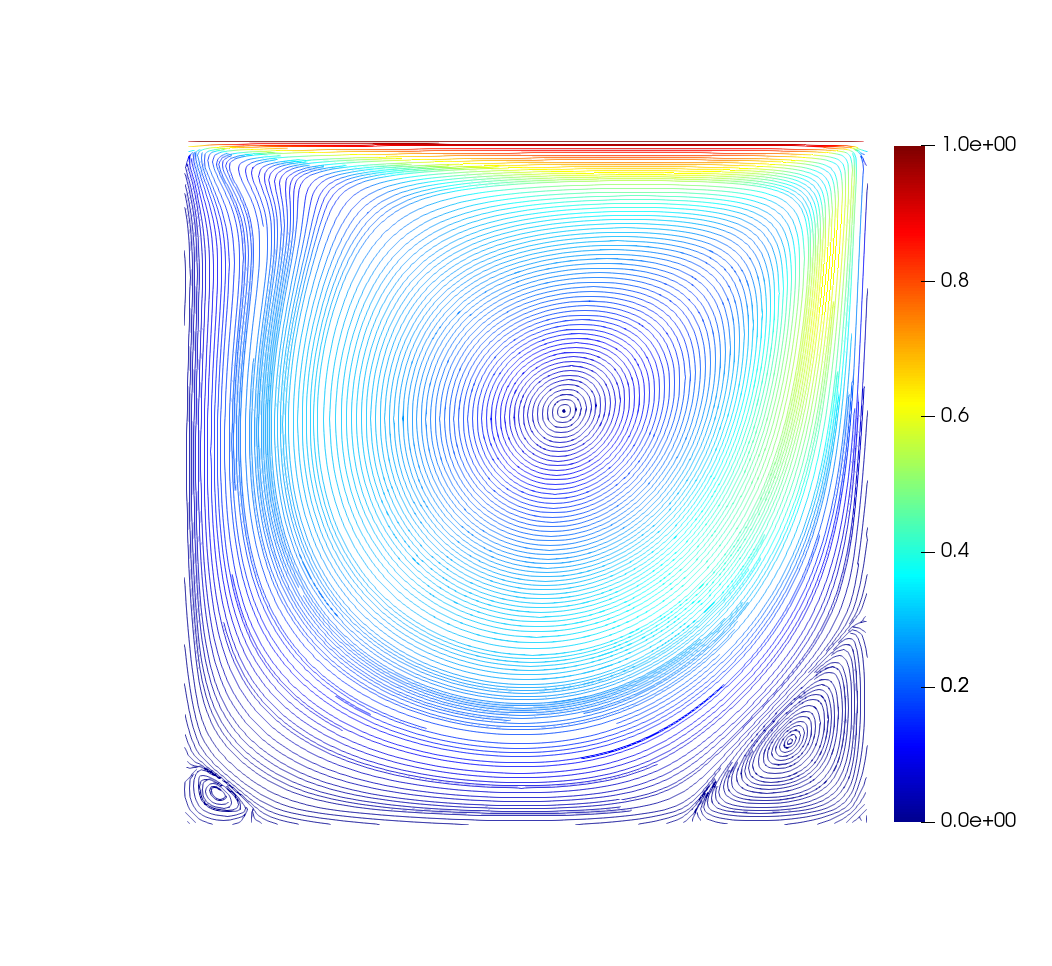}
\end{subfigure}
\caption{Lid Driven Cavity: Stream trace from DG-N with $Re=100$ (left) and $Re=400$ (right) when $k=3$, $h_{\text{max}}=0.0283$ and $\gamma=10$.}
\label{fig:L_trace}
\end{figure}
\begin{figure}[ht]
\centering
\begin{subfigure}{0.4\textwidth}
  \centering
  \includegraphics[width=1.0\linewidth,trim={0cm 0cm 0cm 0cm},clip]{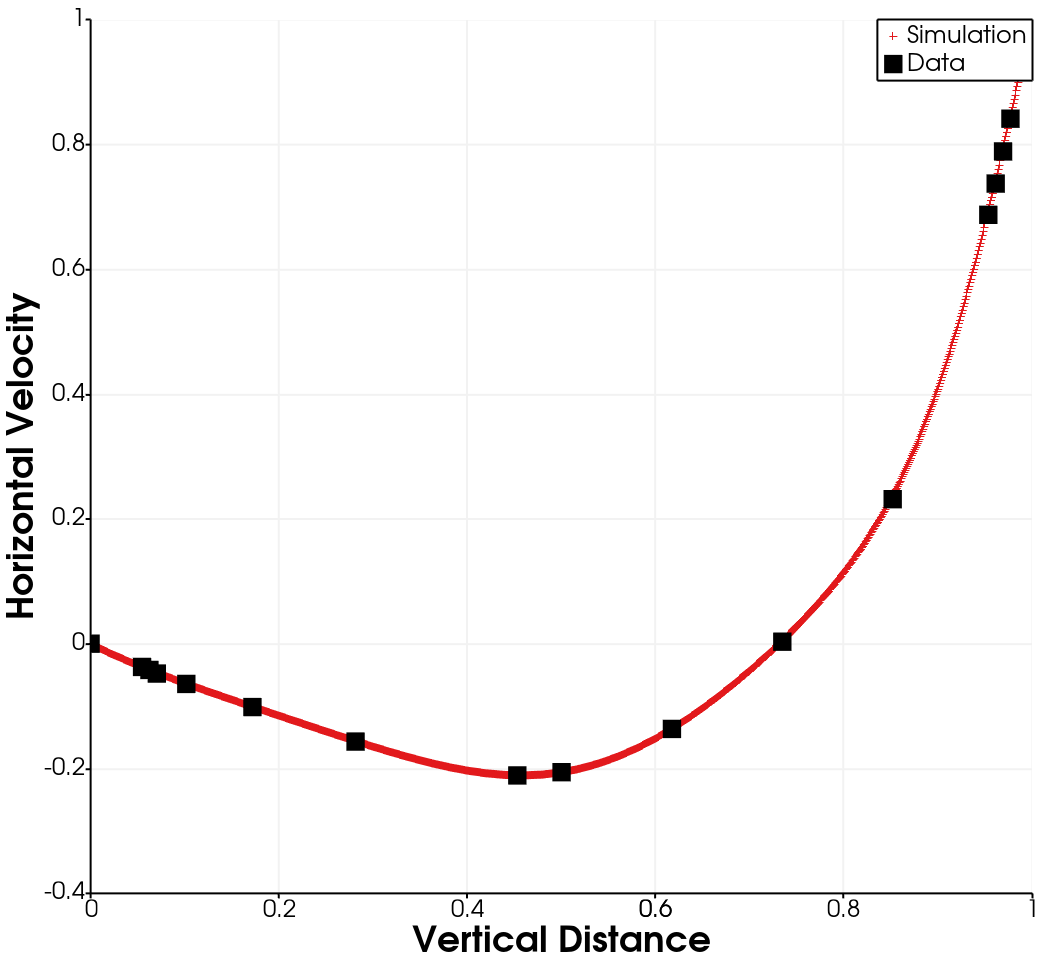}
\end{subfigure}%
\begin{subfigure}{0.4\textwidth}
  \centering
  \includegraphics[width=1.0\linewidth,trim={0cm 0cm 0cm 0cm},clip]{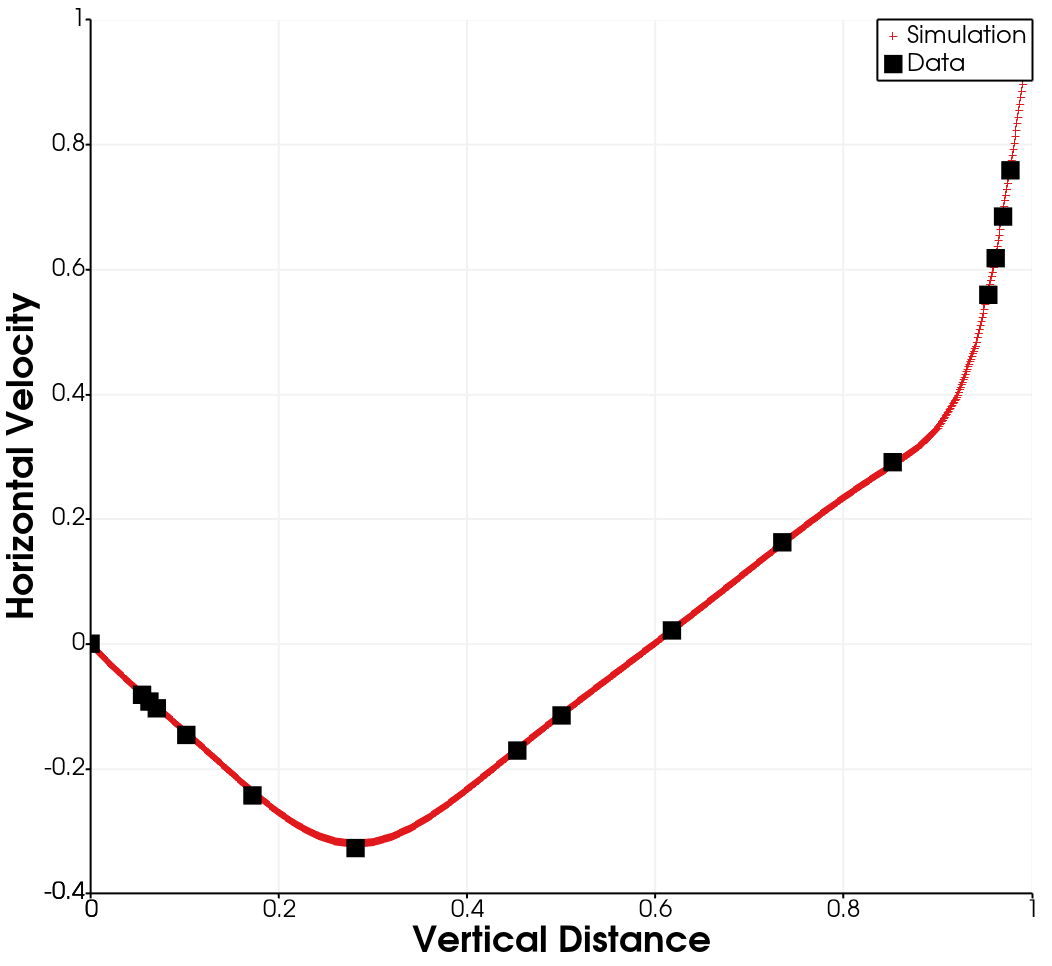}
\end{subfigure}
\caption{Lid Driven Cavity: Horizontal velocity from DG-N with $Re=100$ (left) and $Re=400$ (right) when $k=3$, $h_{\text{max}}=0.0283$, $\gamma=10$ {at $x_1=0.5$ (with $x_2$ varying)}}
\label{fig:L_Horizontal_Velocity}
\end{figure}
\begin{figure}[ht]
\centering
\begin{subfigure}{0.4\textwidth}
  \centering
  \includegraphics[width=1.0\linewidth,trim={0cm 0cm 0cm 0cm},clip]{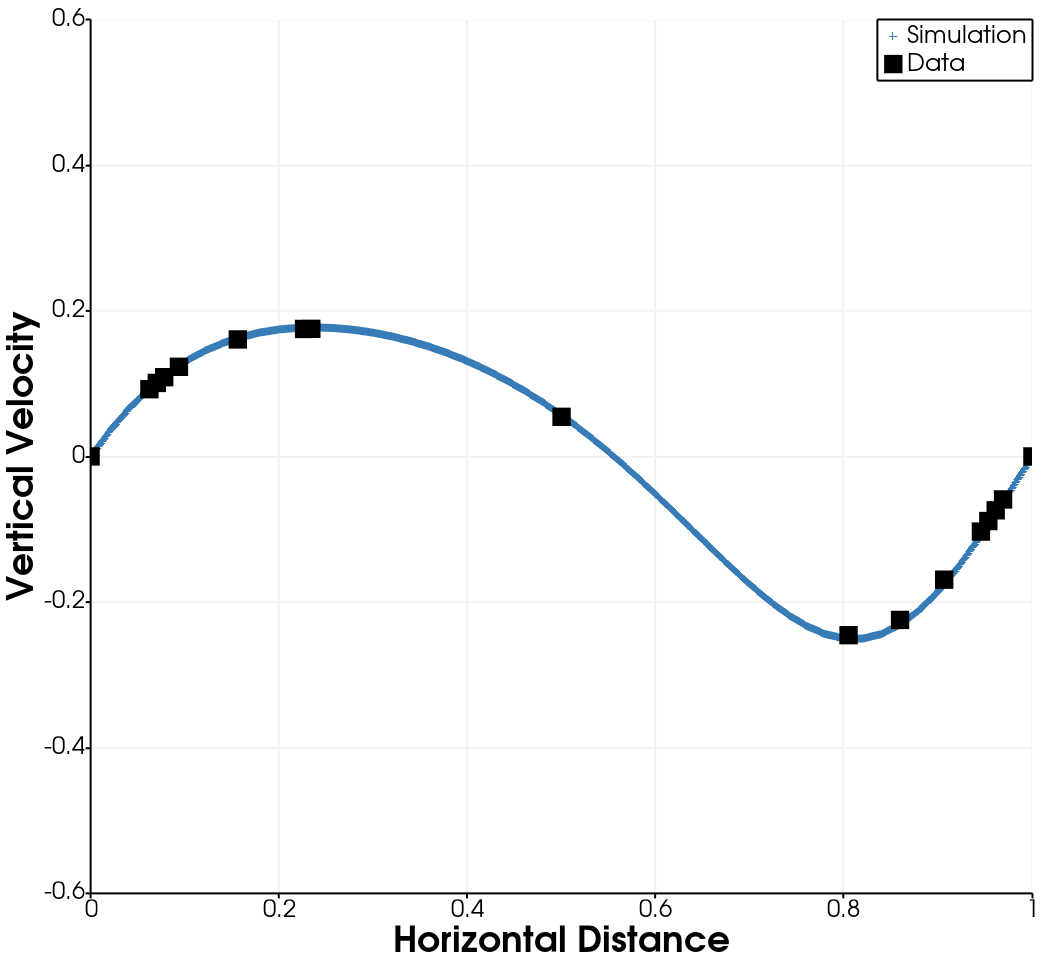}
\end{subfigure}%
\begin{subfigure}{0.4\textwidth}
  \centering
  \includegraphics[width=1.0\linewidth,trim={0cm 0cm 0cm 0cm},clip]{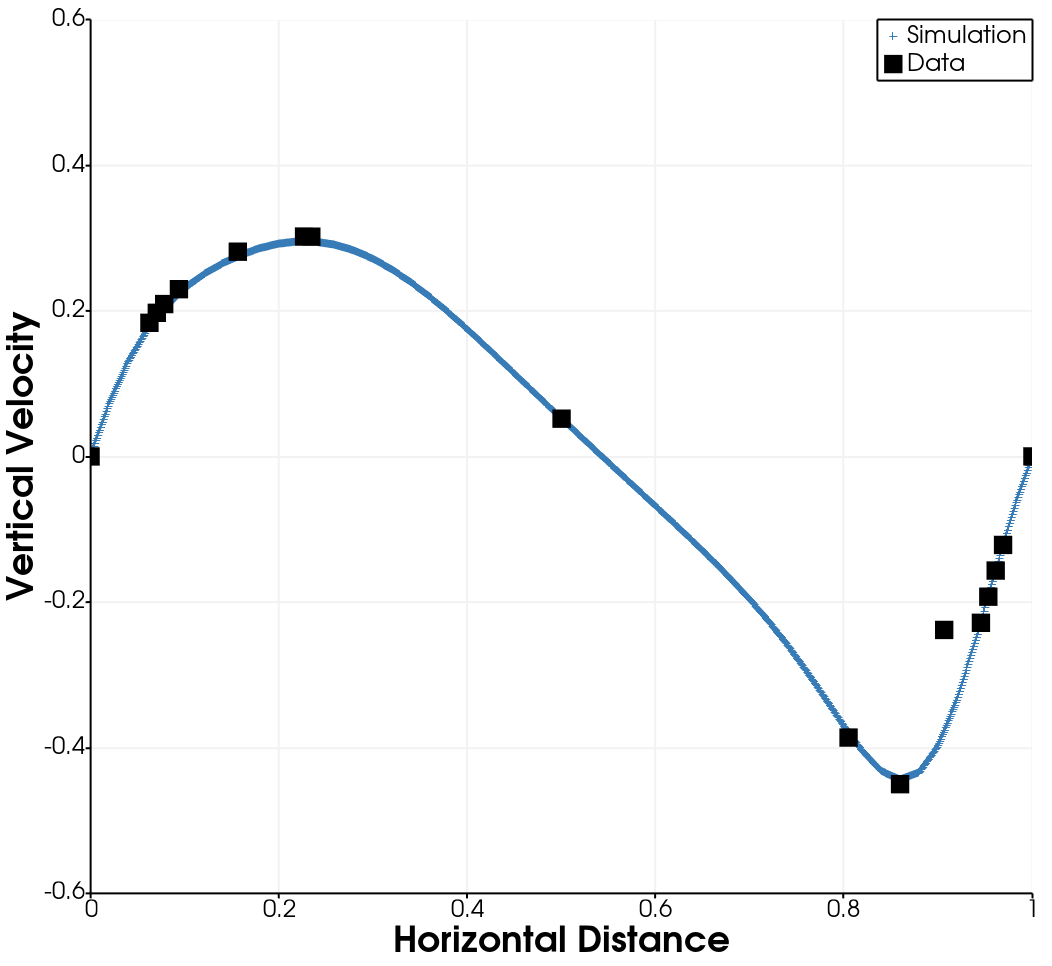}
\end{subfigure}
\caption{Lid Driven Cavity: Vertical velocity from DG-N with $Re=100$ (left) and $Re=400$ (right) when $k=3$, $h_{\text{max}}=0.0283$, $\gamma=10$ {at $x_2=0.5$ (with $x_1$ varying)}}
\label{fig:L_Vertical_Velocity}
\end{figure}
\begin{figure}
      \includegraphics[width=1.0\textwidth,trim={0cm 12cm 0cm 10cm},clip]{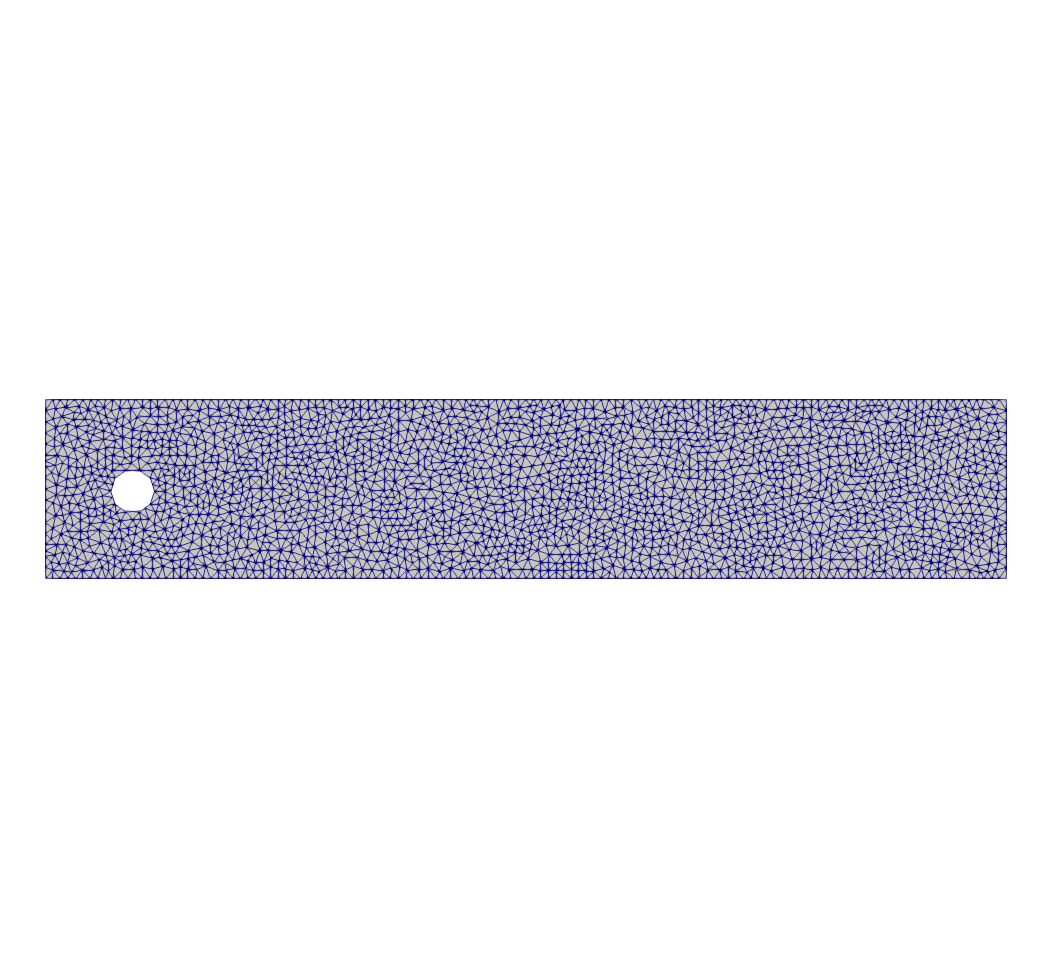}
      \vspace{4pt}
  \caption{Mesh for flow around  a cylinder}
  \label{cylinder_mesh}
\end{figure}
\clearpage
\begin{figure}
      \includegraphics[width=1.0\textwidth,trim={0cm 12cm 0cm 10cm},clip]{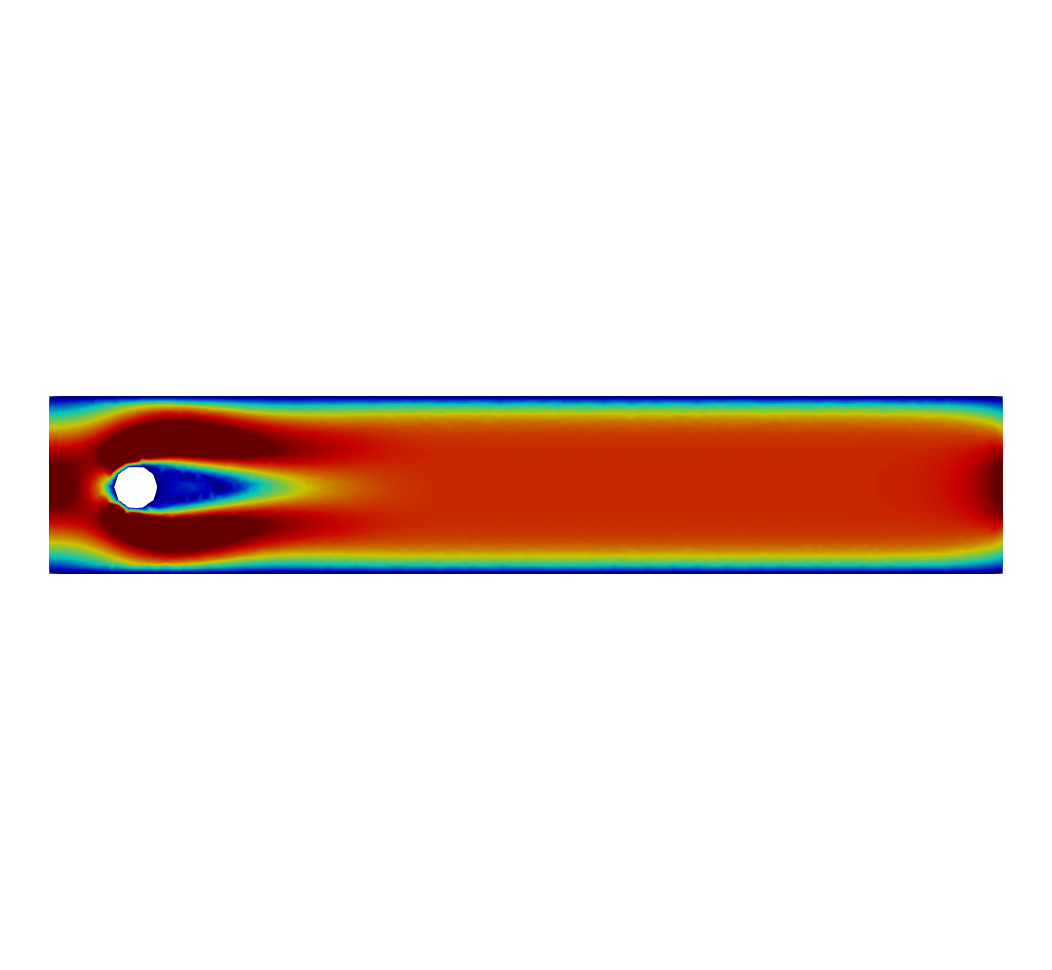}
      \includegraphics[width=1.0\textwidth,trim={0cm 12cm 0cm 10cm},clip]{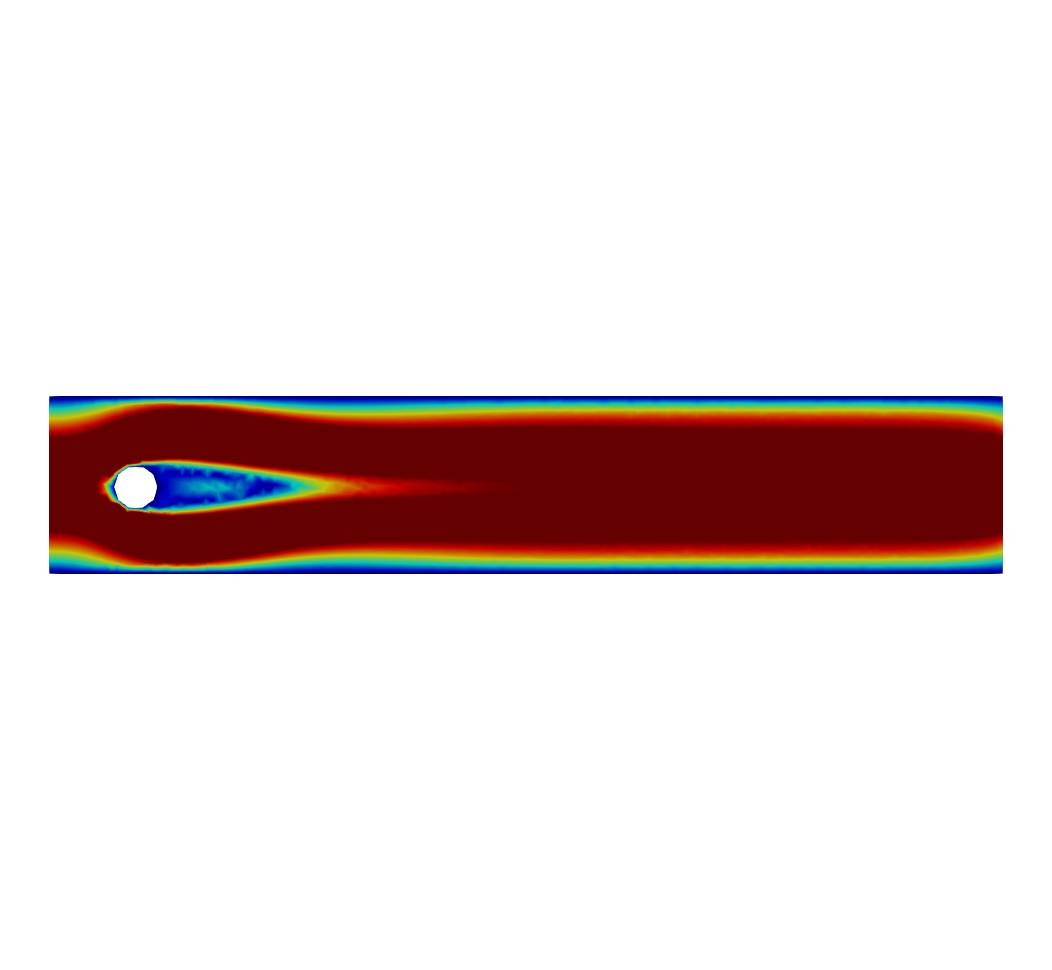}
      \includegraphics[width=1.0\textwidth,trim={0cm 12cm 0cm 10cm},clip]{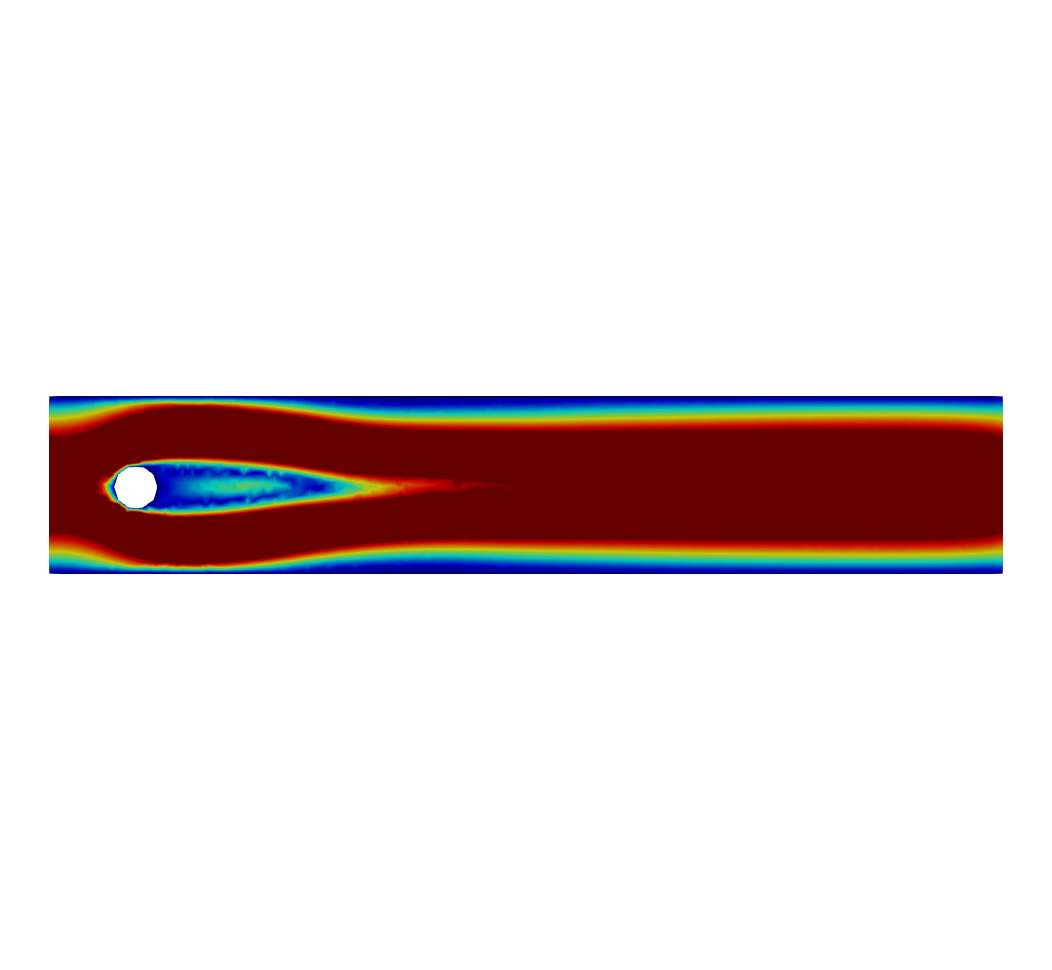}
      \includegraphics[width=1.0\textwidth,trim={0cm 12cm 0cm 10cm},clip]{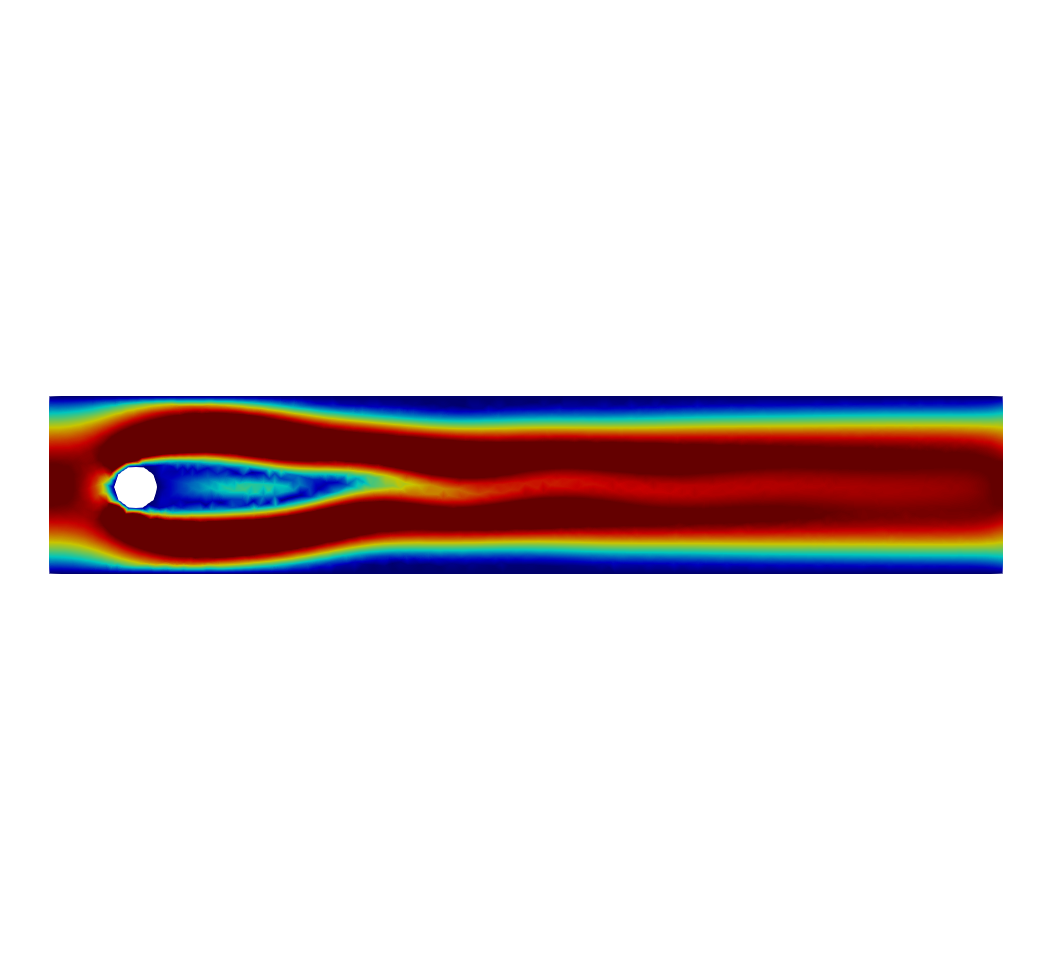}
  \caption{Contour of velocity magnitude when $t=2s, 3s, 5s, 6s$ from top to bottom---BDF2 time discretization}
  \label{DG_BDF2_cylinder_velocity_contour}
\end{figure}
\begin{figure}
      \includegraphics[width=1.0\textwidth,trim={0cm 12cm 0cm 10cm},clip]{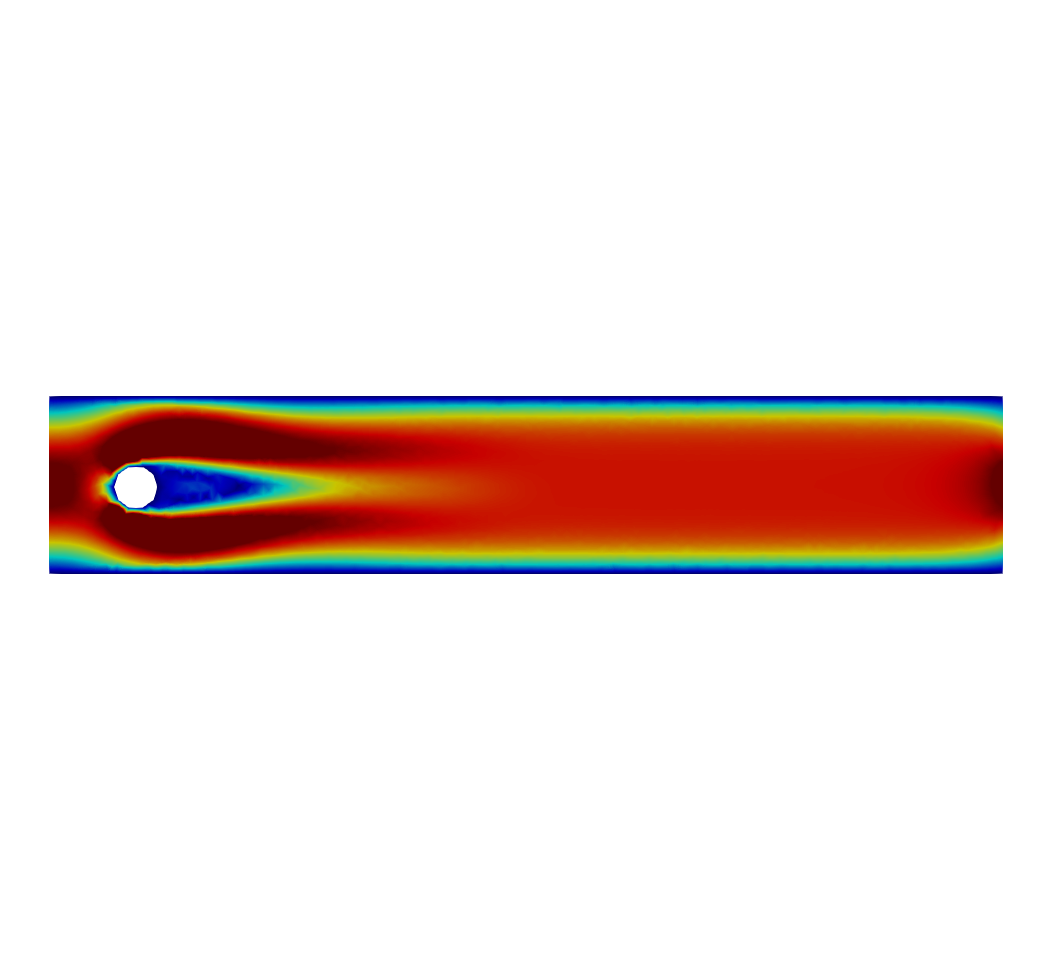}
      \includegraphics[width=1.0\textwidth,trim={0cm 12cm 0cm 10cm},clip]{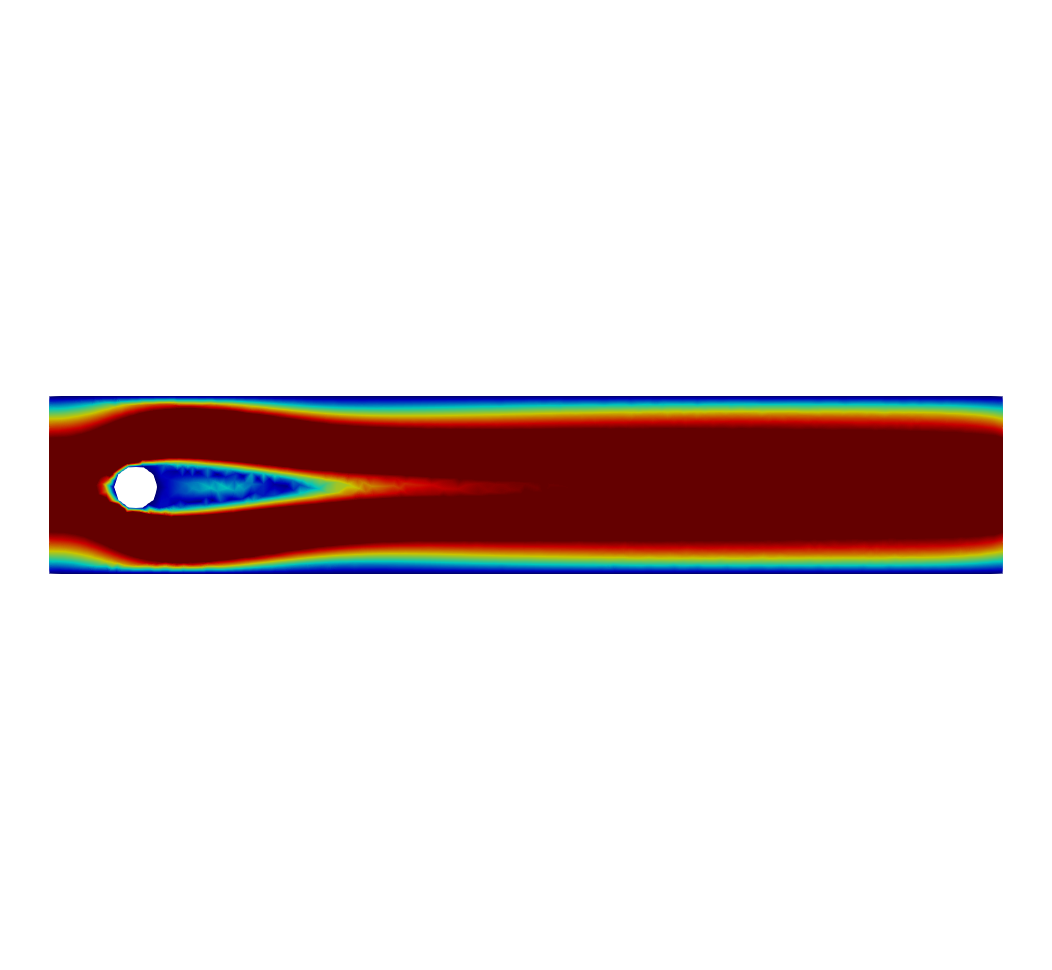}
      \includegraphics[width=1.0\textwidth,trim={0cm 12cm 0cm 10cm},clip]{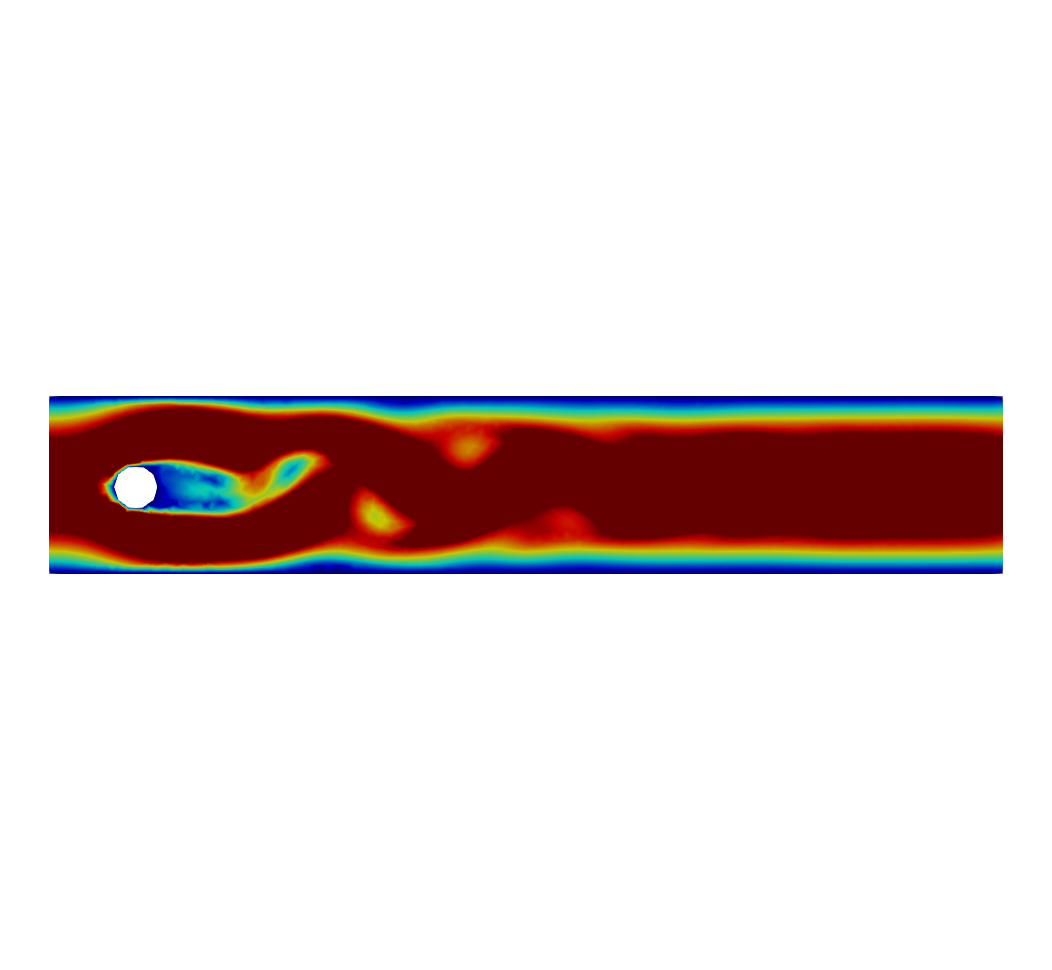}
      \includegraphics[width=1.0\textwidth,trim={0cm 12cm 0cm 10cm},clip]{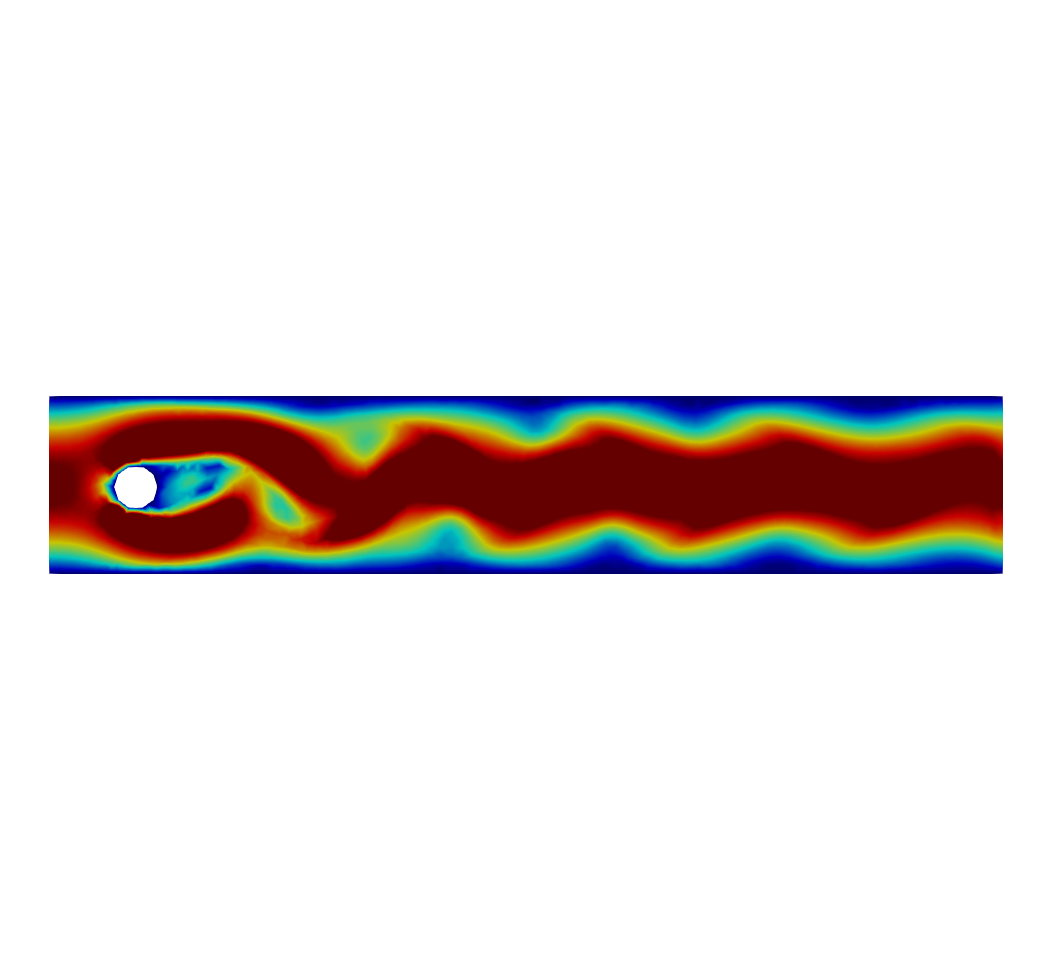}
      \vspace{4pt}
  \caption{Contour of velocity magnitude when $t=2s, 3s, 5s, 6s$ from top to bottom---Crank--Nicolson time discretization}
  \label{cylinder_velocity_contour}
\end{figure}
\clearpage
\subsection{Flow around a cylinder}
In the last example, we follow \cite{layton2009accuracy,schafer1996benchmark,LangtangenLogg2017} and consider the flow over a cylinder with our DG-N scheme when $\gamma=10$ using both BDF2 and Crank--Nicolson time discretization. We follow the example in \cite{LangtangenLogg2017} and use the FEniCS \textsf{mshr} tool to generate a fixed mesh, see Figure~\ref{cylinder_mesh}. Note that the BDF2 method is not necessarily stability preserving. The simulation domain $\Omega:=\left([0,2.2]\times[0,0.41]\right)\backslash B$, where $B$ is the disk centered at $(0.2,0.2)$ with radius $0.05$. The simulation time interval is $[0,T]$ with $T=8s$ and the time step $\tau=0.01.$ A primary feature of this benchmark is the formation of von K\'arm\'an vortex street. Our goal is to study the influence of time discretization on the formulation of the vortex. The inflow and outflow profile is given (cf.~\cite{layton2009accuracy}) as
\begin{align}
    &\bm{u}(t,0,x_2)=\bm{u}(t,2.2,x_2)=\frac{6}{0.41^{2}}\sin(\pi t/8)x_2(0.41-x_2),\\
    &\bm{v}(t,0,x_2)=\bm{v}(t,2.2,x_2)=\bm{0},
\end{align}
and the boundary condition on the rest of $\partial\Omega$ is set to be $\bm{u}=\bm{0}.$
The Reynolds number $Re$ corresponding to the mean velocity inflow ranges from $0$ to $100$. From Figures~\ref{DG_BDF2_cylinder_velocity_contour} and \ref{cylinder_velocity_contour}, we observe that the vortex forms gradually over time when Crank--Nicolson is used for the time discretization, while there is no apparent vortex formulation when using BDF2, which indicates the importance of using stability preserving time discretization. 
\section{Conclusion}\label{conclusion}
We have developed a general framework including the $H^{1}$, $H(\text{div})$ and DG methods with the use of different stress tensors. We proved the stability in general and discussed the expressions for penalty terms for each of the three cases. For Taylor--Green vortex and Kovasznay flow, our DG schemes are comparable to classical schemes in the literature, while the $H^{1}$ scheme from the framework is less accurate but with much less runtime with the Taylor-Green vortex. The $H(\text{div})$ scheme has the longest runtime among four schemes for Taylor-Green vortex implemented in FEniCS. We also show through examples that penalizing normal component of the velocity in DG schemes may fail to decrease absolute errors. In general, we are not able to demonstrate it rigorously and the choice of $\gamma$ is  empirical. In addition, we show that our DG scheme agrees very well with the features and data of lid driven flow. Finally, the importance of stability preserving time discretization has been shown by comparing the BDF2 and Crank--Nicolson scheme for the flow around a cylinder.



\begin{thebibliography}{10}

\bibitem{akbas2018analogue}
Mine Akbas, Alexander Linke, Leo~G. Rebholz, and Philipp~W. Schroeder,
  \emph{The analogue of grad--div stabilization in {DG} methods for
  incompressible flows: Limiting behavior and extension to tensor-product
  meshes}, Computer Methods in Applied Mechanics and Engineering \textbf{341}
  (2018), 917--938.

\bibitem{Arnold1982}
Douglas~N. Arnold, \emph{An interior penalty finite element method with
  discontinuous elements}, SIAM J. Numer. Anal. \textbf{19} (1982), no.~4,
  742--760. \MR{664882}

\bibitem{boffi2013mixed}
Daniele Boffi, Franco Brezzi, and Michel Fortin, \emph{Mixed finite element
  methods and applications}, vol.~44, Springer, 2013.

\bibitem{Brenner2004}
Susanne~C. Brenner, \emph{Korn's inequalities for piecewise {$H^1$} vector
  fields}, Math. Comp. \textbf{73} (2004), no.~247, 1067--1087. \MR{2047078}

\bibitem{BS2008}
Susanne~C. Brenner and L.~Ridgway Scott, \emph{The mathematical theory of
  finite element methods}, third ed., Texts in Applied Mathematics, vol.~15,
  Springer, New York, 2008. \MR{2373954}

\bibitem{BDM1985}
Franco Brezzi, Jim Douglas, Jr., and L.~D. Marini, \emph{Two families of mixed
  finite elements for second order elliptic problems}, Numer. Math. \textbf{47}
  (1985), no.~2, 217--235. \MR{799685}

\bibitem{Butcher1964}
J.~C. Butcher, \emph{Implicit {R}unge-{K}utta processes}, Math. Comp.
  \textbf{18} (1964), 50--64. \MR{159424}

\bibitem{cai2019high}
Xiaofeng Cai, Wei Guo, and Jing-Mei Qiu, \emph{A high order {semi-Lagrangian}
  discontinuous {Galerkin} method for the two-dimensional incompressible
  {Euler} equations and the guiding center {Vlasov} model without operator
  splitting}, Journal of Scientific Computing \textbf{79} (2019), no.~2,
  1111--1134.

\bibitem{case2011connection}
Michael~A. Case, Vincent~J. Ervin, Alexander Linke, and Leo~G. Rebholz, \emph{A
  connection between {Scott--Vogelius} and grad-div stabilized {Taylor--Hood
  FE} approximations of the {Navier--Stokes} equations}, SIAM Journal on
  Numerical Analysis \textbf{49} (2011), no.~4, 1461--1481.

\bibitem{cesmelioglu2017analysis}
Aycil Cesmelioglu, Bernardo Cockburn, and Weifeng Qiu, \emph{Analysis of a
  hybridizable discontinuous {Galerkin} method for the steady-state
  incompressible {Navier-Stokes} equations}, Mathematics of Computation
  \textbf{86} (2017), no.~306, 1643--1670.

\bibitem{charnyi2017conservation}
Sergey Charnyi, Timo Heister, Maxim~A. Olshanskii, and Leo~G. Rebholz, \emph{On
  conservation laws of {Navier--Stokes Galerkin} discretizations}, Journal of
  Computational Physics \textbf{337} (2017), 289--308.

\bibitem{chen2020versatile}
Xi~Chen and David~M Williams, \emph{Versatile mixed methods for the
  incompressible {Navier-Stokes} equations}, arXiv preprint arXiv:2007.08015
  (2020).

\bibitem{cockburn2005locally}
Bernardo Cockburn, Guido Kanschat, and Dominik Sch{\"o}tzau, \emph{A locally
  conservative {LDG} method for the incompressible {Navier-Stokes} equations},
  Mathematics of Computation \textbf{74} (2005), no.~251, 1067--1095.

\bibitem{cockburn2007note}
\bysame, \emph{A note on discontinuous {Galerkin} divergence-free solutions of
  the {Navier--Stokes} equations}, Journal of Scientific Computing \textbf{31}
  (2007), no.~1-2, 61--73.

\bibitem{Ern2012}
{Daniele A.~Di Pietro} and Alexandre Ern, \emph{Mathematical aspects of
  discontinuous {G}alerkin methods}, Math\'{e}matiques \& Applications (Berlin)
  [Mathematics \& Applications], vol.~69, Springer, Heidelberg, 2012.
  \MR{2882148}

\bibitem{falk2013stokes}
Richard~S. Falk and Michael Neilan, \emph{Stokes complexes and the construction
  of stable finite elements with pointwise mass conservation}, SIAM Journal on
  Numerical Analysis \textbf{51} (2013), no.~2, 1308--1326.

\bibitem{franca1988two}
Leopoldo~P. Franca and Thomas J.~R. Hughes, \emph{Two classes of mixed finite
  element methods}, Computer Methods in Applied Mechanics and Engineering
  \textbf{69} (1988), no.~1, 89--129.

\bibitem{ghia1982high}
U.~Ghia, K.~N. Ghia, and C.~T. Shin, \emph{High-{R}e solutions for
  incompressible flow using the {Navier-Stokes} equations and a multigrid
  method}, Journal of Computational Physics \textbf{48} (1982), no.~3,
  387--411.

\bibitem{GR1986}
Vivette Girault and Pierre-Arnaud Raviart, \emph{Finite element methods for
  {N}avier-{S}tokes equations}, Springer Series in Computational Mathematics,
  vol.~5, Springer-Verlag, Berlin, 1986, Theory and algorithms. \MR{851383}

\bibitem{GST2001}
Sigal Gottlieb, Chi-Wang Shu, and Eitan Tadmor, \emph{Strong
  stability-preserving high-order time discretization methods}, SIAM Rev.
  \textbf{43} (2001), no.~1, 89--112. \MR{1854647}

\bibitem{gresho1990theory}
Philip~M. Gresho and Stevens~T. Chan, \emph{On the theory of semi-implicit
  projection methods for viscous incompressible flow and its implementation via
  a finite element method that also introduces a nearly consistent mass matrix.
  part 2: Implementation}, International Journal for Numerical Methods in
  Fluids \textbf{11} (1990), no.~5, 621--659.

\bibitem{guzman2014conforming}
Johnny Guzm{\'a}n and Michael Neilan, \emph{Conforming and divergence-free
  {S}tokes elements on general triangular meshes}, Mathematics of Computation
  \textbf{83} (2014), no.~285, 15--36.

\bibitem{guzman2016h}
Johnny Guzm{\'a}n, Chi-Wang Shu, and Fil{\'a}nder~A. Sequeira, \emph{H(div)
  conforming and {DG} methods for incompressible {Euler's} equations}, IMA
  Journal of Numerical Analysis \textbf{37} (2016), no.~4, 1733--1771.

\bibitem{HNW1993}
E.~Hairer, S.~P. N{\o}rsett, and G.~Wanner, \emph{Solving ordinary differential
  equations. {I}}, second ed., Springer Series in Computational Mathematics,
  vol.~8, Springer-Verlag, Berlin, 1993, Nonstiff problems. \MR{1227985}

\bibitem{HLW2006}
Ernst Hairer, Christian Lubich, and Gerhard Wanner, \emph{Geometric numerical
  integration}, second ed., Springer Series in Computational Mathematics,
  vol.~31, Springer-Verlag, Berlin, 2006, Structure-preserving algorithms for
  ordinary differential equations. \MR{2221614}

\bibitem{Hesthaven07}
Jan~S. Hesthaven and Tim Warburton, \emph{Nodal discontinuous {Galerkin}
  methods: {Algorithms}, analysis, and applications}, Springer Science \&
  Business Media, 2007.

\bibitem{jenkins2014parameter}
Eleanor~W. Jenkins, Volker John, Alexander Linke, and Leo~G. Rebholz, \emph{On
  the parameter choice in grad-div stabilization for the {S}tokes equations},
  Advances in Computational Mathematics \textbf{40} (2014), no.~2, 491--516.

\bibitem{john2017divergence}
Volker John, Alexander Linke, Christian Merdon, Michael Neilan, and Leo~G.
  Rebholz, \emph{On the divergence constraint in mixed finite element methods
  for incompressible flows}, SIAM Review \textbf{59} (2017), no.~3, 492--544.

\bibitem{joshi2016post}
Sumedh~M. Joshi, Peter~J. Diamessis, Derek~T. Steinmoeller, Marek Stastna, and
  Greg~N. Thomsen, \emph{A post-processing technique for stabilizing the
  discontinuous pressure projection operator in marginally-resolved
  incompressible inviscid flow}, Computers \& Fluids \textbf{139} (2016),
  120--129.

\bibitem{LangtangenLogg2017}
Hans~Petter Langtangen and Anders Logg, \emph{Solving {PDEs} in {Python}},
  Springer, 2017.

\bibitem{layton2008introduction}
William Layton, \emph{Introduction to the {Numerical Analysis of Incompressible
  Viscous Flows}}, SIAM, 2008.

\bibitem{layton2009accuracy}
William Layton, Carolina~C. Manica, Monika Neda, Maxim Olshanskii, and Leo~G.
  Rebholz, \emph{On the accuracy of the rotation form in simulations of the
  {Navier--Stokes} equations}, Journal of Computational Physics \textbf{228}
  (2009), no.~9, 3433--3447.

\bibitem{lehrenfeld2016high}
Christoph Lehrenfeld and Joachim Sch{\"o}berl, \emph{High order exactly
  divergence-free hybrid discontinuous {G}alerkin methods for unsteady
  incompressible flows}, Computer Methods in Applied Mechanics and Engineering
  \textbf{307} (2016), 339--361.

\bibitem{linke2011convergence}
Alexander Linke, Leo~G. Rebholz, and Nicholas~E. Wilson, \emph{On the
  convergence rate of grad-div stabilized {Taylor--Hood} to {Scott--Vogelius}
  solutions for incompressible flow problems}, Journal of Mathematical Analysis
  and Applications \textbf{381} (2011), no.~2, 612--626.

\bibitem{liu2000high}
Jian-Guo Liu and Chi-Wang Shu, \emph{A high-order discontinuous {Galerkin}
  method for 2d incompressible flows}, Journal of Computational Physics
  \textbf{160} (2000), no.~2, 577--596.

\bibitem{lube2002stable}
Gert Lube and Maxim~A. Olshanskii, \emph{Stable finite-element calculation of
  incompressible flows using the rotation form of convection}, IMA Journal of
  Numerical Analysis \textbf{22} (2002), no.~3, 437--461.

\bibitem{olshanskii2009grad}
Maxim Olshanskii, Gert Lube, Timo Heister, and Johannes L{\"o}we,
  \emph{Grad--div stabilization and subgrid pressure models for the
  incompressible {Navier--Stokes} equations}, Computer Methods in Applied
  Mechanics and Engineering \textbf{198} (2009), no.~49-52, 3975--3988.

\bibitem{olshanskii2004grad}
Maxim Olshanskii and Arnold Reusken, \emph{Grad-div stablilization for {S}tokes
  equations}, Mathematics of Computation \textbf{73} (2004), no.~248,
  1699--1718.

\bibitem{olshanskii2002low}
Maxim~A. Olshanskii, \emph{A low order {G}alerkin finite element method for the
  {Navier--Stokes} equations of steady incompressible flow: a stabilization
  issue and iterative methods}, Computer Methods in Applied Mechanics and
  Engineering \textbf{191} (2002), no.~47-48, 5515--5536.

\bibitem{palha2017mass}
Artur Palha and Marc Gerritsma, \emph{A mass, energy, enstrophy and vorticity
  conserving {(MEEVC)} mimetic spectral element discretization for the {2D}
  incompressible {Navier--Stokes} equations}, Journal of Computational Physics
  \textbf{328} (2017), 200--220.

\bibitem{RT1977}
P.-A. Raviart and J.~M. Thomas, \emph{A mixed finite element method for 2nd
  order elliptic problems}, Mathematical aspects of finite element methods
  ({P}roc. {C}onf., {C}onsiglio {N}az. delle {R}icerche ({C}.{N}.{R}.), {R}ome,
  1975), 1977, pp.~292--315. Lecture Notes in Math., Vol. 606. \MR{0483555}

\bibitem{schafer1996benchmark}
Michael Sch{\"a}fer, Stefan Turek, Franz Durst, Egon Krause, and Rolf
  Rannacher, \emph{Benchmark computations of laminar flow around a cylinder},
  Flow simulation with high-performance computers II, Springer, 1996,
  pp.~547--566.

\bibitem{schroeder2018towards}
Philipp~W. Schroeder, Christoph Lehrenfeld, Alexander Linke, and Gert Lube,
  \emph{Towards computable flows and robust estimates for inf-sup stable {FEM}
  applied to the time-dependent incompressible {Navier}--{Stokes} equations},
  SeMA Journal, Boletin de la Sociedad Española de Matemática Aplicada
  \textbf{75} (2018), no.~4, 629--653.

\bibitem{schroeder2017stabilised}
Philipp~W. Schroeder and Gert Lube, \emph{Stabilised {dG-FEM} for
  incompressible natural convection flows with boundary and moving interior
  layers on non-adapted meshes}, Journal of Computational Physics \textbf{335}
  (2017), 760--779.

\bibitem{schroeder2018divergence}
\bysame, \emph{Divergence-free {H(div)-FEM} for time-dependent incompressible
  flows with applications to high {Reynolds} number vortex dynamics}, Journal
  of Scientific Computing (2018), 1--29.

\bibitem{Temam1984}
Roger Temam, \emph{Navier-{S}tokes equations}, third ed., Studies in
  Mathematics and its Applications, vol.~2, North-Holland Publishing Co.,
  Amsterdam, 1984, Theory and numerical analysis, With an appendix by F.
  Thomasset. \MR{769654}

\bibitem{zhang2005new}
Shangyou Zhang, \emph{A new family of stable mixed finite elements for the {3D}
  {Stokes} equations}, Mathematics of Computation \textbf{74} (2005), no.~250,
  543--554.

\bibitem{ZIENKIEWICZ2014127}
O.C. Zienkiewicz, R.L. Taylor, and P.~Nithiarasu, \emph{Chapter 4 -
  incompressible {Newtonian} laminar flows}, The Finite Element Method for
  Fluid Dynamics (Seventh Edition) (O.C. Zienkiewicz, R.L. Taylor, and
  P.~Nithiarasu, eds.), Butterworth-Heinemann, Oxford, seventh edition ed.,
  2014, pp.~127 -- 161.

\end{thebibliography}
\providecommand{\bysame}{\leavevmode\hbox to3em{\hrulefill}\thinspace}
\providecommand{\MR}{\relax\ifhmode\unskip\space\fi MR }
\providecommand{\MRhref}[2]{%
  \href{http://www.ams.org/mathscinet-getitem?mr=#1}{#2}
}
\providecommand{\href}[2]{#2}

\end{document}